\documentclass[11pt,reqno]{amsart}
\usepackage{CJK,CJKpunct}
\usepackage{subeqnarray}
\usepackage{cases}
\usepackage{stmaryrd}
\usepackage{color}
\usepackage{amsfonts}
\usepackage{mathrsfs}
\usepackage{amssymb,amsmath,amsthm}
\usepackage{array}
\usepackage{booktabs}
\usepackage{epsfig}
\usepackage{graphicx}
\usepackage{appendix}
\usepackage{lineno}
\usepackage{enumerate}
\usepackage{comment}
\usepackage{caption}
\usepackage[numbers,sort&compress]{natbib}
\newcommand{\dps}{\displaystyle}
\newtheorem{theorem}{\indent Theorem}[section]

\newtheorem{remark}{\indent Remark}[section]

\newtheorem{example}{\indent Example}[section]
\newcommand{\ba}{\begin{array}}\newcommand{\ea}{\end{array}}
\newcommand{\be}{\begin{eqnarray}}\newcommand{\ee}{\end{eqnarray}}
\newcommand{\beq}{\begin{equation*}}\newcommand{\eeq}{\end{equation*}}
\newcommand{\bex}{\begin{eqnarray*}}
\newcommand{\eex}{\end{eqnarray*}}
\newcommand{\tabincell}[2]{\begin{tabular}{@{}#1@{}}#2\end{tabular}}
\newcommand{\PreserveBackslash}[1]{\let\temp=\\#1\let\\=\temp}
\newcolumntype{C}[1]{>{\PreserveBackslash\centering}p{#1}}
\newcolumntype{R}[1]{>{\PreserveBackslash\raggedleft}p{#1}}
\newcolumntype{L}[1]{>{\PreserveBackslash\raggedright}p{#1}}

\def\bq{\begin{equation}}
\def\eq{\end{equation}}
\def\beq{\begin{equation*}}
\def\eeq{\end{equation*}}
\def\br{\begin{eqnarray}}
\def\er{\end{eqnarray}}
\def\brr{\bq\begin{array}{r@{}l}}
\def\err{\end{array}\eq}
\def\bry{\beq\begin{array}{r@{}l}}
\def\ery{\end{array}\eeq}

\font\tenbi=cmmib10   at 11 pt
\font\sevenbi=cmmib10 at 9pt
\font\fivebi=cmmib7 at 6pt
\newfam\bifam
\textfont\bifam=\tenbi \scriptfont\bifam=\sevenbi  \scriptscriptfont\bifam=\fivebi

\def\bi{\fam\bifam\tenbi}
\font\sixtdb=msbm10 at 16 pt \font\tendb=msbm10 at 12 pt  \font\sevendb=msbm7
\newfam\dbfam
\textfont\dbfam=\sixtdb

\textfont\dbfam=\tendb \scriptfont\dbfam=\sevendb

\def\Dt {\tau}

\def\x{{\bi x}}

\title[Efficient Schemes for Fractional Allen-Cahn Eq]
{Highly efficient and energy dissipative schemes for the time fractional Allen-Cahn equation$^*$}
\author[Dianming Hou and Chuanju Xu]
{Dianming Hou$^{1}$
\quad
Chuanju Xu$^{2,3}$}
\thanks{\hskip -12pt
${}^*${\color{black} The work of D. Hou is supported by NSFC grant 12001248 and the Natural Science Foundation of the Jiangsu Higher Education Institutions of China grant
BK20201020.
The second author has received support from NSFC grant 11971408, NNW2018-ZT4A06 project,
and NSFC/ANR joint program 51661135011/ANR-16-CE40-0026-01.}
\\
$^{1}$School of Mathematics and Statistics, Jiangsu Normal
University, 221116 Xuzhou, China.\\
${}^{2}$School of Mathematical Sciences and
Fujian Provincial Key Laboratory of Mathematical Modeling and High Performance
Scientific Computing, Xiamen
University, 361005 Xiamen, China.\\
${}^{3}$Corresponding author. Email: cjxu@xmu.edu.cn (C. Xu)}

\begin{document}
\graphicspath{{figures/},}

\date {\today}
\maketitle

\begin{abstract}
In this paper, we propose and analyze a time-stepping method
for the time fractional Allen-Cahn equation. The key property of the
proposed method is its unconditional stability for general meshes,
including the graded mesh commonly used for this type of equations.
The unconditional stability is proved through establishing
a discrete nonlocal free energy dispassion law,
which is also true for the continuous problem. The main idea used in the analysis is
to split the time fractional derivative into two parts: a local part and a history part, which are
discretized by the well known L1, L1-CN, and $L1^{+}$-CN schemes.
Then an extended auxiliary variable approach is used to deal with the nonlinear and
history term.
The main contributions of the paper are: First,
it is found that the time fractional Allen-Chan equation is a dissipative system related to a nonlocal free energy.
Second, we construct efficient time stepping schemes satisfying the same dissipation law at the discrete level.
In particular, we prove that the proposed
schemes are unconditionally stable for quite general meshes. Finally,
the efficiency of the proposed method is verified by a series of numerical experiments.
\end{abstract}

\section{Introduction}
\setcounter{equation}{0}

As a class of mathematical models,
gradient flows is partial differential equations under the form:
\be\label{prob0}
\dps\frac{\partial \phi}{\partial t}=-\mbox{grad}_{H} E(\phi),
\ee
where $\phi$ is the state function (also called phase function in many cases),
$E(\cdot)$ is the free energy driving functional associated to the physical problem, and
$\mbox{grad}_{H} E(\cdot)$ is the functional derivative of $E$ in the Sobolev space $H$.
It has other names: it is often called variational principle in mathematics and
Onsager principle in physics.
Obviously the gradient flows satisfies the energy dissipation law:
 \bq\label{EDlaw}
 \frac{d}{dt}E(\phi)
 =\Big(\mbox{grad}_{H} E(\phi), \frac{\partial \phi}{\partial t}\Big)
 =-\|\frac{\partial \phi}{\partial t}\|_{0}^{2},
 \eq
 where $(\cdot,\cdot)$ and $\|\cdot\|_{0}$ stand for the $L^2(\Omega)$-inner product and norm, respectivly.
This means that the state function $\phi$ evolves in such a way that the energy functional $E$ dissipates in time, i.e., in the opposite direction to the gradient of
$E$ at $\phi$.
This makes the models very useful in many fields of science and engineering, such as
interface dynamics \cite{And97,Chen98,Gur96,Yue04},  thin films \cite{GO01,OF98},
crystal growth \cite{EKHG02,EG04,EPB07}, polymers \cite{Fr93,Fr03} and liquid crystals \cite{Doi86,Lar90,Lar91,Les79}.

In this paper, we are interested in the following model:
\be\label{FGF}
_{0}{}\!D^{\alpha}_{t}\phi=-\mbox{grad}_{H} E(\phi),
\ee
deriving from gradient flows having a modified dissipation mechanism.
Here $0<\alpha<1$, $_{0}{}\!D^{\alpha}_{t}$ is the Caputo fractional derivative defined by
\bex
_{0}{}\!D^{\alpha}_{t}\phi(t)=\frac{1}{\Gamma(1-\alpha)}\int_{0}^{t}(t-s)^{-\alpha}\phi'(s)ds.
\eex
In mathematics gradient flows involving fractional derivatives have been extensively studied
in recent years; see, .e.g, \cite{CZCWZ18,ZLWW14,DJLQ18,AM17,AIM17,ASS16,CZW18,LCWZ18,SXK16,LWY17,JLGZ19}.
From the definition it is seen that the fractional derivative is some kind of weighted average
in the history of the traditional derivative. This means that the change rate, i.e., the derivative, at
the current time is affected by the historical rates. In a larger field
this property has been found quite useful in describing the
memory effect which can be present, for example,
in some materials such as viscoelastic materials or polymers.
Intuitively, the gradient
flows model \eqref{FGF} can be used to describe the systems in which dissipation
of the associated free energy has memory effect in some circumstances.
One of typical examples of such models is the fractional Allen-Cahn equation, which is also
the focus of this paper. There exist a number of studies for this equation.
Tang et al. \cite{TYZ18} proved
that the time-fractional phase field model admits an integral type energy dissipation law.
They investigated the L1 time stepping scheme on the uniform mesh,
which is of first order energy stable accuracy.
Du et al. \cite{DYZ19} developed several time schemes based on the convex
splitting and weighted stabilization, and
proved that the convergence rates of their schemes are of order-$\alpha$ in the uniform mesh without regularity assumption on the solution.
Recently, Liao et al. \cite{JLGZ20} proposed an adaptive second-order Crank-Nicolson time-stepping scheme using SAV approach for the
time-fractional MBE model, and showed that
the proposed scheme are unconditional stable on the nonuniform mesh.
{\color{black}Very recently, Quan et al. \cite{QTY20_1} theoretically proved the time-fractional energy law and the weighted dissipation law. Accordingly,
they constructed a first order numerical method on the uniform time mesh \cite{QTY20_2},
which preserved the energy laws. However, it seems not easy to construct higher order schemes
for nonuniform meshes satisfying the same energy laws.}

The aim of the present paper is to propose easy-to-implement and
unconditionally stable schemes, which
preserve a non-local energy dissipation law to be specified.
The main idea in constructing the schemes is to use
existing efficient approximations to discretize the local part and history part of the
fractional derivative respectively, and use auxiliary variable
approaches \cite{Shen17_1,Shen17_2,HAX19} to deal with the nonlinear potential
in the free energy. The contributions of the paper are threefold:
\begin{itemize}

\item Finding of a non-local energy dissipation law of the time fractional gradient flows.

\item Construction of several unconditionally stable schemes for the time fractional Allen-Cahn equation,
which satisfy a discrete version of the energy dissipation law.
It is proved that the stability and energy dissipation law
remain true on the graded mesh, which is useful in recovering the optimal convergence order
for typical solutions having low regularity at the initial time.

\item The proposed schemes are very easy to implement. That is,
only several Poisson-type equations with constant coefficients need to be solved at each time step.
Furthermore, a fast evaluation technique
based on the sum-of-exponentials approach is used to accelerate the calculation and reduce the storage.
\end{itemize}

The paper is organized as follows: In the next section, we derive the non-local energy dissipation law for the time fractional gradient flows.
In Section 3, we construct and analyze the first order numerical scheme.
A discrete energy dissipation law of the scheme is established for general time girds.
In Section 4, we propose and analyze two higher order schemes:
a $2-\alpha$ order and a second order schemes based on Crank-Nicolson formula.
The unconditional stability of the both schemes are rigorously proved.
The numerical experiments are carried out in Section 5, not only to validate stability and accuracy of the proposed methods, but also to
numerically investigate the coarsening dynamics.
Finally, the paper ends with some concluding remarks.

\section {Non-local energy dissipation law}\label{sect2_0}
\setcounter{equation}{0}

We consider the time fractional gradient flows \eqref{FGF} in the bounded domain $\Omega\in \mathbb{R}^{n}~(n=1,2,3)$. When $\alpha=1$, it follows from
integrating \eqref{EDlaw} from $t_e$ to $t_l$ for any $0\leq t_e < t_l$ that:
\bex
 E(\phi(t_l))-E(\phi(t_e))
 =-\int_{t_e}^{t_l}\Big\|\frac{\partial\phi(\cdot,s)}{\partial s}\Big\|_{0}^{2}ds\leq0.
 \eex
That is, the free energy $E(\cdot)$ is decreasing in $t$. However, for $0<\alpha<1$, this energy law does not
hold no longer. Instead, the solution of \eqref{FGF} satisfies an non-local energy law that we derive
below. To see that, we split the fractional derivative into two parts as follows:
\be\label{split}
_{0}{}\!D^{\alpha}_{t}\phi(t) =D^{\alpha,\hat{t}}_{l,t}\phi+D^{\alpha,\hat{t}}_{h,t}\phi,
\ee
where $0<\hat{t}<t$, and the local term $D^{\alpha,\hat{t}}_{l,t}\phi$ and the history term
$D^{\alpha,\hat{t}}_{h,t}\phi$ are respectively defined by
\bex
D^{\alpha,\hat{t}}_{l,t}\phi
=
\frac{1}{\Gamma(1-\alpha)}\int_{\hat{t}}^{t}(t-s)^{-\alpha}\phi'(s)ds,
\quad\quad
D^{\alpha,\hat{t}}_{h,t}\phi
=
\frac{1}{\Gamma(1-\alpha)}\int_{0}^{\hat{t}}(t-s)^{-\alpha}\phi'(s)ds.
\eex
Now we introduce the non-local  {\color{black}``energy"}
\bq\label{ener2}
\overline{E}(\hat{t},t; \phi):=E(\phi)+F_h(\hat{t},t; \phi),\ \ 0<\hat{t}<t,
\eq
where $F_h(\hat{t},t; \phi)$ is the non-local part of the energy, defined by
\bex
F_h(\hat{t},t; \phi):=\int_{\Omega}\int_{\hat{t}}^{t}D^{\alpha,\hat{t}}_{h,s}\phi\frac{\partial\phi}{\partial s}dsd\x,\ \ 0<\hat{t}<t.
\eex
A direct calculation shows
\bex
\dps\frac{d}{dt}\overline{E}(\hat{t},t; \phi)
=\frac{d}{dt}{E}(\phi) + \Big(D^{\alpha,\hat{t}}_{h,t}\phi,\frac{\partial\phi}{\partial t}\Big)
=\Big(\mbox{grad}_{H} E(\phi), \frac{\partial\phi}{\partial t}\Big)
+
\Big(D^{\alpha,\hat{t}}_{h,t}\phi,\frac{\partial\phi}{\partial t}\Big),\ \ 0<\hat{t}<t.
\eex
Then using \eqref{FGF} and \eqref{split} gives
\be\label{te1}
\dps\frac{d}{dt}\overline{E}(\hat{t},t; \phi)
=(- _{0}{}\!D^{\alpha}_{t}\phi(t), \frac{\partial\phi}{\partial t}\Big)
+
\Big(D^{\alpha,\hat{t}}_{h,t}\phi,\frac{\partial\phi}{\partial t}\Big)
=
-\Big(D^{\alpha,\hat{t}}_{l,t}\phi,\frac{\partial\phi}{\partial t}\Big),\ \ 0<\hat{t}<t.
\ee
This allows to establish the following {\color{black}``energy"} decay property:
for $0<\hat{t}<t$, integrating \eqref{te1} yields
\be\label{te2}
\overline{E}(\hat{t},t;\phi(t)) - \overline{E}(\hat{t},\hat{t};\phi(\hat{t}))
= - \int_\Omega\int^t_{\hat{t}} D^{\alpha,\hat{t}}_{l,s}\phi \frac{\partial\phi}{\partial s}dsd\x.
\ee
Let's define the bilinear form $\dps\mathcal{A}_{\alpha}^{\hat{t}, t}(\cdot,\cdot)$ with $0<\hat{t}<t$ by:
for the functions $\varphi$ and $\psi$,
\bex
\mathcal{A}_{\alpha}^{\hat{t}, t}(\varphi, \psi)
:=\frac{1}{\Gamma(1-\alpha)}\int_{\hat{t}}^{t}\int_{\hat{t}}^{s}(s-\sigma)^{-\alpha}\varphi(\sigma)\psi(s)d\sigma ds.
\eex
It has been known; see, e.g., \cite{TYZ18,HZX20}, that the bilinear form
$\dps\mathcal{A}_{\alpha}^{\hat{t}, t}(\cdot,\cdot)$ is positive for any $0<\hat{t}<t$. That is, for any
$\psi\in L^{2}(\hat{t},t)$ so that the following expression makes sense, it holds:
\be\label{coer}
\mathcal{A}_{\alpha}^{\hat{t}, t}(\psi, \psi) \ge 0.
\ee
It then follows from \eqref{te2} and \eqref{coer} that
\bex
\overline{E}(\hat{t},t;\phi(t)) - \overline{E}(\hat{t},\hat{t};\phi(\hat{t}))
= - \int_\Omega \mathcal{A}_{\alpha}^{\hat{t}, t}(\partial_t\phi, \partial_t\phi) d\x
\le 0.
\eex
This can be regarded as an energy law associated to the {\color{black}fucntional} $\overline{E}(\hat{t},t; \phi)$ defined
in \eqref{ener2}. However, the inconvenience in using $\overline{E}(\hat{t},t; \phi)$ is that it depends
on $\hat{t}$, which makes the fucntional discontinuous in $t$.
For the numerical purpose it is desirable to derive a continuous-in-time energy, which is dissipative
in a given time grid, so that we have a clear goal to construct our numerical scheme satisfying the
same dissipation law.
To this end, for a given time mesh, say $0=t_{0}<t_{1}<t_{2}<\cdots<t_{M}=T$, we
define a new non-local  {\color{black}``energy"} functional as follows:
\bq\label{energ_eps}
\overline{E}(\phi)=\!\!\!
\begin{array}{l}
\begin{cases}
E(\phi),\  t\in[t_{0},t_{1}],\\[9pt]
E(\phi)+F_h(t_1,t; \phi),\ t\in[t_{1},t_{2}],\\[9pt]
\dps E(\phi)+F_h(t_n,t; \phi)+\sum_{k=1}^{n-1}F_h(t_{k},t_{k+1}; \phi),\
t\in[t_{n},t_{n+1}], n=2,\cdots, M-1.
\end{cases}
\end{array}
\eq
It is readily seen that
\bex
\dps\frac{d}{dt} \overline{E}(\phi)
=
- \Big(D^{\alpha,t_n}_{l,t}\phi, \frac{\partial\phi}{\partial t}\Big), \ \ t\in[t_{n},t_{n+1}], n=0,1, \cdots, M-1.
\eex
Integrating the above equality in the interval $[t_{n},t_{n+1}], n=0,1, \cdots, M-1$ gives
\bq\label{AClaw}
\overline{E}(\phi(t_{n+1}))-\overline{E}(\phi(t_{n}))=-\int_{\Omega}\mathcal{A}_\alpha^{t_n,t_{n+1}}(\phi_{t},\phi_{t})d\x\leq 0, \mbox{ for all } n=0,1,\cdots M-1.
\eq
We see that for any given time mesh $\{t_n\}_{n=0}^{M}$, the corresponding non-local free energy, defined in
\eqref{energ_eps} is dissipative at the grid points.
 {\color{black} We would like to point out that the functionals defined in \eqref{ener2} and \eqref{energ_eps} do not necessarily have any physical meaning.
The motivation for introducing such a modified ``energy" is purely mathematical.
That is, we want to find suitable functionals related to the equation, which decay in time.
This provides insight into how a stable scheme should look like. }
Our aim in the next section is to construct numerical schemes that satisfy the same dissipation law.

\section {Numerical methods --- a first order scheme}\label{sect2a}
\setcounter{equation}{0}

To simplify the presentation, we will only consider the time fractional Allen-Cahn equation, i.e.,
\be\label{FAC}
_{0}{}\!D^{\alpha}_{t}\phi=-\mbox{grad}_{H} E(\phi),
\ee
subject to
the periodic boundary condition or Neumann boundary condition,
where the free energy $E(\cdot)$ is defined
by
\bq\label{E}
E(\phi):=\dps\int_{\Omega}\Big[\frac{\varepsilon^{2}}{2}|\nabla\phi|^{2}+F(\phi)\Big]d\x,
\eq
and $F(\cdot)$ is a nonlinear potential. $H:=L^{2}(\Omega)$.

{\bf Auxiliary variable approach.} The proposed schemes are based on a reformulation of the time fractional Allen-Cahn equation
by introducing an auxiliary variable --- an approach intensively studied recently for gradient
flows; see, e.g., \cite{shen2018scalar,HAX19} and the references therein.
The key to is to rewrite the original equation \eqref{FAC}-\eqref{E}
into the following equivalent form:
\be\label{re_prob0}
_{0}{}\!D^{\alpha}_{t}\phi-\varepsilon^{2}\Delta\phi+\Big(1-\frac{R(t)}{R(t)}\Big)
\theta^{2}\Delta\phi+\frac{R(t)}{R(t)} F'(\phi)=0,
\ee
where
\be\label{R}
R(t)=\sqrt{\overline{E}_{\theta}(\phi)+C_{0}}, \ \ \ \theta^{2}\leq\varepsilon^{2},
\ee
and, for the time grid $\{t_n\}_0^M$, $\overline{E}_{\theta}(\phi)$ is defined by
\beq
\overline{E}_{\theta}(\phi)=\!\!\!
\begin{array}{l}
\begin{cases}
E_{\theta}(\phi),\  t\in[t_{0},t_{1}],\\[9pt]
E_{\theta}(\phi)+F_h(t_1,t; \phi),\ t\in[t_{1},t_{2}],\\[9pt]
\dps E_{\theta}(\phi)+F_h(t_n,t; \phi)+\sum_{k=1}^{n-1}F_h(t_{k},t_{k+1}; \phi),\
t\in[t_{n},t_{n+1}], n=2,\cdots, M-1,
\end{cases}
\end{array}
\eeq
with $E_{\theta}(\phi):=\int_{\Omega}[\frac{\theta^{2}}{2}|\nabla\phi|^{2}+F(\phi)]d\x$, and
$C_{0}$ being a constant such that $\overline{E}_{\theta}(\phi)+C_{0}>0$.

To find a suitable way to discretize the auxiliary variable $R(t)$,
we take the derivative of \eqref{R} with respect to $t$ to obtain the auxiliary equation:
\be\label{re_prob2}
 \dps\frac{d R}{d t}=\frac{1}{2R(t)}\Big(-\theta^{2}\Delta\phi+F'(\phi)+D^{\alpha,t_n}_{h,t}\phi,
\frac{\partial \phi}{\partial t}\Big),\ \  \forall t\in [t_n, t_{n+1}], n=2,\cdots, M-1.
\ee
Furthermore, we use the operator splitting \eqref{split} to rewrite the equation \eqref{re_prob0}
under the equivalent form: for $n=2,\cdots, M-1$,
\be\label{re_prob}
D^{\alpha, t_n}_{l,t}\phi-\varepsilon^{2}\Delta\phi+\Big(1-\frac{R(t)}{R(t)}\Big)
\theta^{2}\Delta\phi+\frac{R(t)}{R(t)} (F'(\phi) + D^{\alpha, t_n}_{h,t}\phi)
=0,\ \  \forall t\in [t_n, t_{n+1}].
\ee
Now we are led to discretize the equations \eqref{re_prob} and \eqref{re_prob2}.
The great advantage of this approach is that, although we have one more variable and one more equation to discretize compared to the original equation, constructing stable and efficient
schemes with help of the auxiliary variable turns out to be a much easier task.

Before describing our schemes, we first realize, by
taking $L^2(\Omega)$-inner products and integrating from $t_{n}$ to $t_{n+1}$ of \eqref{re_prob} and \eqref{re_prob2} with $\frac{\partial \phi}{\partial t}$
 and $2R(t)$ respectively, that
 \brr\label{law2}
\dps\Big[R^{2}(t_{n+1})+\frac{\varepsilon^{2}-\theta^{2}}{2}\|\nabla\phi(\cdot,t_{n+1})\|^{2}_{0}\Big]-\Big[R^{2}(t_{n})+\frac{\varepsilon^{2}-\theta^{2}}{2}\|\nabla\phi(\cdot,t_{n})\|^{2}_{0}\Big]&\dps=-\int_{\Omega}\mathcal{A}_\alpha^{t_n,t_{n+1}}(\phi_{t},\phi_{t})d\x\\
&\leq 0.
 \err
Noticing
 $$\dps R^{2}+\frac{\varepsilon^{2}-\theta^{2}}{2}\|\nabla\phi\|^{2}_{0}
 =\overline{E}(\phi)+C_{0},
 $$
with $\overline{E}(\cdot)$ being the non-local energy functional
defined in \eqref{energ_eps},
we have
\be\label{law2b}
\overline{E}(\phi(t_{n+1})) - \overline{E}(\phi(t_{n})) \leq 0.
\ee
This is exactly the same dissipation law as \eqref{AClaw}, derived from the original equation
without the auxiliary variable.
We will see that after discretization, the discrete solution satisfies a discrete dissipation law
under the form \eqref{law2} rather than \eqref{law2b}.

Starting with the equivalent equations \eqref{re_prob2} and \eqref{re_prob}, we are now
in a position to construct various efficient time stepping schemes to
calculate the solution $\phi$.
\medskip

 {\bf A first order scheme.}
 Let $\tau_n=t_n-t_{n-1}, n=1,\cdots,M$ be the time step size,
 and $\tau=\max\{\Dt_{n},n=1,\cdots,M\}$ be the maximum step size.

 We propose to use the popular L1 approximation \cite{LX07} to discretize
 the local and history parts of the Caputo fractional derivative at $t=t_{n+1}$:
  \bry
 D^{\alpha,t_{n}}_{l,t_{n+1}}\phi
  =&\dps\frac{1}{\Gamma(1-\alpha)}\int_{t_{n}}^{t_{n+1}}(t_{n+1}-s)^{-\alpha}\phi_{s}(s)ds
  \\[9pt]
  =&\dps\frac{1}{\Gamma(1-\alpha)}\frac{\phi(t_{n+1})-\phi(t_{n})}{\Dt_{n+1}}
  \int_{t_{n}}^{t_{n+1}}(t_{n+1}-s)^{-\alpha}ds + e^{n+1}_{l,\tau}
  \\[12pt]
  := &L^\alpha_l\phi(t_{n+1})+e^{n+1}_{l,\tau},
  \ery
  \bry
 D^{\alpha,t_{n}}_{h,t_{n+1}}\phi
  =&\dps\frac{1}{\Gamma(1-\alpha)}\sum_{k=0}^{n-1}\int_{t_k}^{t_{k+1}}(t_{n+1}-s)^{-\alpha}\phi_{s}(s)ds\\[9pt]
  =&\dps\sum_{k=0}^{n-1}\frac{1}{\Gamma(1-\alpha)}\frac{\phi(t_{k+1})-\phi(t_{k})}{\Dt_{k+1}}\int_{t_{k}}^{t_{k+1}}(t_{n+1}-s)^{-\alpha}ds + e^{n+1}_{h,\tau}\\[12pt]
 :=& L^\alpha_h\phi(t_{n+1})+e^{n+1}_{h,\tau},
  \ery
where the discrete fractional operators $L^{\alpha}_{l}$ and $L^{\alpha}_{h}$ are
defined respectively by
\bex
&& L^\alpha_l\phi(t_{n+1})
 = b_{0}\frac{\phi(t_{n+1})-\phi(t_{n})}{\Dt_{n+1}},\\
&& L^\alpha_h\phi(t_{n+1})
=\sum_{k=0}^{n-1}b_{n-k}\frac{\phi(t_{k+1})-\phi(t_{k})}{\Dt_{k+1}},
\eex
and the coefficients $b_{k}$ are given by
\bex
&&  b_{n-k}=\frac{1}{\Gamma(1-\alpha)}\int_{t_{k}}^{t_{k+1}}(t_{n+1}-s)^{-\alpha}ds>0, k=0,1,\cdots,n.
\eex
The truncation errors $e^{n+1}_{l,\tau}$ and $e^{n+1}_{h,\tau}$ are defined respectively by
\bex
e^{n+1}_{l,\tau}
= \frac{1}{\Gamma(1-\alpha)}\Big[\int_{t_{n}}^{t_{n+1}}(t_{n+1}-s)^{-\alpha}\phi_{s}(s)ds
- \dps\frac{\phi(t_{n+1})-\phi(t_{n})}{\Dt_{n+1}}
\int_{t_{n}}^{t_{n+1}}(t_{n+1}-s)^{-\alpha}ds\Big],
\eex
and
\bex
e^{n+1}_{h,\tau}
= \frac{1}{\Gamma(1-\alpha)}\Big[\int_{0}^{t_{n}}(t_{n+1}-s)^{-\alpha}\phi_{s}(s)ds
- \dps\sum_{k=0}^{n-1}\frac{\phi(t_{k+1})-\phi(t_{k})}{\Dt_{k+1}}
\int_{t_{k}}^{t_{k+1}}(t_{n+1}-s)^{-\alpha}ds\Big].
\eex
For the graded mesh, i.e., $t_{n}=\big({n\over M}\big)^rT, r\geq1, n=0,1,\cdots,M$,
which is particularly interesting for this problem and also the focus of this paper,
a direct calculation gives
\beq
b_j=\frac{T^{1-\alpha}}{\Gamma(2-\alpha)M^{(1-\alpha)r}}\Big[\big((n+1)^{r}-(n-j)^{r}\big)^{1-\alpha}-\big((n+1)^{r}-(n-j+1)^{r}\big)^{1-\alpha}\Big].
\eeq
Noting that when $r=1$, it is the uniform mesh.

It can be proved \cite{LX07,HZX20} that
the truncation error $e^{n+1}_{l,\tau}$ and $e^{n+1}_{h,\tau}$ can be bounded by $c_{\phi}\tau^{2-\alpha}$, {\color{black}where $c_{\phi}$ is a positive constant depending on the regularity of $\phi$.} 

The above approximation motivates us to consider the following scheme for
\eqref{re_prob} and \eqref{re_prob2}:
 \begin{subequations}\label{sche1}
 \begin{align}
 &\dps L^\alpha_l\phi^{n+1}-\varepsilon^{2}\Delta\phi^{n+1}+\Big(1-\frac{R^{n+1}}{R^{n}}\Big)\theta^{2}\Delta\phi^{n}+\frac{R^{n+1}}{R^{n}}\big(F'(\phi^{n})+L^\alpha_h\phi^{n+1}\big)=0,\label{eq1_2}\\
 &\dps\frac{R^{n+1}-R^{n}}{\tau_{n+1}}=\frac{1}{2R^{n}}\Big(-\theta^{2}\Delta\phi^{n}+F'(\phi^{n})+L^\alpha_h\phi^{n+1}, \frac{\phi^{n+1}-\phi^{n}}{\tau_{n+1}}\Big),\label{eq1_3}
 \end{align}
 \end{subequations}
where $\phi^{n}$ is an approximation to $\phi(t_n)$.
Intuitively, this is a first order scheme since it is a combination of some approximations of
first order and $2-\alpha$ order to different terms of the equations.
However a rigorous proof of the convergence order
remains an open question. Instead, we will provide a proof for the stability of the scheme,
and the convergence order
will be verified through the numerical experiments to be presented later on.
\medskip

{\bf Stability.}
The stability property of the first order scheme \eqref{sche1} is presented and proved in the following theorem.
\begin{theorem}\label{th1}
Without any restriction on the mesh, the scheme \eqref{sche1} is stable in the sense that the following discrete energy dissipation law holds
\bex
\dps E^{n+1}_{\varepsilon,\theta}-E^{n}_{\varepsilon,\theta}
\leq0,  \ \ n=0,1,\cdots,
\eex
where
\beq
\dps E^{n}_{\varepsilon,\theta}:=\frac{1}{2}(\varepsilon^{2}-\theta^{2})\|\nabla\phi^{n}\|^{2}_{0}+|R^{n}|^{2},\ \ \theta^2\le \varepsilon^2.
 \eeq
\end{theorem}
\begin{proof}
First, taking the $L^2-$inner products of \eqref{eq1_2} and \eqref{eq1_3} with
$\phi^{n+1}-\phi^{n}$ and $2R^{n+1}$ respectively, we obtain:
\bry
 &\dps\big(L^\alpha_l\phi^{n+1},\phi^{n+1}-\phi^{n}\big)
 +\dps(\varepsilon^{2}-\theta^{2})\big(\nabla\phi^{n+1},\nabla(\phi^{n+1}-\phi^{n})\big)+\theta^{2}\|\nabla(\phi^{n+1}-\phi^{n})\|^{2}_{0}\\[9pt]
 &\hspace{5.3cm}\dps+\frac{R^{n+1}}{R^{n}}\big(-\theta^{2}\Delta\phi^{n}+F'(\phi^{n}+L^\alpha_h\phi^{n+1}),\phi^{n+1}-\phi^{n}\big)=0, \\[9pt]
& \dps2R^{n+1}(R^{n+1}-R^{n})
=\dps\frac{R^{n+1}}{R^{n}}\big(-\theta^{2}\Delta\phi^{n}+F'(\phi^{n})+L^\alpha_h\phi^{n+1},\phi^{n+1}-\phi^{n}\big).
\ery
Then combining these two equalities, using the identity
\bex
2a^{n+1}(a^{n+1}-a^n)=|a^{n+1}|^{2}-|a^{n}|^{2}+|a^{n+1}-a^{n}|^{2},
\eex
and dropping some non-essential positive terms, we obtain
\beq
\dps E^{n+1}_{\varepsilon,\theta}-E^{n}_{\varepsilon,\theta}
\leq -\big(L^\alpha_l\phi^{n+1},\phi^{n+1}-\phi^{n}\big)
=-\frac{b_0}{\Dt_{n+1}}\|\phi^{n+1}-\phi^{n}\|^{2}_{0}
\leq0,~~n=0,1,2,\cdots,
\eeq
This ends the proof.
\end{proof}

{\bf Implementation.}
Beside of its unconditional stability proved in Theorem \ref{th1},
another notable property of the scheme \eqref{sche1} is that it can be efficiently implemented.
To see that, we first eliminate $R^{n+1}$ from \eqref{eq1_2} by using \eqref{eq1_3} to obtain
\bq\label{te1b}
\dps b_{0}\frac{\phi^{n+1}-\phi^{n}}{\Dt_{n+1}}-\varepsilon^{2}\Delta\phi^{n+1}+\theta^{2} \Delta\phi^{n}
+\Big[1+\frac{1}{2|R^{n}|^2}(\gamma^{n},\phi^{n+1}-\phi^{n})\Big]\gamma^{n}=0,
\eq
where
\beq
\gamma^{n}:=-\theta^{2}\Delta\phi^{n}+F'(\phi^{n})+L^\alpha_h\phi^{n+1}.
\eeq
Then reformulating \eqref{te1b} gives
\be\label{eq3}
\dps\Big(\frac{b_{0}}{\Dt_{n+1}}I_d-\varepsilon^{2}\Delta\Big)\phi^{n+1}+(\gamma^{n},\phi^{n+1})\frac{\gamma^{n}}{2|R^{n}|^{2}}
=\frac{b_0}{\Dt_{n+1}}\phi^{n}-\theta^{2}\Delta\phi^{n}\!-\!\Big[\frac{R^{n}}{R^{n}}\!-\!\frac{1}{2|R^{n}|^{2}}(\gamma^{n},\phi^{n})\Big]\gamma^{n}.
\ee
Denoting the right hand side of \eqref{eq3} by $g(\phi^{n})$,
we see that the problem \eqref{eq3} can be solved in two steps
as follows:
\begin{subequations}\label{semi_D}
 \begin{align}
 &\begin{cases}
   \begin{array}{r@{}l}
    &\dps\big(\frac{b_{0}}{\Dt_{n+1}}I_d-\varepsilon^{2}\Delta\big)\phi^{n+1}_{1}=-\frac{\gamma^{n}}{2|R^{n}|^{2}},\\[9pt]
    &\dps\mbox{Neumann boundary condition or periodic boundary condition on $\phi^{n+1}_1$;}\\
    \end{array}
 \end{cases}\label{semi_D_1}\\
 &\begin{cases}
    \begin{array}{r@{}l}
      &\dps\big(\frac{b_{0}}{\Dt_{n+1}}I_d-\varepsilon^{2}\Delta\big)\phi^{n+1}_{2}=g(\phi^{n}),\\[9pt]
      &\dps\mbox{Neumann boundary condition or periodic boundary condition on $\phi^{n+1}_2$;}\\
    \end{array}
  \end{cases}\label{semi_D_2}\\
 &\dps \phi^{n+1}=(\gamma^{n},\phi^{n+1})\phi^{n+1}_{1}+\phi^{n+1}_{2}.\label{semi_D_3}
 \end{align}
 \end{subequations}
In a first look it seems that \eqref{semi_D_3} governing the unknown $\phi^{n+1}$ is an implicit equation.
However a careful examination shows that
$(\gamma^{n},\phi^{n+1})$ in \eqref{semi_D_3} can be determined explicitly.
In fact, taking the inner product of \eqref{semi_D_3} with $\gamma^{n}$ yields
\bq\label{eq4}
(\gamma^{n},\phi^{n+1})+ \sigma^{n}(\gamma^{n},\phi^{n+1})=(\gamma^{n},\phi^{n+1}_{2}),
\eq
where
\bex
\sigma^{n}=-(\gamma^{n},\phi^{n+1}_{1})
=\Big(\gamma^{n},A^{-1}\frac{\gamma^{n}}{2|R^{n}|^{2}}\Big)\ \ \mbox{with } A=\frac{b_{0}}{\Dt_{n+1}} I_d-\varepsilon^{2}\Delta.
\eex
Note that
$A^{-1}$ is a positive definite operator. Thus $\sigma^{n}\geq0$.
Then it follows from \eqref{eq4} that
\bq\label{eq5}
(\gamma^{n},\phi^{n+1})=\frac{(\gamma^{n},\phi^{n+1}_{2})}{1+\sigma^{n}}.
\eq
Using this expression, $\phi^{n+1}$ can be explicitly computed from \eqref{semi_D_3}.

In detail, the scheme \eqref{sche1} results in the following algorithm at each time step:

(i) Calculation of $\phi^{n+1}_{1}$ and $\phi^{n+1}_{2}$: solving the elliptic  problems
\eqref{semi_D_1} and \eqref{semi_D_2} respectively, which can be realized in parallel.

(ii) Evaluation of $(\gamma^{n},\phi^{n+1})$ using \eqref{eq5}, then $\phi^{n+1}$ using \eqref{semi_D_3}.

Thus the overall computational cost at each time step comes essentially from
solving two second-order elliptic problems with constant coefficients, for which there exist different fast solvers depending on
the spatial discretization method.

\section {Higher order schemes}\label{sect3}
\setcounter{equation}{0}

\subsection {A $2-\alpha$ order scheme}

Using L1-CN formula \cite{HZX20} to discrete both the local and history parts of the fractional derivative at $t=t_{n+\frac{1}{2}}:=\frac{t_{n}+t_{n+1}}{2}$ gives
  \bry
 D^{\alpha,t_{n}}_{l,t_{n+\frac{1}{2}}}\phi
  =&\dps\frac{1}{\Gamma(1-\alpha)}\int_{t_{n}}^{t_{n+\frac{1}{2}}}(t_{n+\frac{1}{2}}-s)^{-\alpha}\phi_{s}(s)ds\\[9pt]
  =&\dps\frac{1}{\Gamma(1-\alpha)}\frac{\phi(t_{n+1})-\phi(t_{n})}{\Dt_{n+1}}
  \int_{t_{n}}^{t_{n+\frac{1}{2}}}(t_{n+1}-s)^{-\alpha}ds + e^{n+\frac{1}{2}}_{l,\tau}\\[12pt]
   :=&\widetilde{L}^{\alpha}_{l}\phi(t_{n+\frac{1}{2}}) + e^{n+\frac{1}{2}}_{l,\tau},
  \ery
  \bry
 D^{\alpha,t_{n}}_{h,t_{n+\frac{1}{2}}}\phi
  =&\dps\frac{1}{\Gamma(1-\alpha)}\sum_{k=0}^{n-1} \int_{t_{k}}^{t_{k+1}}
  (t_{n+\frac{1}{2}}-s)^{-\alpha}\phi_{s}(s)ds\\[9pt]
  =&\dps\sum_{k=0}^{n-1}\frac{1}{\Gamma(1-\alpha)}\frac{\phi(t_{k+1})-\phi(t_{k})}{\Dt_{k+1}}
  \int_{t_{k}}^{t_{k+1}}(t_{n+\frac{1}{2}}-s)^{-\alpha}ds + e^{n+\frac{1}{2}}_{h,\tau}\\[12pt]
  :=&\dps \widetilde{L}^{\alpha}_{h}\phi(t_{n+\frac{1}{2}})+e^{n+\frac{1}{2}}_{h,\tau},
  \ery
 where the L1-CN difference operators
 $\widetilde{L}^{\alpha}_{l}$ and $\widetilde{L}^{\alpha}_{h}$ are defined
 respectively by
 \beq
  \widetilde{L}^\alpha_l\phi(t_{n+\frac{1}{2}})
 = \widetilde{b}_{0}\frac{\phi(t_{n+1})-\phi(t_{n})}{\Dt_{n+1}},\quad \widetilde{L}^\alpha_h\phi(t_{n+\frac{1}{2}})
=\sum_{k=0}^{n-1}\widetilde{b}_{n-k}\frac{\phi(t_{k+1})-\phi(t_{k})}{\Dt_{k+1}},
\eeq
with
\bex
 \widetilde{b}_{0}=\frac{\Dt_{n+1}^{1-\alpha}}{\Gamma(2-\alpha)2^{1-\alpha}}>0,~~~~~ \widetilde{b}_{n-k}=\frac{1}{\Gamma(1-\alpha)}\int_{t_{k}}^{t_{k+1}}(t_{n+\frac{1}{2}}-s)^{-\alpha}ds>0, k=0,1,\cdots,n-1.
\eex
The corresponding truncation errors $e^{n+1}_{l,\tau}$ and $e^{n+1}_{h,\tau}$ are defined respectively by
\bex
e^{n+\frac{1}{2}}_{\l,\tau}
= \frac{1}{\Gamma(1-\alpha)}\Big[\int_{t_{n}}^{t_{n+\frac{1}{2}}}(t_{n+\frac{1}{2}}-s)^{-\alpha}\phi_{s}(s)ds
- \dps\frac{\phi(t_{n+1})-\phi(t_{n})}{\Dt_{n+1}}\int_{t_{n}}^{t_{n+\frac{1}{2}}}(t_{n+\frac{1}{2}}-s)^{-\alpha}ds\Big],
\eex
and
\bex
e^{n+\frac{1}{2}}_{h,\tau}
= \frac{1}{\Gamma(1-\alpha)}\Big[\int_{0}^{t_{n}}(t_{n+\frac{1}{2}}-s)^{-\alpha}\phi_{s}(s)ds
- \dps\sum_{k=0}^{n-1}\frac{\phi(t_{k+1})-\phi(t_{k})}{\Dt_{k+1}}
\int_{t_{k}}^{t_{k+1}}(t_{n+\frac{1}{2}}-s)^{-\alpha}ds\Big].
\eex
It has been proved \cite{HZX20} that the truncation error $e^{n+1}_{l,\tau}$ and $e^{n+1}_{h,\tau}$ are both
of $2-\alpha$ order with respect to $\tau$.

Applying the difference operators $\widetilde{L}^{\alpha}_{l}$ and $\widetilde{L}^{\alpha}_{h}$ to discretize
the fractional derivative and
the Crank-Nicolson scheme to the remaining terms of the system \eqref{re_prob2} and \eqref{re_prob}, we obtain L1-CN scheme as follows:
 \begin{subequations}\label{L1_CN_s}
 \begin{align}
 &\dps \widetilde{L}^{\alpha}_{l}\phi^{n+\frac{1}{2}}-\varepsilon^{2}\frac{\Delta(\phi^{n+1}+\phi^{n})}{2}+\Big(1-\frac{R^{n+1}+R^{n}}{2R^{n+\frac{1}{2}}}\Big)\theta^{2}\Delta\phi^{n+\frac{1}{2}} \notag\\
 &\hspace{7cm}+\frac{R^{n+1}+R^{n}}{2R^{n+\frac{1}{2}}}\Big(F'(\phi^{n+\frac{1}{2}})+\widetilde{L}^{\alpha}_{h}\phi^{n+\frac{1}{2}}\Big),\label{L1_CN_1}\\
 &\dps\frac{R^{n+1}-R^{n}}{\tau_{n+1}}=\frac{1}{2R^{n+\frac{1}{2}}}
 \Big(-\theta^{2}\Delta\phi^{n+\frac{1}{2}}+F'(\phi^{n+\frac{1}{2}})+\widetilde{L}^{\alpha}_{h}\phi^{n+\frac{1}{2}},
\frac{\phi^{n+1}-\phi^{n}}{\tau_{n+1}}\Big),\label{L1_CN_2}
 \end{align}
 \end{subequations}
 where $\phi^{n+\frac{1}{2}}:=\phi^{n}+\frac{\tau_{n+1}}{2\Dt_{n}}[\phi^{n}-\phi^{n-1}]$ and  $R^{n+\frac{1}{2}}:=R^{n}+\frac{\tau_{n+1}}{2\Dt_{n}}[R^{n}-R^{n-1}]$ are the explicit approximation to $\phi(t_{n+\frac{1}{2}})$ and $R(t_{n+\frac{1}{2}}),$ respectively.

The scheme \eqref{L1_CN_s} is expected to have $2-\alpha$ order
convergence, since formally the approximation to the fractional derivative $_{0}{}\!D^{\alpha}_{t_{n+{1\over 2}}}\phi$ is of $2-\alpha$ order,
and the remaining approximations are based on the Crank-Nicolson formula, which is
second-order accurate.

The unconditional stability of the L1-CN scheme \eqref{L1_CN_s} is proved in the following theorem.
\begin{theorem}\label{th2}
The L1-CN scheme \eqref{L1_CN_s} satisfies the energy law:
\bex
&&\dps\frac{\varepsilon^{2}-\theta^{2}}{2}\|\nabla\phi^{n+1}\|^{2}_{0}+\frac{\theta^{2}}{4}\|\nabla(\phi^{n+1}-\phi^{n})\|_{0}^{2}+|R^{n+1}|^{2}\\[9pt]
&&\dps-\Big[\frac{\varepsilon^{2}-\theta^{2}}{2}\|\nabla\phi^{n}\|^{2}_{0}+\frac{\theta^{2}}{4}\Big(\frac{\Dt_{n+1}}{\Dt_{n}}\Big)^{2}\|\nabla(\phi^{n}-\phi^{n-1})\|_{0}^{2}+|R^{n}|^{2}\Big]
\leq0,~~n=0,1,\cdots.
\eex
Therefore, the scheme \eqref{L1_CN_s} is unconditionally stable
when i) the mesh is uniform, i.e., $\frac{\Dt_{n+1}}{\Dt_{n}}=1$;
ii) $\theta=0$.
In the former case, the discrete energy $\frac{\varepsilon^{2}-\theta^{2}}{2}\|\nabla\phi^{n+1}\|^{2}_{0}+\frac{\theta^{2}}{4}\|\nabla(\phi^{n+1}-\phi^{n})\|_{0}^{2}+|R^{n+1}|^{2}$
decreases during the time stepping,
while in the latter case, the energy $\widetilde{E}^{n+1}_{\varepsilon}:=\frac{\varepsilon^{2}}{2}\|\nabla\phi^{n+1}\|^{2}_{0}+|R^{n+1}|^{2}$ decays in time.
\end{theorem}

\begin{proof}
By taking the inner products of \eqref{L1_CN_1} and \eqref{L1_CN_2} with $\phi^{n+1}-\phi^{n}$ and $R^{n+1}+R^{n}$ respectively, we have
\bry
 &\dps(\widetilde{L}^{\alpha}_{l}\phi^{n+\frac{1}{2}},\phi^{n+1}-\phi^{n})
 +\frac{\varepsilon^{2}-\theta^{2}}{2}\big(\|\nabla\phi^{n+1}\|^{2}-\|\nabla\phi^{n}\|^{2}\big)\\[9pt]
 &\hspace{4cm}\dps+\frac{\theta^{2}}{2}\big(\nabla(\phi^{n+1}-\phi^{n}-\frac{\Dt_{n+1}}{\Dt_{n}}(\phi^{n}-\phi^{n-1})),\nabla(\phi^{n+1}-\phi^{n})\big)\\[9pt]
 &\hspace{4cm}\dps+\frac{R^{n+1}+R^{n}}{2R^{n+\frac{1}{2}}}\Big(-\theta^{2}\Delta\phi^{n+\frac{1}{2}}+F'(\phi^{n+\frac{1}{2}})+\widetilde{L}^{\alpha}_{h}\phi^{n+\frac{1}{2}},\frac{\phi^{n+1}-\phi^{n}}{\Dt_{n+1}}\Big)=0,\\[20pt]
 &\dps\frac{|R^{n+1}|^{2}-|R^{n}|^{2}}{\Dt_{n+1}}
 =\frac{R^{n+1}+R^{n}}{2R^{n+\frac{1}{2}}}\Big(-\theta^{2}\Delta\phi^{n+\frac{1}{2}}
 +F'(\phi^{n+\frac{1}{2}})+\widetilde{L}^{\alpha}_{h}\phi^{n+\frac{1}{2}},\frac{\phi^{n+1}-\phi^{n}}{\Dt_{n+1}}\Big).
\ery
Summing up the above two equalities, applying the identity
\beq
2a^{n+1}(a^{n+1}-a^{n})=|a^{n+1}|^{2}-|a^{n}|^{2}+|a^{n+1}-a^{n}|^2,
\eeq
and dropping some non-essential positive terms, we obtain
\bex
&&\dps\frac{\varepsilon^{2}-\theta^{2}}{2}\|\nabla\phi^{n+1}\|^{2}_{0}+\frac{\theta^{2}}{4}\|\nabla(\phi^{n+1}-\phi^{n})\|_{0}^{2}+|R^{n+1}|^{2}\\[9pt]
&&\dps-\Big[\frac{\varepsilon^{2}-\theta^{2}}{2}\|\nabla\phi^{n}\|^{2}_{0}+\frac{\theta^{2}}{4}\Big(\frac{\Dt_{n+1}}{\Dt_{n}}\Big)^{2}\|\nabla(\phi^{n}-\phi^{n-1})\|_{0}^{2}+|R^{n}|^{2}\Big]
\\[9pt]
&& \le -\big(\widetilde{L}^{\alpha}_{l}\phi^{n+\frac{1}{2}},\phi^{n+1}-\phi^{n})
= -\frac{\widetilde{b}_{0}}{\Dt_{n+1}}\|\phi^{n+1}-\phi^{n}\|^{2}_{0}
\leq0.
\eex
This completes the proof.
\end{proof}

The L1-CN scheme \eqref{L1_CN_s} can be efficiently implemented
by following the lines similar to the first order scheme \eqref{sche1} described in the previous section.

\subsection {A second order scheme}

We first integrate the equation \eqref{re_prob} and \eqref{re_prob2} from $t_{n}$ to $t_{n+1}$,
and multiply by $\frac{1}{\Dt_{n+1}}$ to give:
\be\label{inteq}\nonumber
 &\dps\frac{1}{\Dt_{n+1}}\int_{t_{n}}^{t_{n+1}}D^{\alpha,t_{n}}_{l,t}\phi dt-\frac{1}{\Dt_{n+1}}\int_{t_{n}}^{t_{n+1}}\Big[\varepsilon^{2}\Delta\phi+\Big(1-\frac{R(t)}{R(t)}\Big)\theta^{2}\Delta\phi+\frac{R(t)}{R(t)}\big(F'(\phi)+D^{\alpha,t_{n}}_{h,t}\phi\big)\Big]dt=0,\\[13pt]
 &\hspace{-1.5cm}
 \dps\frac{1}{\Dt_{n+1}}\int_{t_{n}}^{t_{n+1}}\frac{d R}{d t}dt=\frac{1}{\Dt_{n+1}}\int_{t_{n}}^{t_{n+1}}\frac{1}{2R(t)}\Big(-\theta^{2}\Delta\phi+F'(\phi)+D^{\alpha,t_n}_{h,t}\phi,
\frac{\partial \phi}{\partial t}\Big)dt, \ n=0,1,\cdots. \hspace{-.5cm}
\ee
The idea is to construct second order approximations for the necessary terms involved in the above equations.
We define the finite difference operators $\widehat{L}^{\alpha}_{l}$ and $ \widehat{L}^{\alpha}_{h}$ by:
  \bry
  \widehat{L}^{\alpha}_{l}\phi(t_{n+\frac{1}{2}}):=&\dps\frac{1}{\Dt_{n+1}}\int_{t_{n}}^{t_{n+1}}\Big[\frac{1}{\Gamma(1-\alpha)}\int_{t_{n}}^{t}(t-s)^{-\alpha}\frac{\phi(t_{n+1})-\phi(t_{n})}{\Dt_{n+1}}ds\Big]dt\\[12pt]
  =&\dps \widehat{b}_{0}\frac{\phi(t_{n+1})-\phi(t_{n})}{\Dt_{n+1}},
  \ery
  \bry
 \widehat{L}^{\alpha}_{h}\phi(t_{n+\frac{1}{2}})
  =&\dps\frac{1}{\Dt_{n+1})}\int_{t_{n}}^{t_{n+1}}\Big[\sum_{k=1}^{n}\frac{1}{\Gamma(1-\alpha)}\int_{t_{k}}^{t_{k+1}}(t_{n+\frac{1}{2}}-s)^{-\alpha}\frac{\phi(t_{k+1})-\phi(t_{k})}{\Dt_{k+1}}ds\Big]dt\\[9pt]
  =&\dps \sum_{k=0}^{n-1}\widehat{b}_{n-k}\frac{\phi(t_{k+1})-\phi(t_{k})}{\Dt_{k+1}},
  \ery
 where
\bex
 \widehat{b}_{0}=\frac{\Dt_{n+1}^{1-\alpha}}{\Gamma(3-\alpha)},~~~~~ \widehat{b}_{n-k}=\frac{1}{\Gamma(1-\alpha)\Dt_{n+1}}\int_{t_{n}}^{t_{n+1}}\int_{t_{k}}^{t_{k+1}}(t-s)^{-\alpha}dsdt, k=0,1,\cdots,n-1.
\eex
We want to use these two operators to approximate the local term
$\frac{1}{\Dt_{n+1}}\int_{t_{n}}^{t_{n+1}} D^{\alpha,t_{n}}_{l,t}dt$ and the history term
$\frac{1}{\Dt_{n+1}}\int_{t_{n}}^{t_{n+1}} D^{\alpha,t_{n}}_{h,t}dt$ respectively.
Note that a similar operator without splitting, called $L1^{+}$ formula,
has been used in \cite{JLGZ20} to approximate the Caputo fractional derivative,
and the resulting scheme has been numerically found to be second order accurate.
However, there is no available analysis for the truncation error, nor for the stability.

Applying $\widehat{L}^{\alpha}_{l}\phi(t_{n+\frac{1}{2}})$ and
$ \widehat{L}^{\alpha}_{h}\phi(t_{n+\frac{1}{2}})$ to
approximate $\frac{1}{\Dt_{n+1}}\int_{t_{n}}^{t_{n+1}} D^{\alpha,t_{n}}_{l,t}dt$ and
$\frac{1}{\Dt_{n+1}}\int_{t_{n}}^{t_{n+1}} D^{\alpha,t_{n}}_{h,t}dt$ respectively, and
trapezoidal formula to approximate the remaining terms in \eqref{inteq}, we arrive at
the following $L1^{+}$-CN scheme:
 \begin{subequations}\label{L1p_CN1}
 \begin{align}
 &\dps \widehat{L}^{\alpha}_{l}\phi^{n+\frac{1}{2}}-\varepsilon^{2}\frac{\Delta(\phi^{n+1}+\phi^{n})}{2}+\Big(1-\frac{R^{n+1}+R^{n}}{2R^{n+\frac{1}{2}}}\Big)\theta^{2}\Delta\phi^{n+\frac{1}{2}} \notag\\
 &\hspace{7cm}+\frac{R^{n+1}+R^{n}}{2R^{n+\frac{1}{2}}}\Big(F'(\phi^{n+\frac{1}{2}})+\widehat{L}^{\alpha}_{h}\phi^{n+\frac{1}{2}}\Big),\label{L1_CN_1p}\\
 &\dps\frac{R^{n+1}-R^{n}}{\tau_{n+1}}=\frac{1}{2R^{n+\frac{1}{2}}}
 \Big(-\theta^{2}\Delta\phi^{n+\frac{1}{2}}+F'(\phi^{n+\frac{1}{2}})+\widehat{L}^{\alpha}_{h}\phi^{n+\frac{1}{2}},
\frac{\phi^{n+1}-\phi^{n}}{\tau_{n+1}}\Big),\label{L1_CN_2p}
 \end{align}
 \end{subequations}
 where $\phi^{n+\frac{1}{2}}:=\phi^{n}+\frac{\tau_{n+1}}{2\tau_{n}}[\phi^{n}-\phi^{n-1}]$ and  $R^{n+\frac{1}{2}}:=R^{n}+\frac{\tau_{n+1}}{2\tau_{n}}[R^{n}-R^{n-1}]$ are the explicit approximation to $\phi(t_{n+\frac{1}{2}})$ and $R(t_{n+\frac{1}{2}}),$ respectively.

We leave the error estimation as an open question, but present the stability result for the scheme in the following theorem.

\begin{theorem}\label{th3}
For the $L1^{+}$-CN scheme \eqref{L1p_CN1}, it holds
\bex 
&&\dps\frac{\varepsilon^{2}-\theta^{2}}{2}\|\nabla\phi^{n+1}\|^{2}_{0}+\frac{\theta^{2}}{4}\|\nabla(\phi^{n+1}-\phi^{n})\|_{0}^{2}+|R^{n+1}|^{2}\\[9pt]
&&\dps-\Big[\frac{\varepsilon^{2}-\theta^{2}}{2}\|\nabla\phi^{n}\|^{2}_{0}+\frac{\theta^{2}}{4}\Big(\frac{\Dt_{n+1}}{\Dt_{n}}\Big)^{2}\|\nabla(\phi^{n}-\phi^{n-1})\|_{0}^{2}+|R^{n}|^{2}\Big]
\\[9pt]
&&\dps\leq-\frac{\widehat{b}_{0}}{\Dt_{n+1}}\|\phi^{n+1}-\phi^{n}\|^{2}_{0}\leq0,~~n=0,1,\cdots.
\eex
This implies that the scheme is unconditionally stable
in the cases of the uniform mesh or $\theta=0$.
\end{theorem}

\begin{proof}
The proof is very similar to Theorem \ref{th2}, we leave it to interested readers.
\end{proof}

\begin{remark}
One can also reformulate the original equation \eqref{FAC}
into the following equivalent form:
 \begin{subequations}\label{reform2}
 \begin{align}
& \dps D^{\alpha,t_n}_{l,t}\phi-\varepsilon^{2}\Delta\phi+\Big(1-\frac{R(t)}{\sqrt{\overline{E}_{\theta}(\phi)+C_{0}}}\Big)\theta^{2}\Delta\phi+\frac{R(t)}{\sqrt{\overline{E}_{\theta}(\phi)+C_{0}}}\Big(F'(\phi)+D^{\alpha,t_n}_{h,t}\phi\Big)=0,\\[11pt]
 &\dps\frac{d R}{d t}=\frac{1}{2\sqrt{\overline{E}_{\theta}(\phi)+C_{0}}}
 \Big(-\theta^{2}\Delta\phi+F'(\phi)+D^{\alpha,t_n}_{h,t}\phi,
\frac{\partial \phi}{\partial t}\Big), \ \ t\in [t_n, t_{n+1}], n=0,1,\cdots.
\end{align}
\end{subequations}
Starting with this equivalent system and following the discussion in the above sections,
it is also possible to construct unconditionally stable schemes based on L1, L1-CN and $L1^{+}-$CN formula for the discretization of
the fractional derivatives.
However, compared to the schemes constructed for the reformulation \eqref{re_prob2}-\eqref{re_prob},
a drawback using \eqref{reform2} is that one has to compute $\overline{E}^{n}_{\theta}$ or $\overline{E}^{n+\frac{1}{2}}_{\theta}$, which is more expensive than computing $R^{n}$ or $R^{n+\frac{1}{2}}$.
Remember that $\overline{E}^{n+\frac{1}{2}}_{\theta}$ is an explicit approximation to
$\overline{E}_{\theta}(\phi(t_{n+1/2}))$ involving the nonlocal terms $F_h(t_n,t_{n+\frac{1}{2}}; \phi)+\sum_{k=1}^{n-1}F_h(t_{k},t_{k+1}; \phi)$.
{\color{black} It is notable that
the new SAV approach developed recently in \cite{CLS20,CS20} may be applied to
deal with the nonlinear term and the history part of the fractional derivative.
Then energy stable schemes can be constructed based on the following reformulation:
\bry
&\dps D^{\alpha, t_n}_{l,t}\phi-\varepsilon^{2}\Delta\phi+\big(1-\eta(t)\big)
\theta^{2}\Delta\phi+\eta(t)\big(F'(\phi) + D^{\alpha, t_n}_{h,t}\phi\big)
=0,\\[9pt]
&\dps\frac{d\overline{E}_{\theta}(\phi)}{dt}=\eta(t)\Big(-\theta^{2}\Delta\phi+F'(\phi) + D^{\alpha, t_n}_{h,t}\phi,\frac{\partial \phi}{\partial t}\Big),
\ery
where the scalar auxiliary function $\eta(t)$ is a Lagrange multiplier with $\eta(0)=1.$

}
\end{remark}


\section{Numerical results}
\label{sect4}
\setcounter{equation}{0}

This section is devoted to numerical investigation of the proposed schemes
in terms of the accuracy and stability.
For the comparison purpose, we will repeat most of the numerical examples in our previous work \cite{HZX20}. In the following examples, we always set $\theta=0$ and $C_0=0$ in the schemes unless specified otherwise.
The spatial discretization is the Fourier method or Legendre Galerkin spectral method using numerical quadratures.
In order to test the accuracy of the proposed schemes, the error is measured by the maximum norm, i.e., $\dps\max_{1\leq n\leq M}\|\phi^{n}-\phi(t_{n})\|_{\infty}$
or $\dps\max_{1\leq n\leq M}\|\phi^{n}_{M}-\phi^{2n}_{2M}\|_{\infty}$, the latter norm will be used when  the exact solution is not available.
In order to reduce the computational complexity, a fast evaluation technique
based on the sum-of-exponentials approach \cite{JZZ17,YSZ17} is used to calculate
the history part $D^{\alpha,t_n}_{h,t}$ of the time fractional derivative.

\subsection{Convergence order test}
\begin{example}\label{expl1}
Consider the following fractional Allen-Cahn equation:
\bex
\dps_{0}{}\!D^{\alpha}_{t}\phi-\varepsilon^{2}\Delta \phi-\phi(1-\phi^{2})=s,\quad
(\x,t)\in (0,2\pi)^2\times(0,T],
\eex
subject to the periodic boundary condition, where $s(\x,t)$ is a fabricated source term
such that the exact solution is
$$\phi(\x,t)=0.2 t^5 \sin(x)\cos(y).$$

\end{example}
The Fourier spectral method with $128\times128$ modes is used to discretize the equations in space.
It has been checked that this Fourier mode number is large enough
so that the spatial discretization error is negligible compared to the temporal discretization.
We present in Figure \ref{fig1}
the error as functions of the time
step sizes in log-log scale with $T=1$.
It is observed, as expected, that the L1 scheme \eqref{semi_D}, L1-CN scheme \eqref{L1_CN_s},
and L$1^{+}$-CN scheme \eqref{L1p_CN1} achieve respectively the first order, $2-\alpha$ order, and second order convergence for all tested $\alpha$.

\begin{figure*}[htbp]
\begin{minipage}[t]{0.49\linewidth}
\centerline{\includegraphics[scale=0.55]{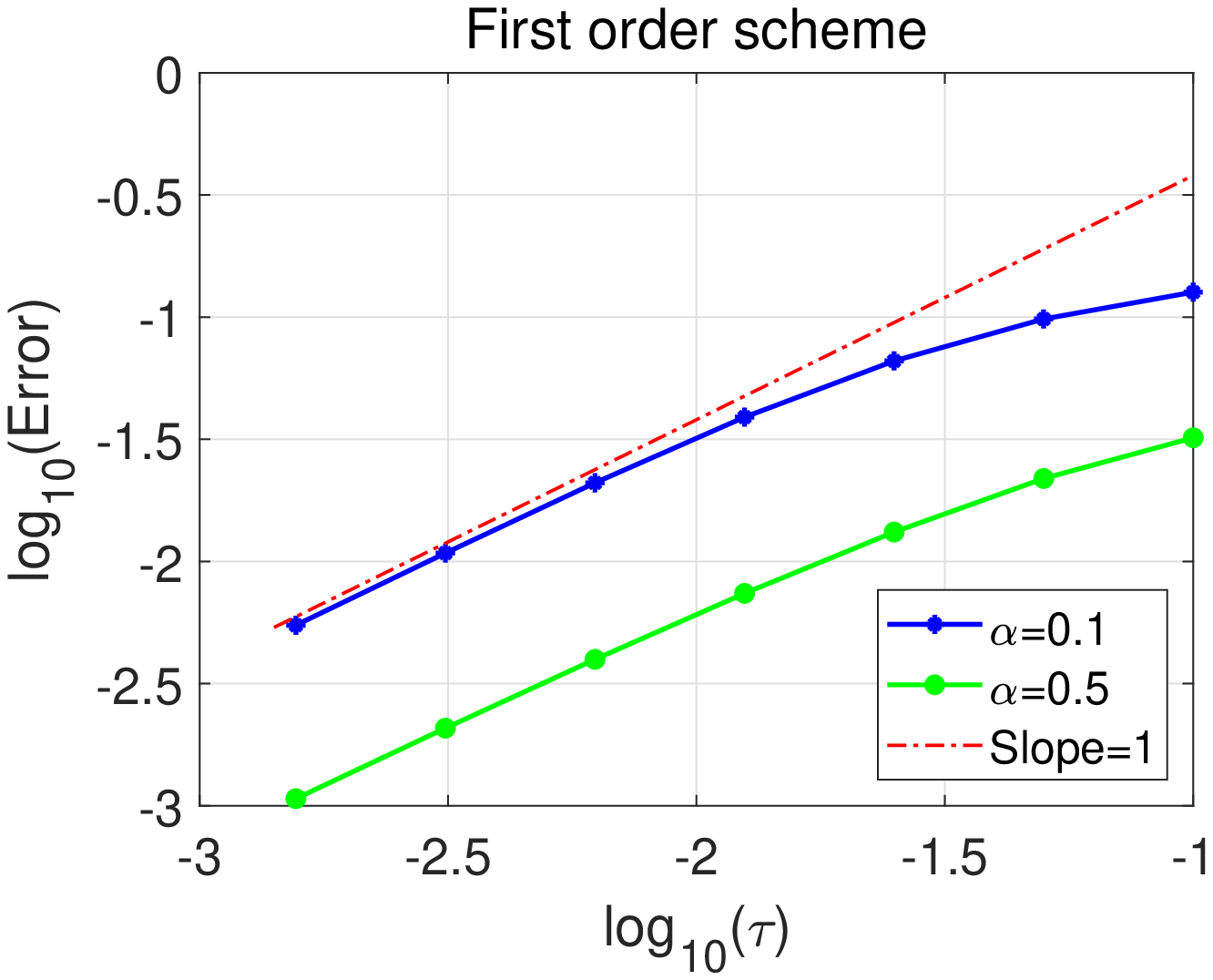}}
\centerline{(a) first order scheme}
\end{minipage}
\vskip 3mm
\begin{minipage}[t]{0.49\linewidth}
\centerline{\includegraphics[scale=0.55]{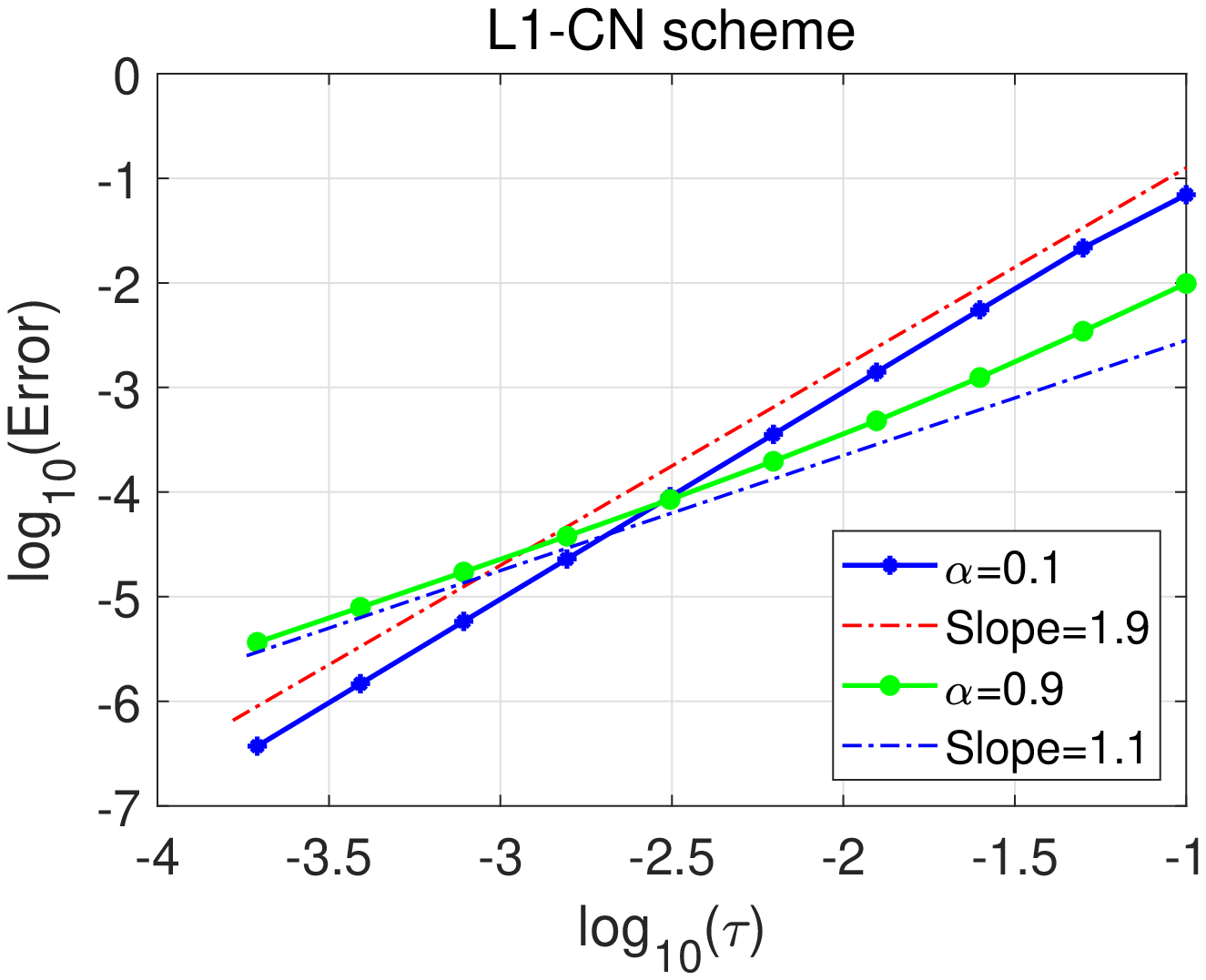}}
\centerline{(b) L1-CN scheme }
\end{minipage}
\begin{minipage}[t]{0.49\linewidth}
\centerline{\includegraphics[scale=0.55]{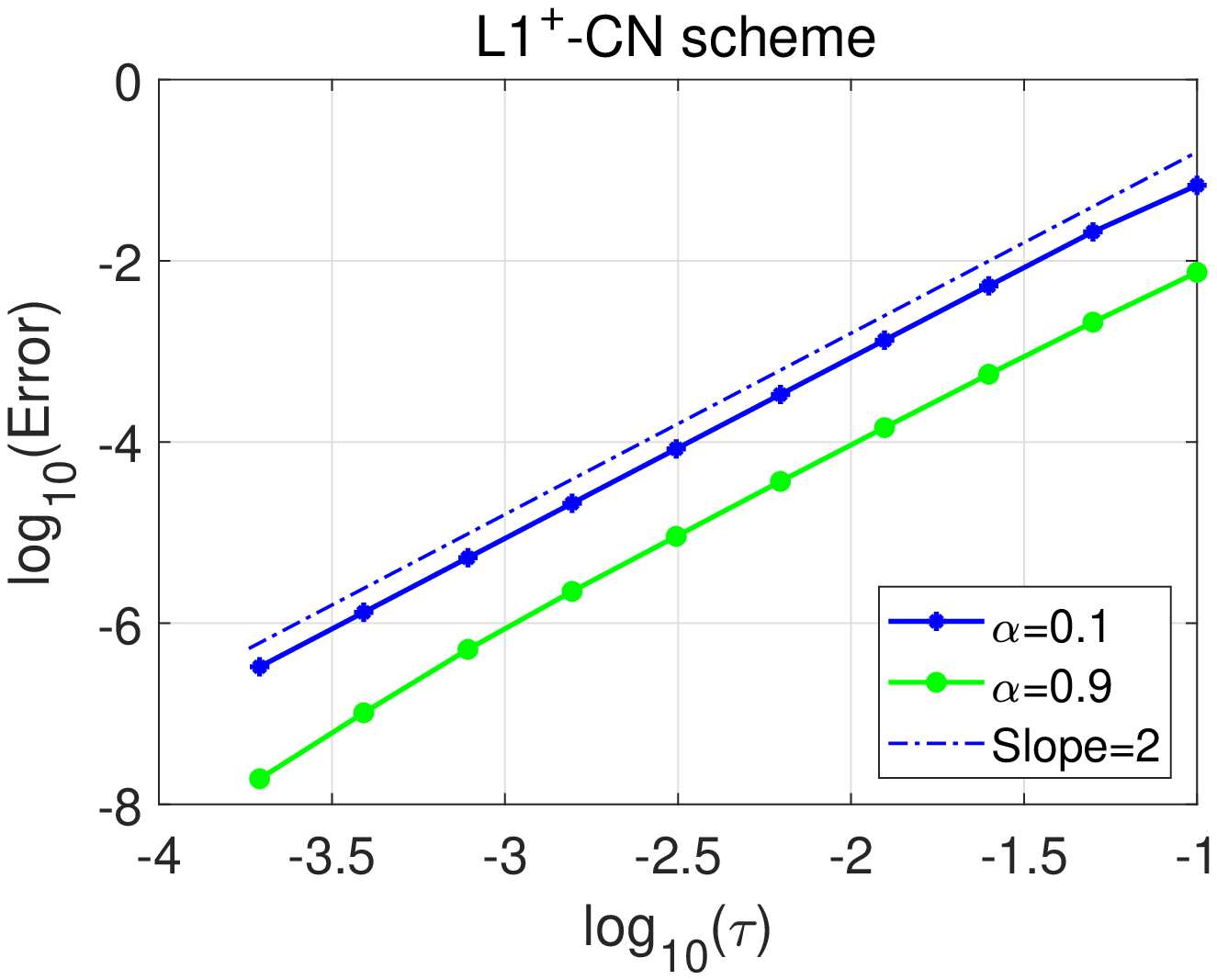}}
\centerline{(c) L$1^{+}$-CN scheme}
\end{minipage}
\caption{
(Example \ref{expl1}) Error decay versus the time step sizes for the first order scheme, L1-CN scheme,  and L$1^{+}$-CN scheme with several different $\alpha$.
}\label{fig1}
\end{figure*}

\begin{example}\label{expl2}
Consider the same equation as in Example \ref{expl1}, but with the Neumann boundary condition, and the exact solution
$$\phi(\x,t)=0.2(t^{\mu}+1)\cos(\pi x)\cos(\pi y), ~(\x,t)\in (-1,1)^2\times(0,T],$$
which has limited regularity at the initial time $t=0$.
\end{example}

In this test, a Legendre Galerkin spectral method with polynomials of degree 32 in each spatial direction
is used for the spatial discretization. The purpose of
this test is to not only verify the convergence rate of the schemes,
but also investigate the impact of the regularity on the accuracy.
In particular, we are interested in studying the impact of the graded mesh parameter $r$ on the convergence rate.
It will help us to choose the optimal value of $r$ to recover the convergence rate of the proposed schemes for
low regular solutions.
The calculation is performed by using the L1-CN scheme \eqref{L1_CN_s} and L$1^{+}$-CN scheme \eqref{L1p_CN1} with $M=64\times 2^{k},k=1,2,\cdots,9$.
In Figure \ref{fig2}, we
plot the $ L^{\infty}$ errors in log-log scale with respect to the maximum time step size, i.e.,
$\Dt=t_{M}-t_{M-1}$.
The presented results are in a perfect agreement with the expected convergence rates, i.e.,
$\min\{\mu r,2-\alpha\}$ order for the L1-CN scheme, and $\min\{\mu r,2\}$ order for the L$1^{+}$-CN scheme.
This suggests use of the graded mesh with $r=\frac{2-\alpha}{\mu}$ for the  L1-CN scheme and
$r=\frac{2}{\mu}$ for the L$1^{+1}$-CN scheme.
Doing so the schemes reach the optimal convergence rates for solutions of this class.

\begin{figure*}[htbp]
\begin{minipage}[t]{0.49\linewidth}
\centerline{\includegraphics[scale=0.55]{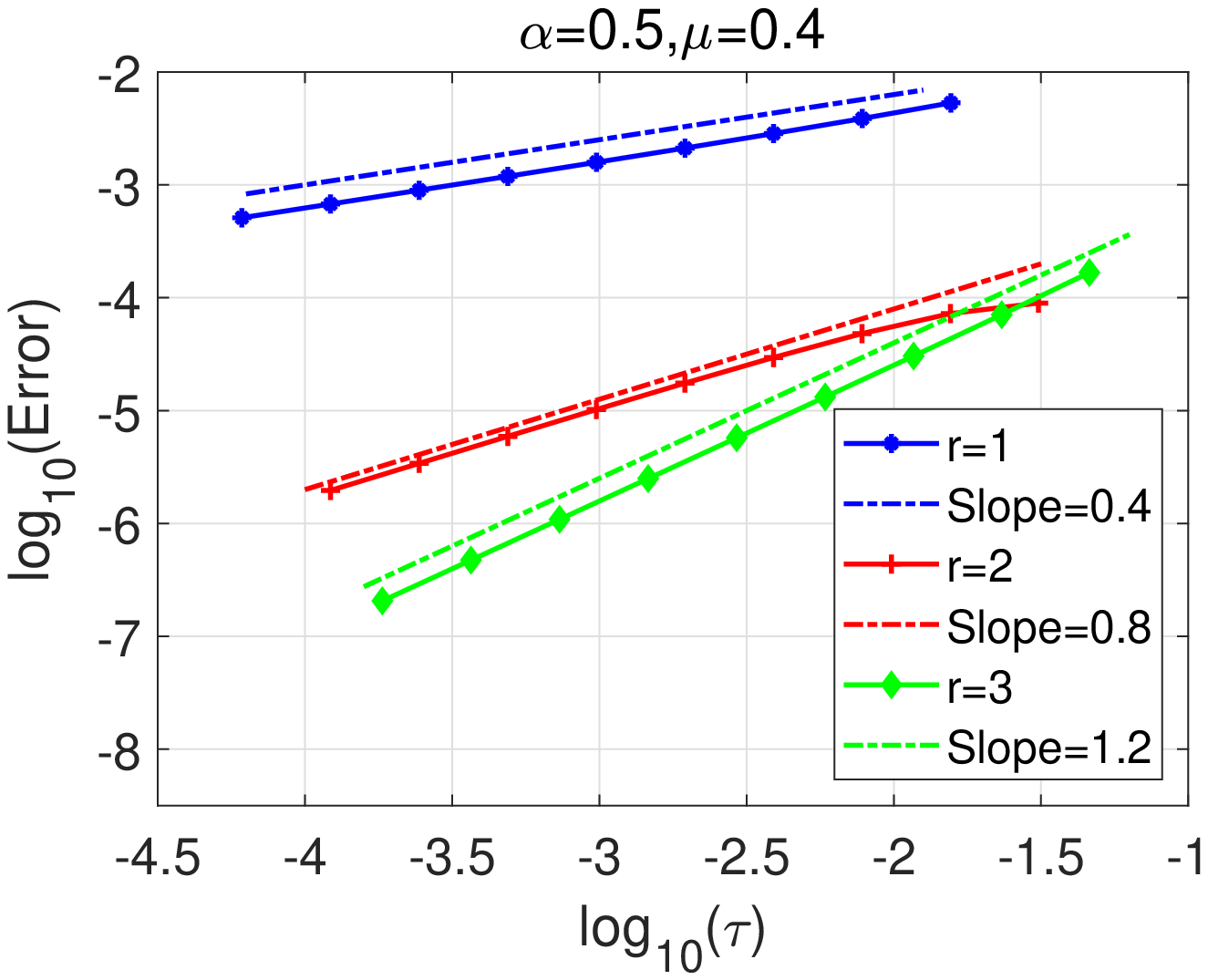}}
\centerline{\hspace{9cm}(a) L1-CN scheme}
\end{minipage}
\begin{minipage}[t]{0.49\linewidth}
\centerline{\includegraphics[scale=0.55]{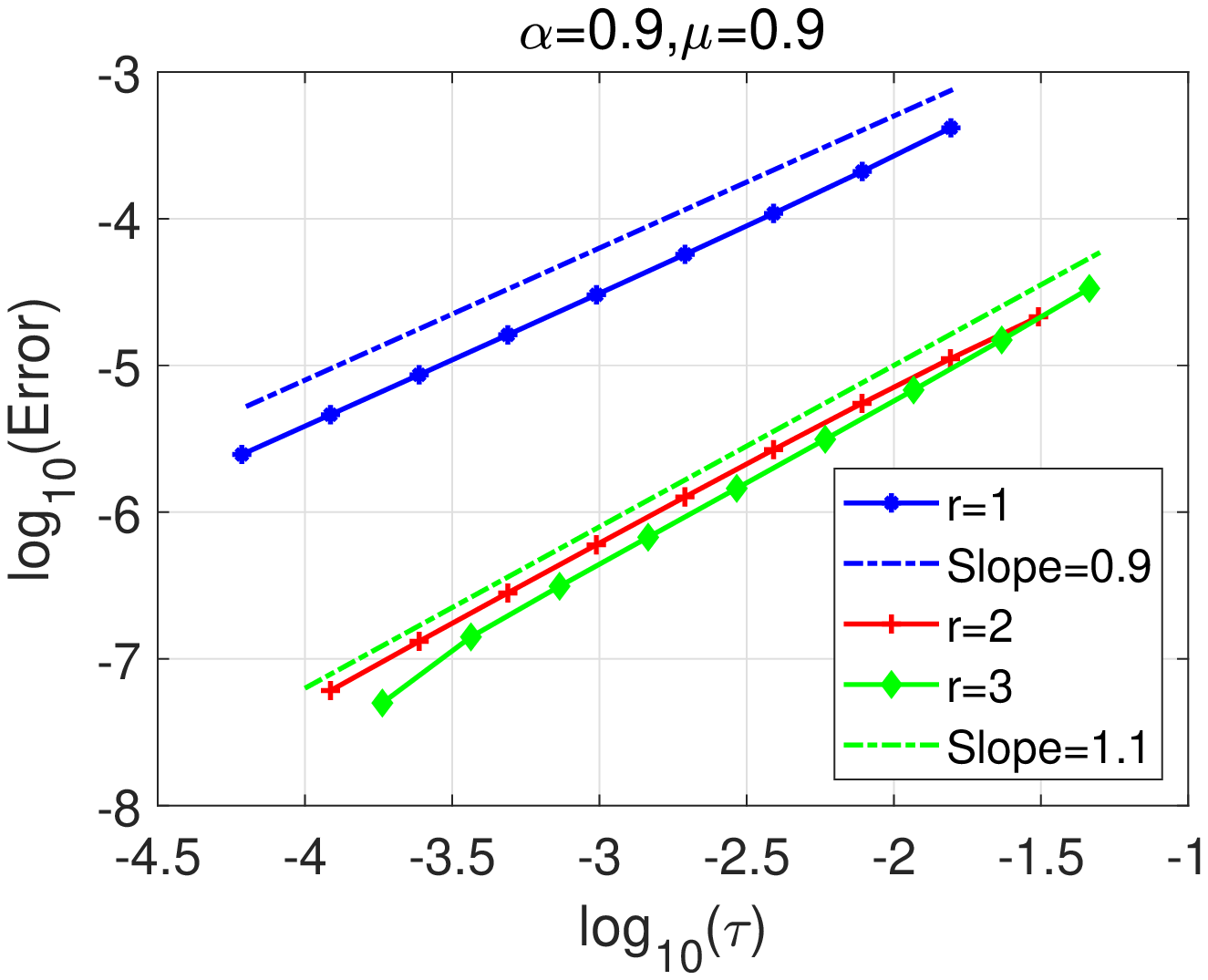}}
\centerline{}
\end{minipage}
\vskip 3mm
\begin{minipage}[t]{0.49\linewidth}
\centerline{\includegraphics[scale=0.55]{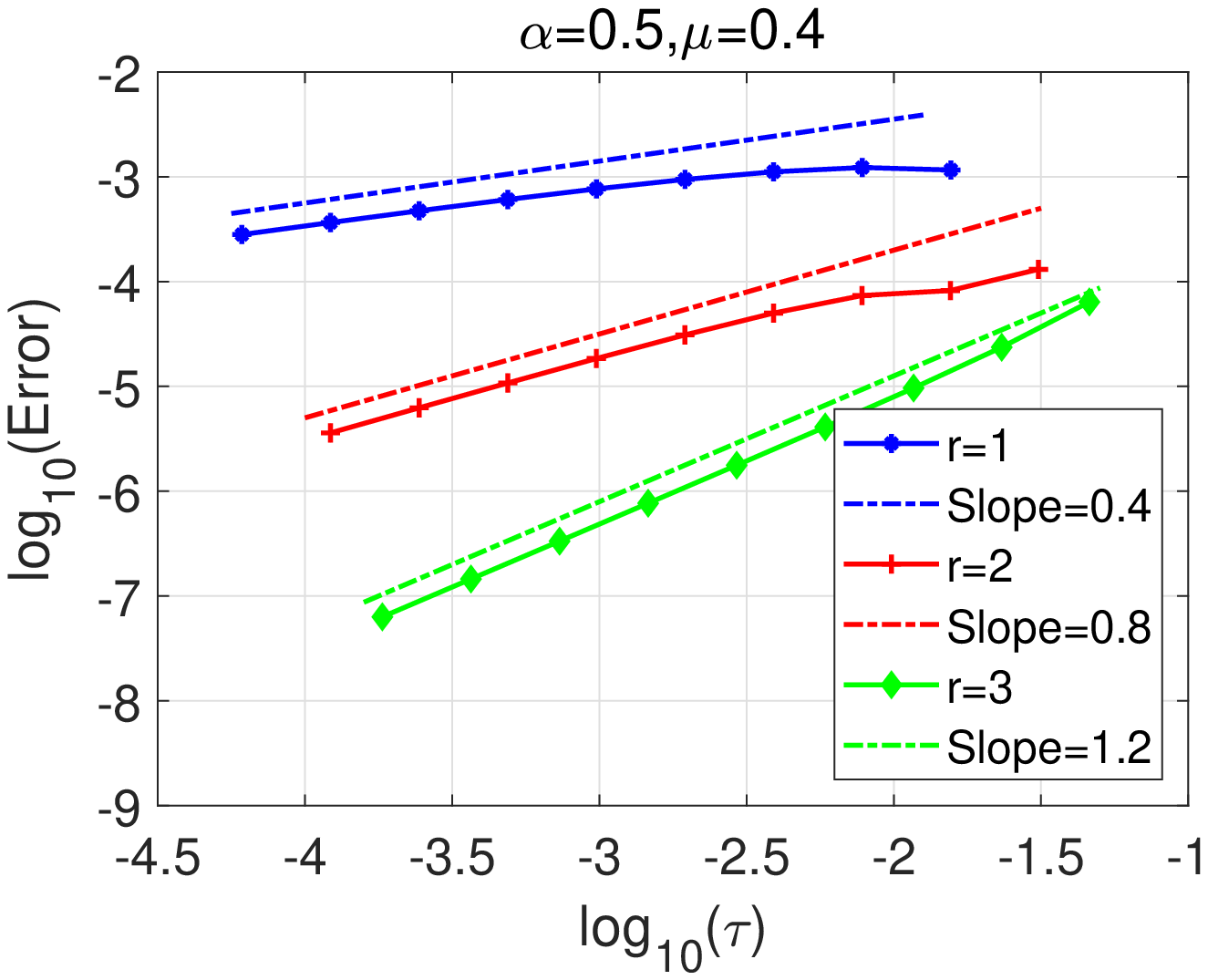}}
\centerline{\hspace{9cm}(b) L$1^{+}$-CN scheme}
\end{minipage}
\begin{minipage}[t]{0.49\linewidth}
\centerline{\includegraphics[scale=0.55]{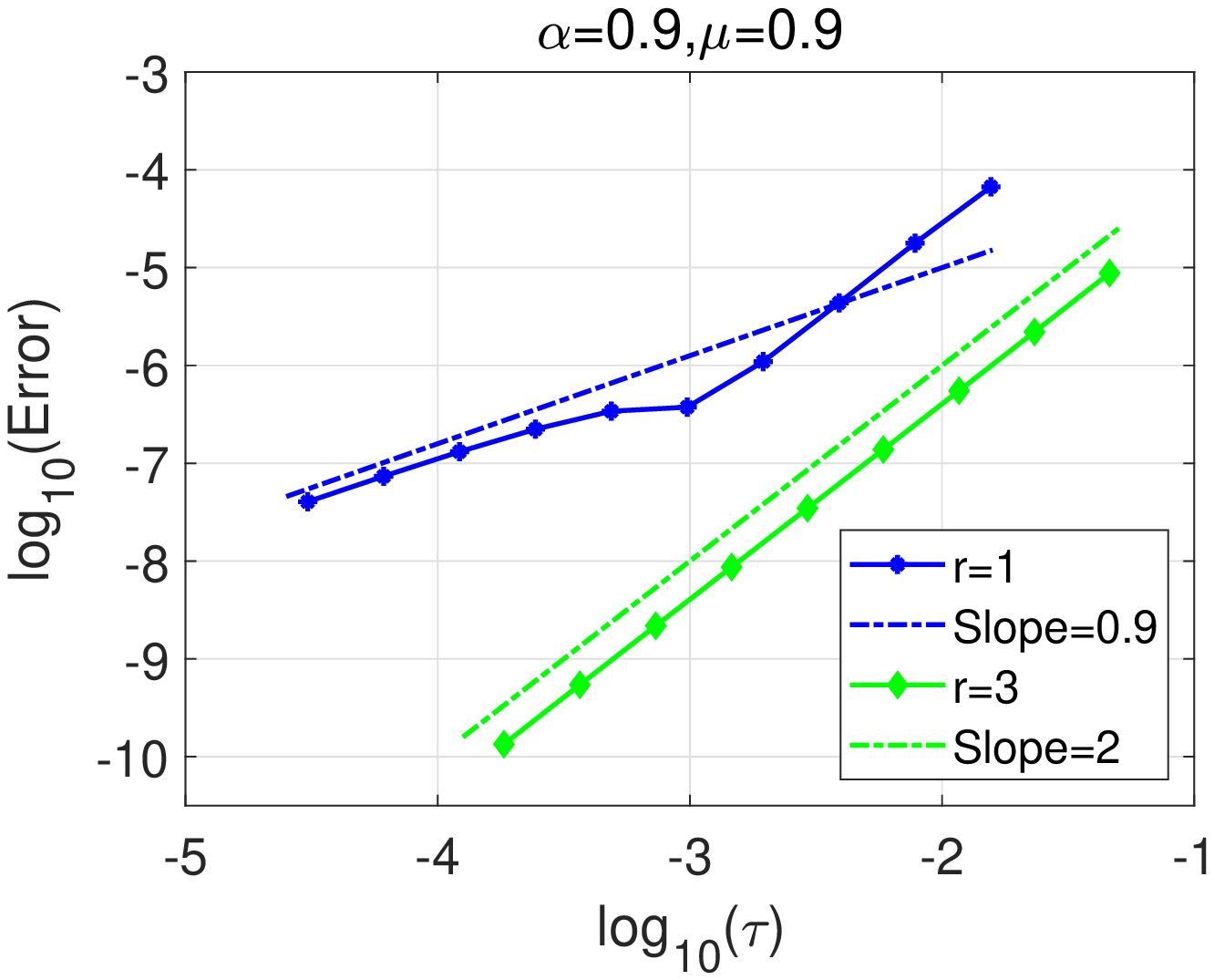}}
\centerline{}
\end{minipage}
\caption{(Example \ref{expl2}) Error history for the L1-CN scheme and L$1^{+}$-CN scheme
for different values of the mesh parameter $r$ and fractional derivative order $\alpha$:
the first line corresponds to L1-CN scheme;
the second line is for L$1^{+}$-CN scheme.
}\label{fig2}
\end{figure*}

\begin{example}\label{expl3}
Consider the fractional Allen-Cahn Neumann problem in the domain $(-1,1)\times(-1,1)$ with
the initial condition $\phi(\x,0)=\cos(4\pi x)\cos(4\pi y)$ without a source term. The exact solution is unavailable.
\end{example}

The spatial discretization uses
the Legendre spectral method with high enough mode number to avoid possible spatial error contamination.
Since the exact solution is unknown, the error of the numerical solution is defined as:
\bex
e(\tau)=\dps\max_{1\leq n\leq M}\|\phi^{n}_{M}-\phi^{2n}_{2M}\|_{\infty}
\eex
It is known that the time fractional operator creates some kind of singularity at the initial time, and the solution behaves like
$t^{\alpha}$ with respect to the time variable; see, e.g., \cite{DYZ19} for a numerical confirmation of this behavior.
In this example, we will take a closer look at the initial error
to investigate how the initial low regularity affects the accuracy of the computed solutions.
We also study the impact of the mesh ratio $r$ on the convergence order.
The calculation is run up $T=0.01$ using graded mesh with grid points $M$ ranging from
20 to $10\times 2^{10}$.
In Figure \ref{fig3}, we
present the errors in log-log scale with respect to the maximum mesh size, i.e.,
$\Dt=\dps\max_{1\leq n\leq M}(t_{n}-t_{n-1})=t_{M}-t_{M-1}$.
It is observed in the figure that
the L1-CN scheme \eqref{L1_CN_s} and L$1^{+}$-CN scheme \eqref{L1p_CN1} attain the convergence rate $\min\{\alpha r,2-\alpha\}$ and $\min\{\alpha r,2\}$ respectively for all tested $\alpha$ and $r$.
Clearly the optimal $r$ is $\frac{2-\alpha}{\alpha}$ for the scheme \eqref{L1_CN_s} and $\frac{2}{\alpha}$ for the scheme \eqref{L1p_CN1}.
In these cases, both schemes reach the theoretical convergence order, i.e., $2-\alpha$ order and second order respectively.
For the uniform mesh, i.e., $r=1$, the schemes lose the optimal convergence order. This is indicative that
the regularity of the solution is lower than what the theoretical convergence order of the schemes requests.

The stability of the proposed schemes is investigated through running the calculation
with $\varepsilon^{2}=0.001$ for long time,
i.e., $T=50$, using a time step size as large as possible.
In view of the singularity feature at the beginning time,
we split the interval $[0,T]$ into two subintervals $[0,1]$ and $(1,T]$.
We compute the solution using the graded mesh with optimal $r$ in
$[0,1]$,
and using the uniform mesh with the time step size {\color{black}${\triangle t}$} in $(1,T]$.
The computed modified energies and original energy
are presented in Figure \ref{fig4}.
The modified energies shown in the left figures exhibits dissipative features during the running time even
if a large time step size {\color{black}${\triangle t} =1$} is used.
This demonstrates good stability and modified energy dissipation properties
of the schemes proved in Theorems \ref{th2} and \ref{th3}.
However, as shown in the right figures, the original energy is dissipative
only for relatively small time step size. Precisely,
the original energy fails to keep dissipation during some time period for the solution computed with
{\color{black}${\triangle t}=1$}.
It is noteworthy that this failure is not due to the instability of the schemes,
but due to possible large error caused by
the use of the large time step size.
Another notable fact is that all the original energy curves coincide with each other
for the time step sizes ranging from $0.0001$ to $0.1$.
In fact the computed original energy is
a key indication of the efficiency of the numerical methods for phase field models.
The observed dissipation feature of the original energy signifies that the solution evolves in the right way,
which is important for long time simulation.
The last point we want to emphasize is that although the time step size {\color{black}${\triangle t}=1$} was not able to
produce correct original energy dissipation,  use of larger {\color{black}${\triangle t}$} is still possible through an adaptive strategy.
That is, adaptively utilize larger {\color{black}${\triangle t}$} during the time period when the phase transition is slow.
Doing so will benefit the most from the unconditional stability of the proposed schemes.

\begin{figure*}[htbp]
\begin{minipage}[t]{0.49\linewidth}
\centerline{\includegraphics[scale=0.55]{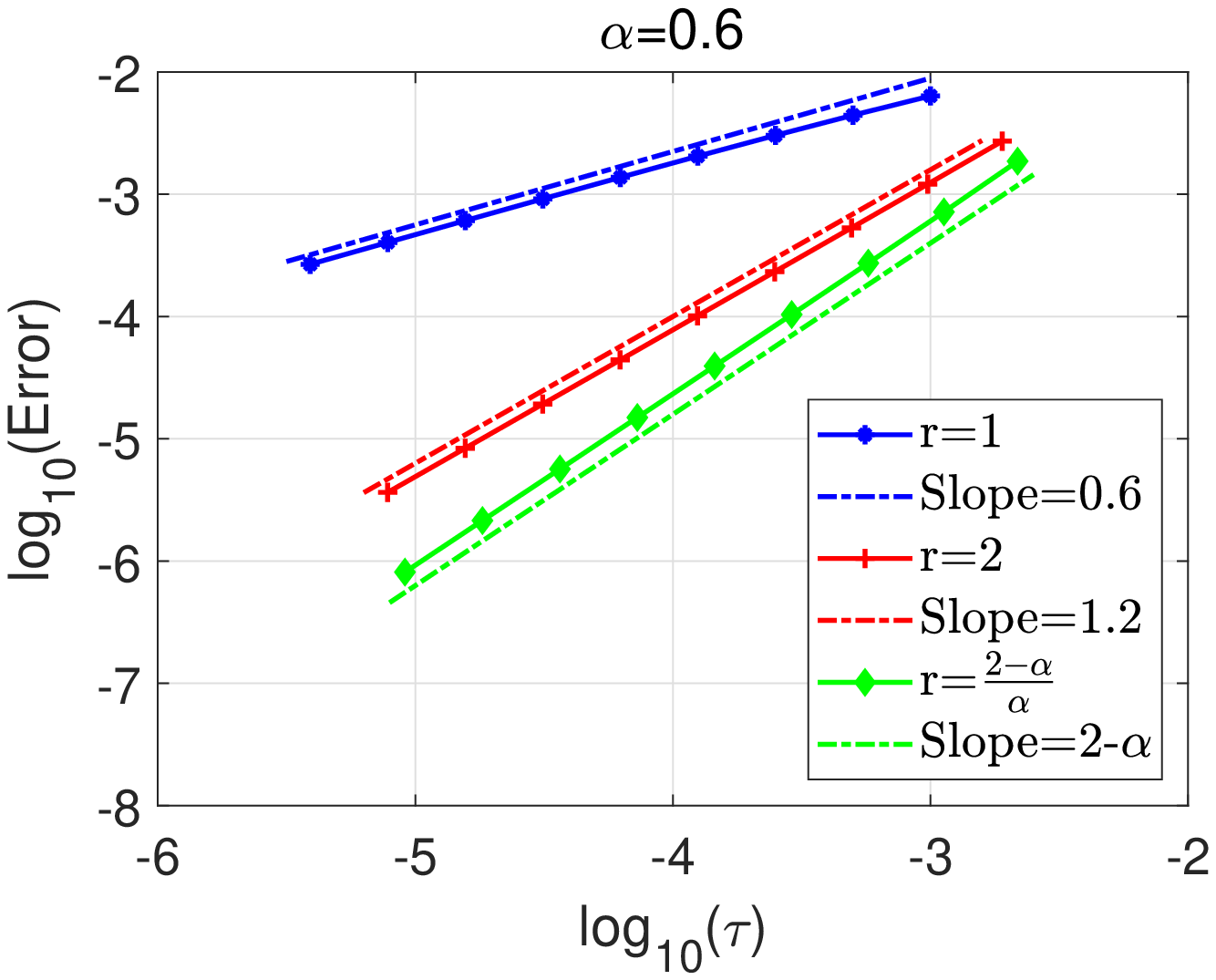}}
\centerline{\hspace{9cm}(a) L1-CN scheme}
\end{minipage}
\begin{minipage}[t]{0.49\linewidth}
\centerline{\includegraphics[scale=0.55]{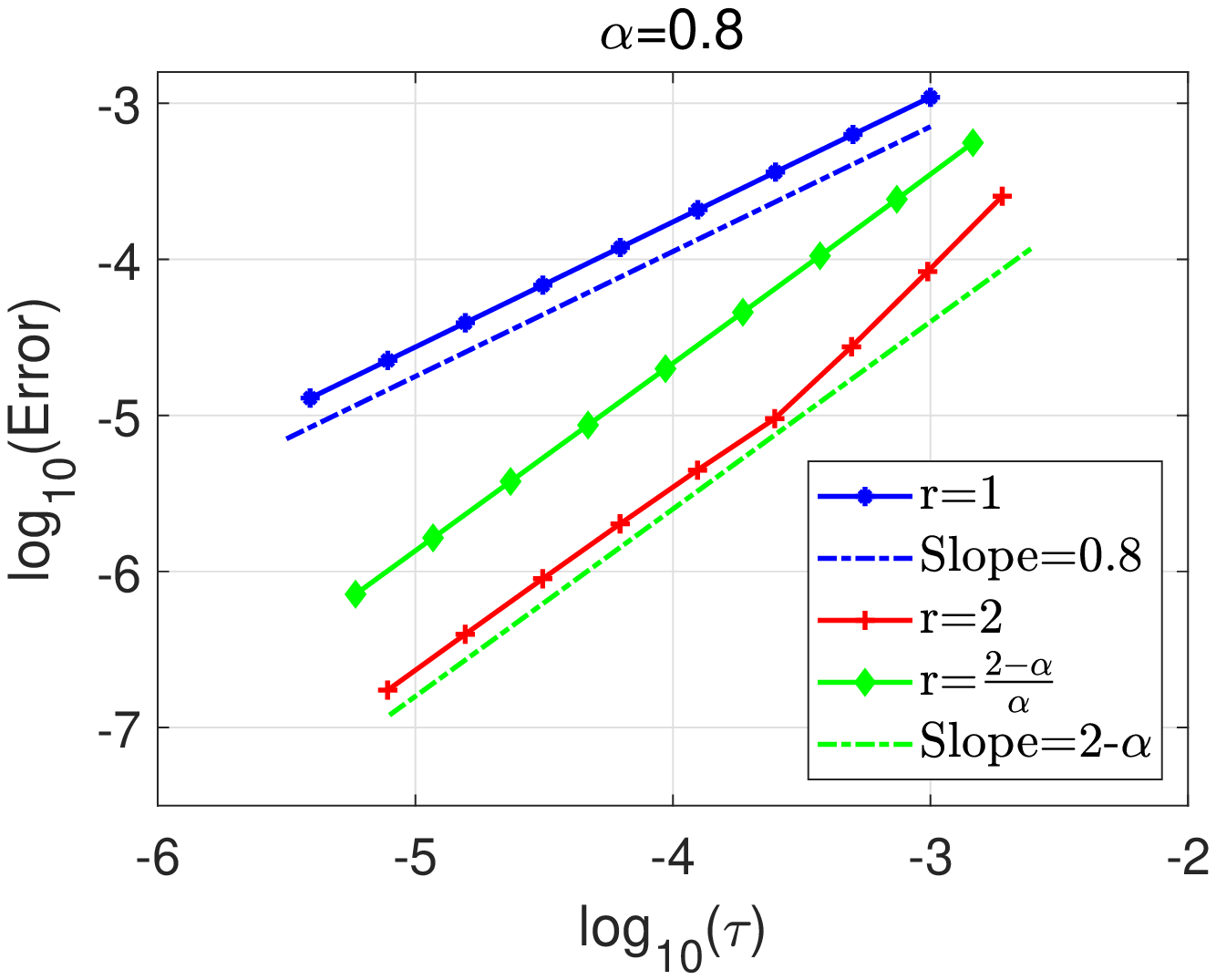}}
\centerline{}
\end{minipage}
\vskip 3mm
\begin{minipage}[t]{0.49\linewidth}
\centerline{\includegraphics[scale=0.55]{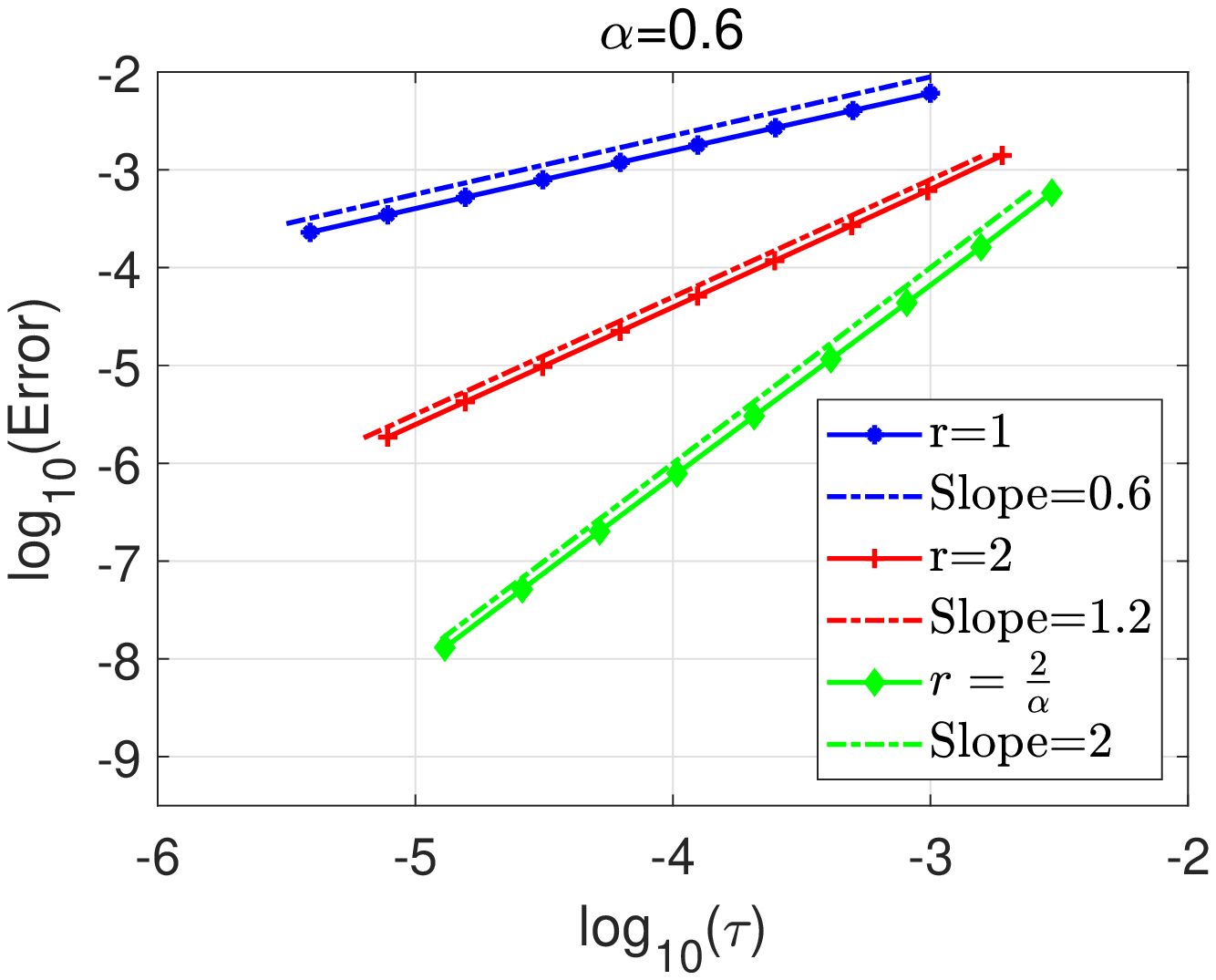}}
\centerline{\hspace{9cm}(b) L$1^{+}$-CN scheme}
\end{minipage}
\begin{minipage}[t]{0.49\linewidth}
\centerline{\includegraphics[scale=0.55]{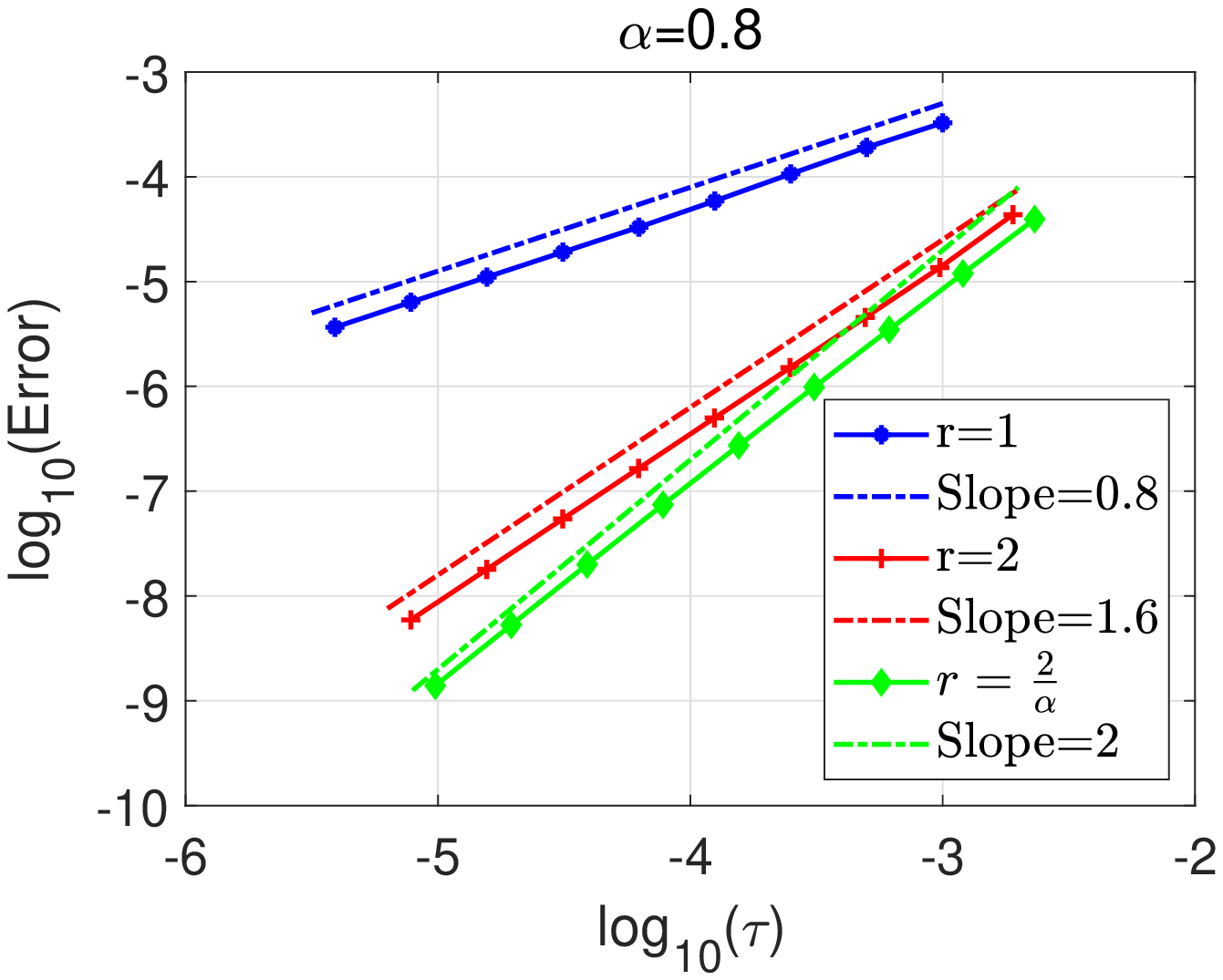}}
\centerline{}
\end{minipage}
\caption{(Example \ref{expl3}) Effect of the graded mesh parameter $r$ on the convergence rate of the L1-CN scheme (the first line)
and  L$1^{+}$-CN scheme (the second line) with $\varepsilon^{2}=0.01$.
}\label{fig3}
\end{figure*}

\begin{figure*}[htbp]
\begin{minipage}[t]{0.48\linewidth}
\centerline{\includegraphics[scale=0.5]{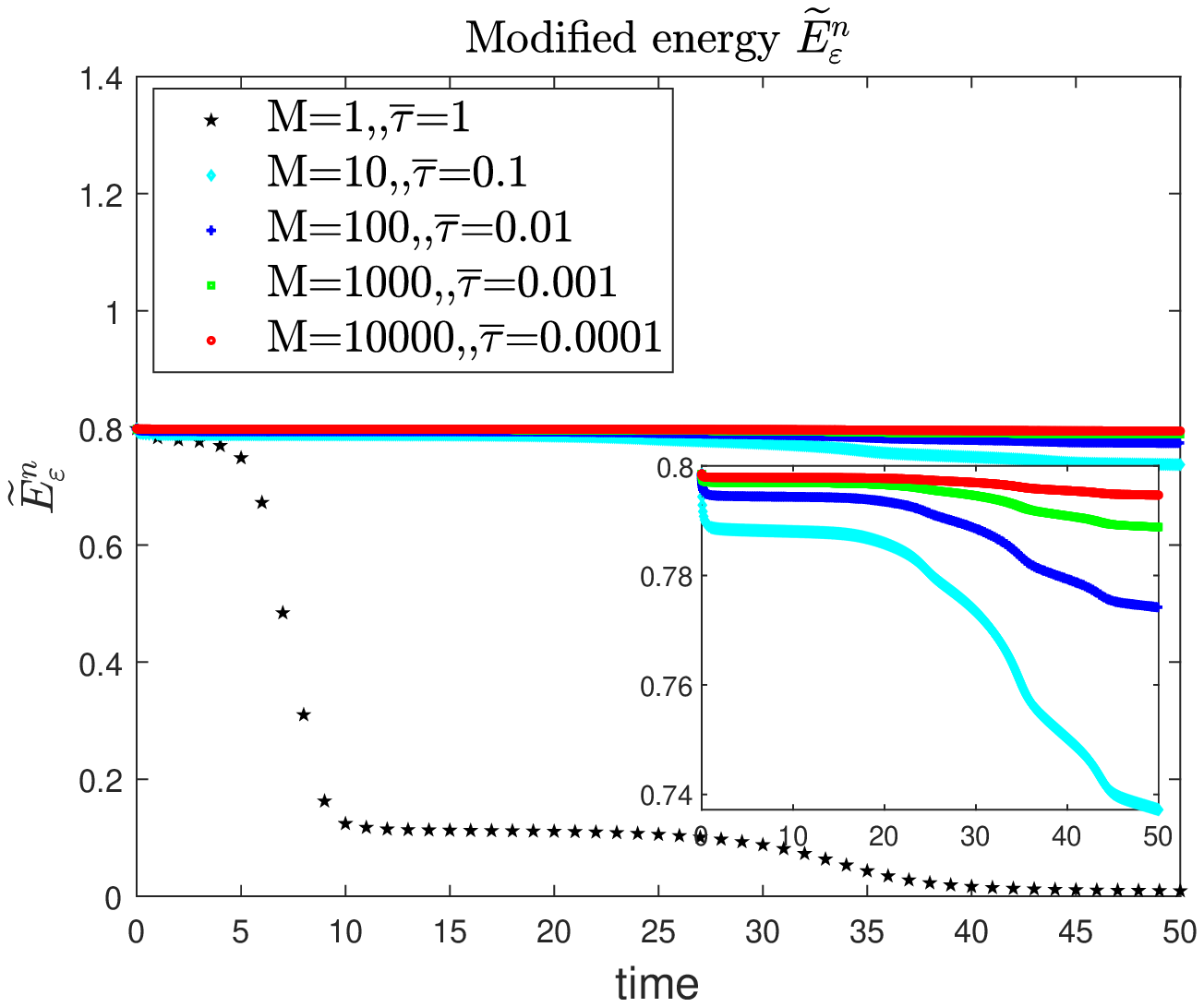}}
\centerline{}
\end{minipage}
\begin{minipage}[t]{0.48\linewidth}
\centerline{\includegraphics[scale=0.5]{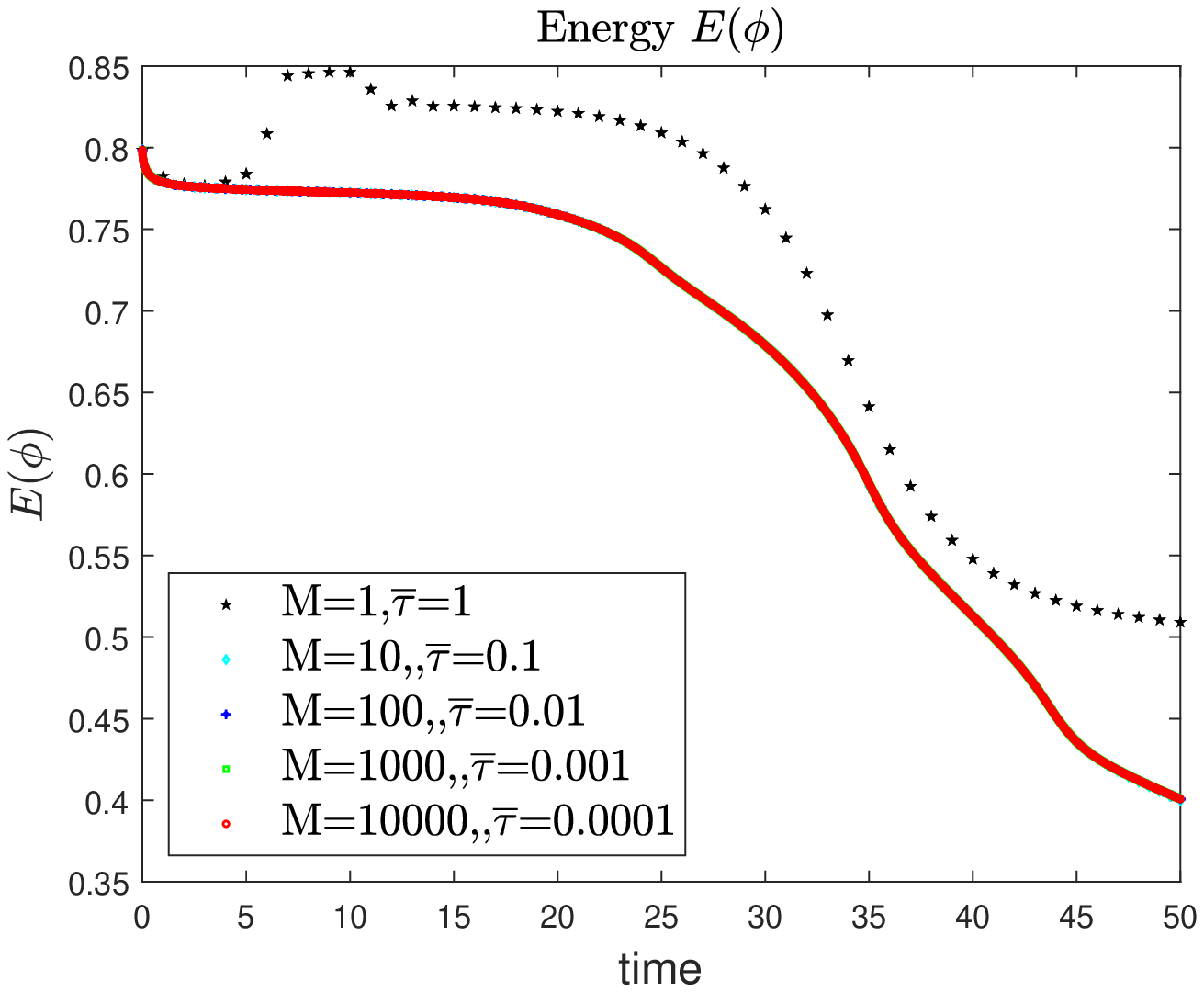}}
\end{minipage}
\vskip -3mm
\centerline{(a) L1-CN scheme with $\alpha=0.6$}
\vskip 6mm
\begin{minipage}[t]{0.48\linewidth}
\centerline{\includegraphics[scale=0.5]{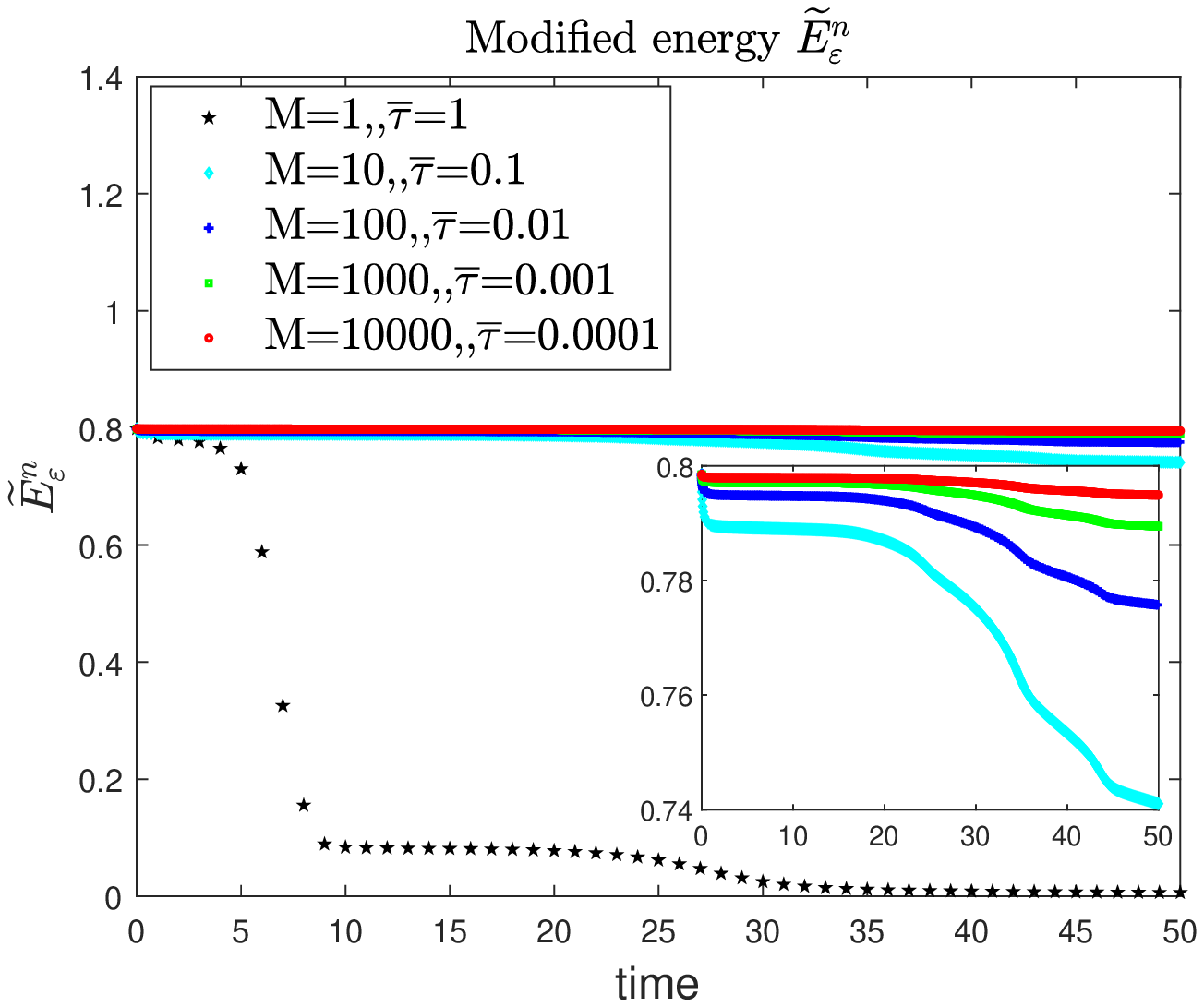}}
\centerline{}
\end{minipage}
\begin{minipage}[t]{0.48\linewidth}
\centerline{\includegraphics[scale=0.5]{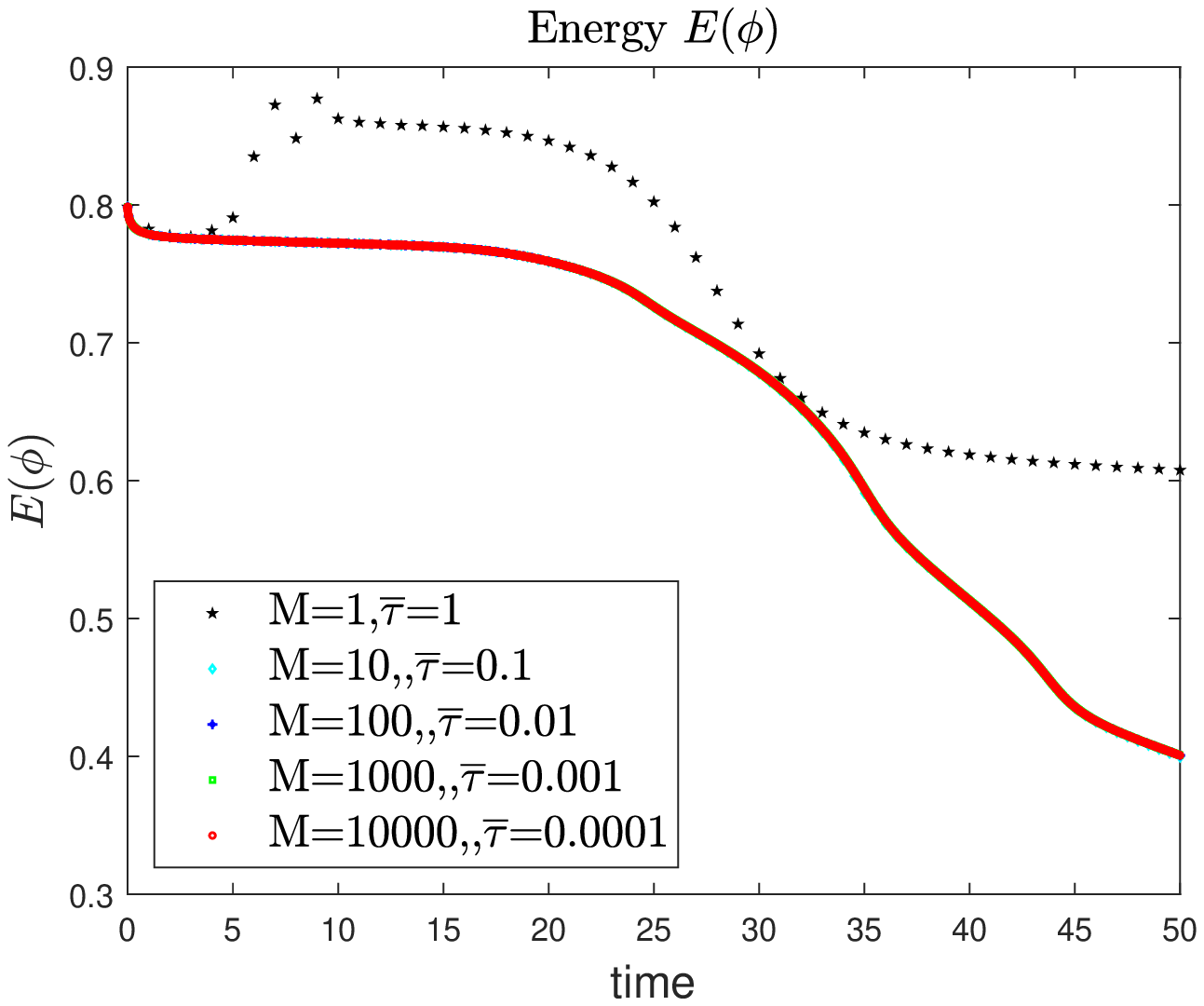}}
\end{minipage}
\vskip -3mm
\centerline{(b) L$1^+$-CN scheme with $\alpha=0.6$}
\caption{(Example \ref{expl3}) Evolution in time of the modified energy and original energy
computed by using different time step sizes.
}\label{fig4}
\end{figure*}

\subsection{Order sensibility of a benchmark problem}
In this test case, we intend to apply the proposed schemes to the interface moving problem
governed by the fractional Allen-Cahn equation in the domain
$(-32,32)\times(-32,32)$.
The initial state of the interface is the circle of the radius
$R_{0}=8$. It is known that the circle interface will shrink and eventually disappear due to the driving force.
This problem was also studied in \cite{HZX20} to verify the performance of the scheme proposed in that paper.
In the classical case, i.e., $\alpha=1$,
it was shown \cite{Yan16,Chen98} that
the radius $R(t)$ of the circle at the given time $t$ evolves as
$R(t)=\sqrt{R_{0}^{2}-2t}$.
That is, the ratio is monotonously decreasing and vanishes at $T=32$.
This problem has been frequently served as a benchmark to test the efficiency
of the numerical methods.



The problem is discretized by
the L1-CN time stepping scheme using the graded mesh with $r=\frac{2-\alpha}{\alpha}$ in the
subinterval $[0,1]$ and the uniform mesh in the subinterval $(1,T]$.
The Legendre spectral method in space uses polynomials of degree $128$ in each direction.
For comparison purposes the computation is performed with the same meshes for the same fractional orders $\alpha$ as in \cite{HZX20}.
The computed interface evolution is shown in Figure \ref{fig5}.
The computed $R^{2}$ and the free energy $E(\phi^{n})$ versus the time are also plotted
in Figure \ref{fig6}.
The interface movement shown in Figure \ref{fig5} is almost the same compared to
the results reported in \cite{HZX20}.
The agreement on $R^{2}$ and $E(\phi)$
between the current scheme and the one in \cite{HZX20} can be observed equally.
This demonstrates the efficiency of the both methods
proposed in \cite{HZX20} and in the present paper.
However, as we have already emphasized, the novelty of present work is the rigorous proof
of an energy dissipation law, both in the continuous and discrete cases,
not only in the uniform mesh but also in the graded mesh.

\begin{figure*}[htbp]
\begin{minipage}[t]{0.19\linewidth}
\centerline{\includegraphics[width=3.5cm,height=3.5cm]{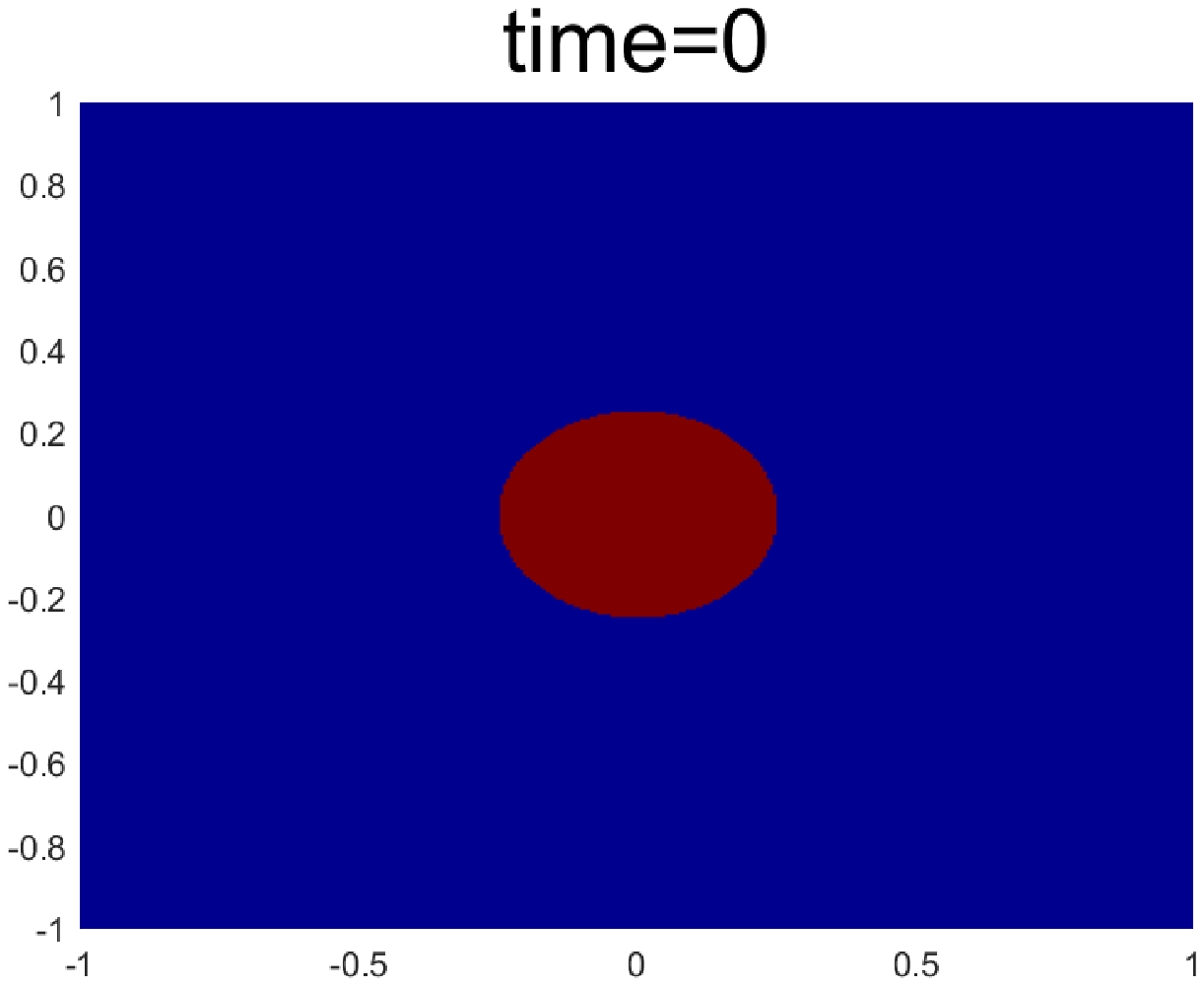}}
\centerline{}
\end{minipage}
\begin{minipage}[t]{0.19\linewidth}
\centerline{\includegraphics[width=3.5cm,height=3.5cm]{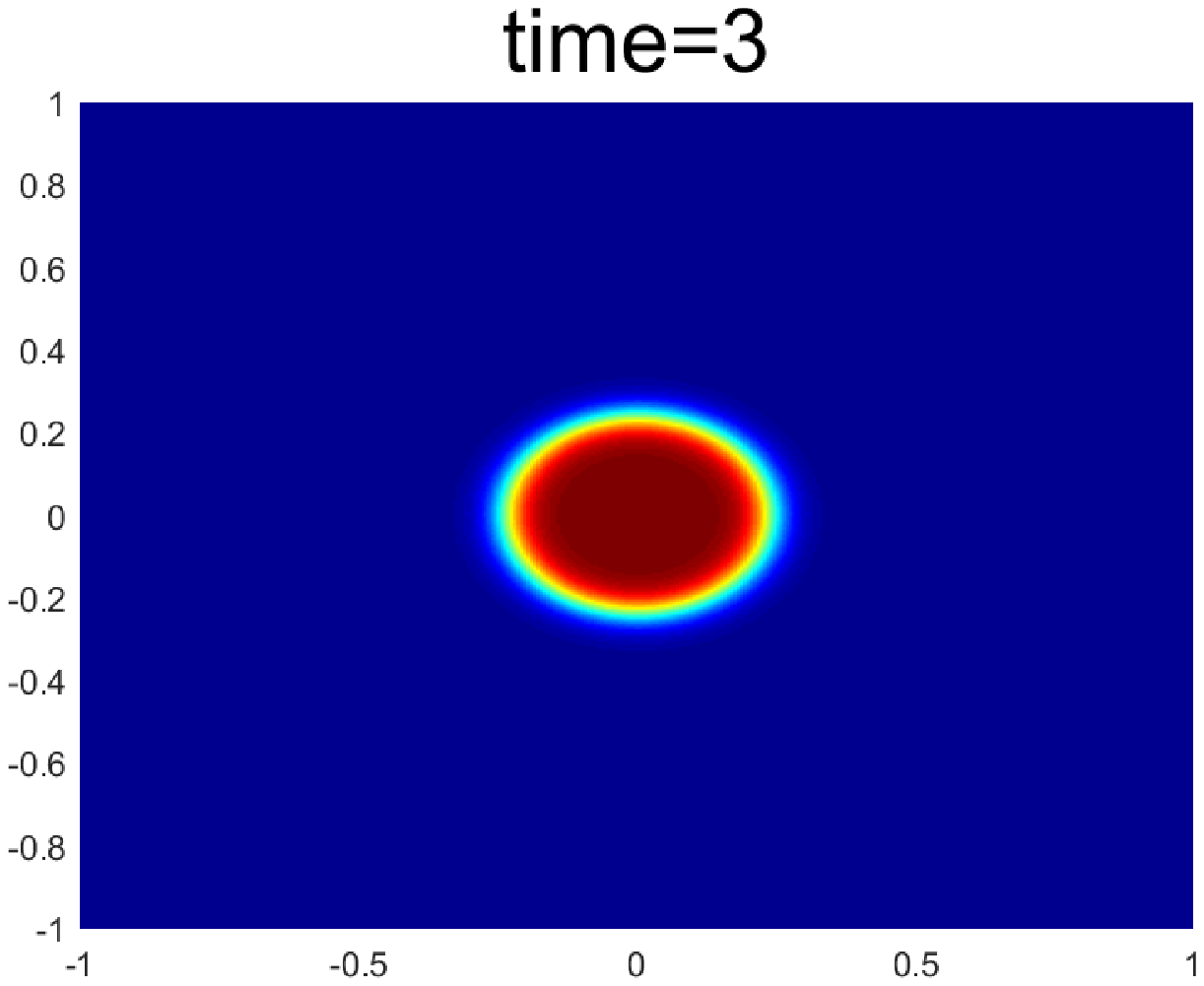}}
\centerline{}
\end{minipage}
\begin{minipage}[t]{0.19\linewidth}
\centerline{\includegraphics[width=3.5cm,height=3.5cm]{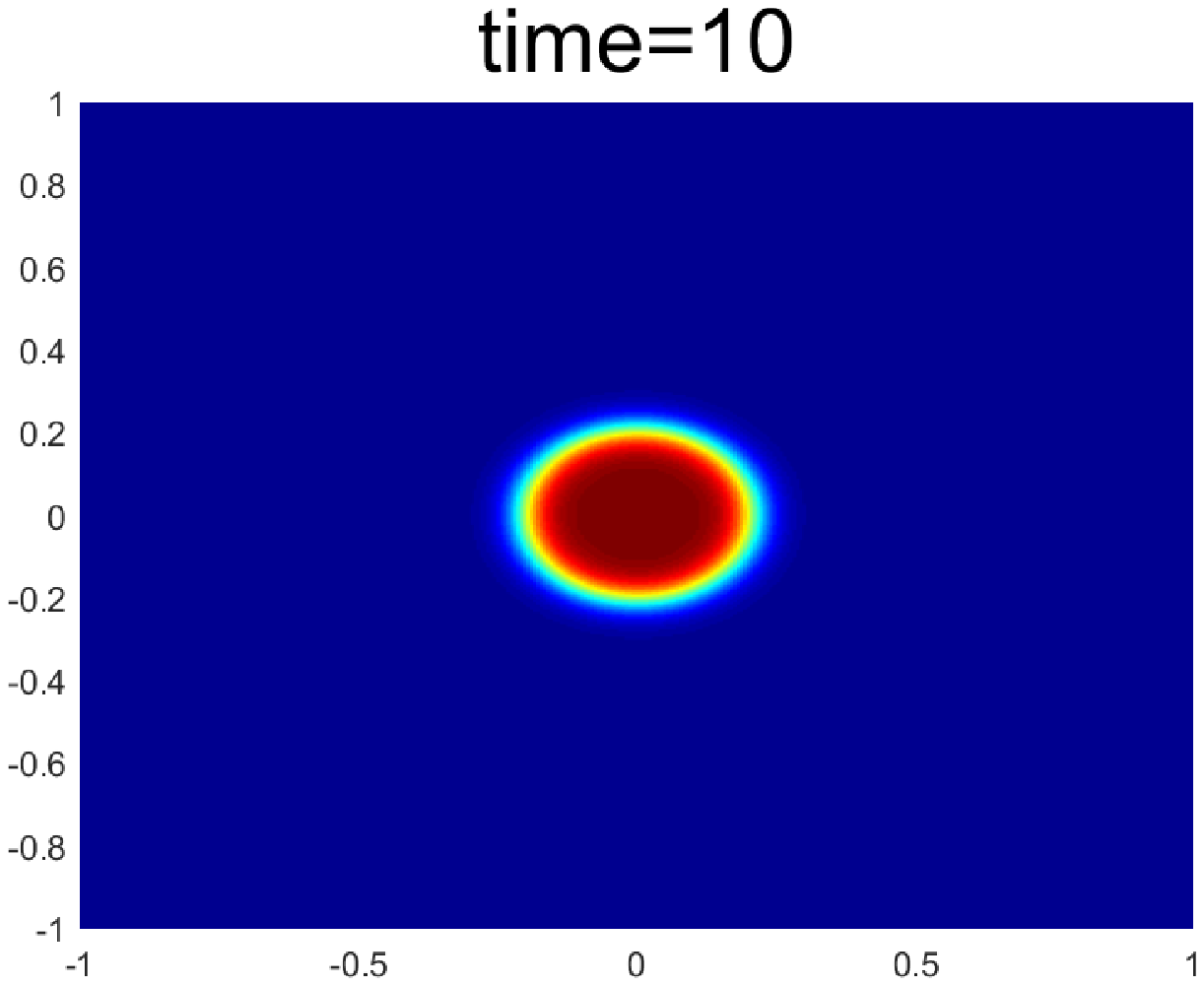}}
\centerline{ (a) $\alpha=1$ and ${\triangle t}=0.01$}
\end{minipage}
\begin{minipage}[t]{0.19\linewidth}
\centerline{\includegraphics[width=3.5cm,height=3.5cm]{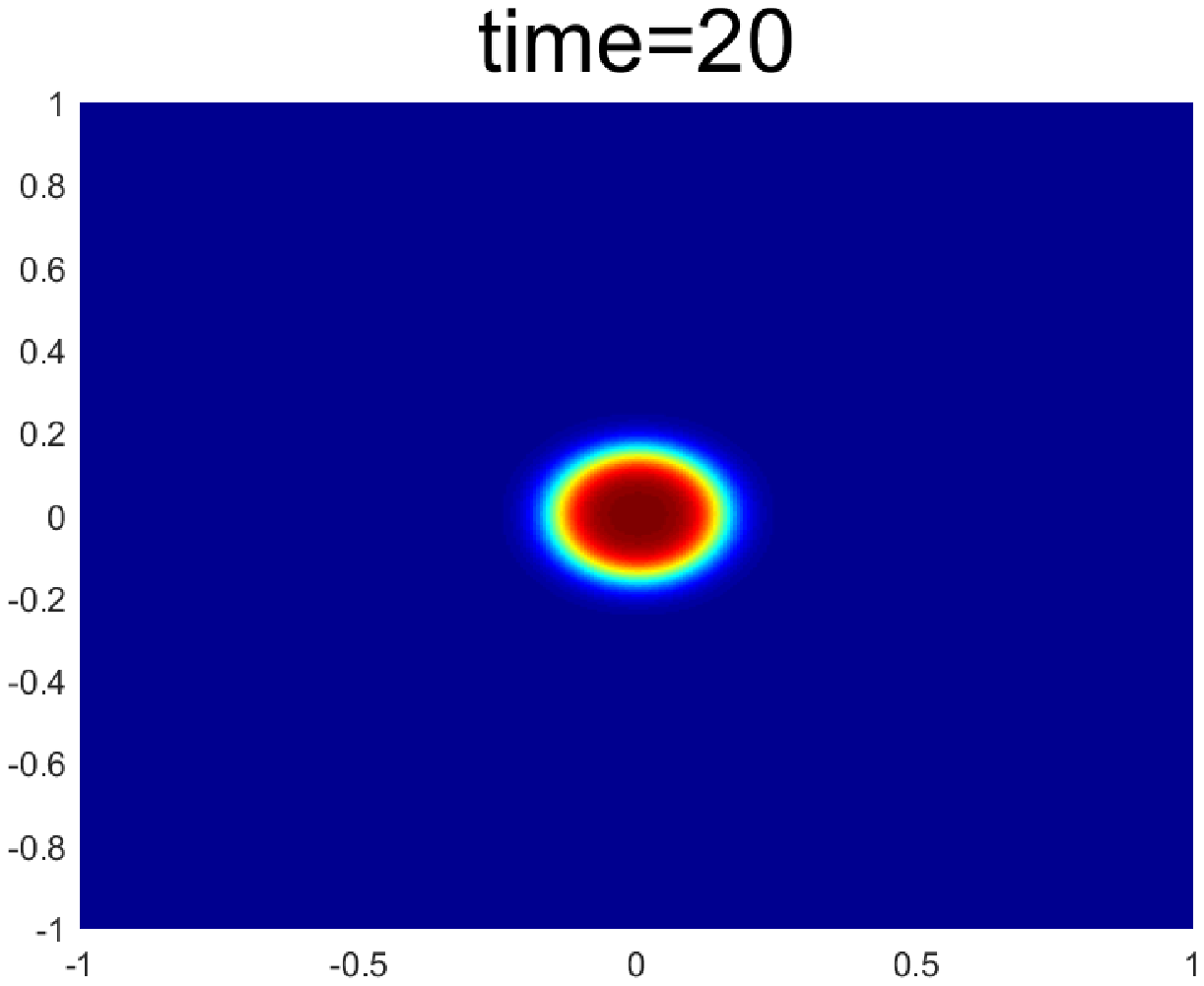}}
\centerline{}
\end{minipage}
\begin{minipage}[t]{0.19\linewidth}
\centerline{\includegraphics[width=3.5cm,height=3.5cm]{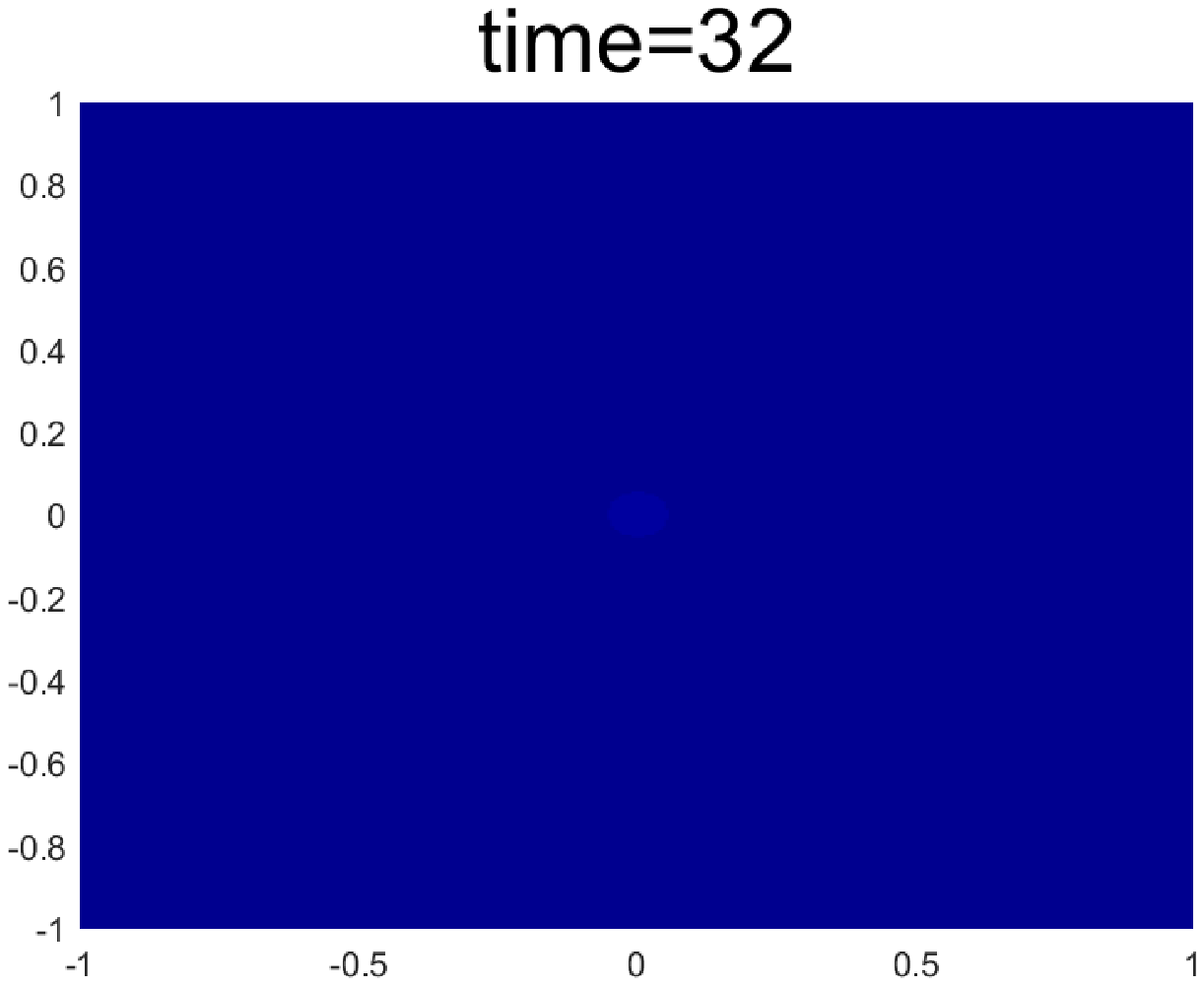}}
\centerline{}
\end{minipage}
\vskip 3mm
\begin{minipage}[t]{0.19\linewidth}
\centerline{\includegraphics[width=3.5cm,height=3.5cm]{FAC4_1.eps}}
\centerline{}
\end{minipage}
\begin{minipage}[t]{0.19\linewidth}
\centerline{\includegraphics[width=3.5cm,height=3.5cm]{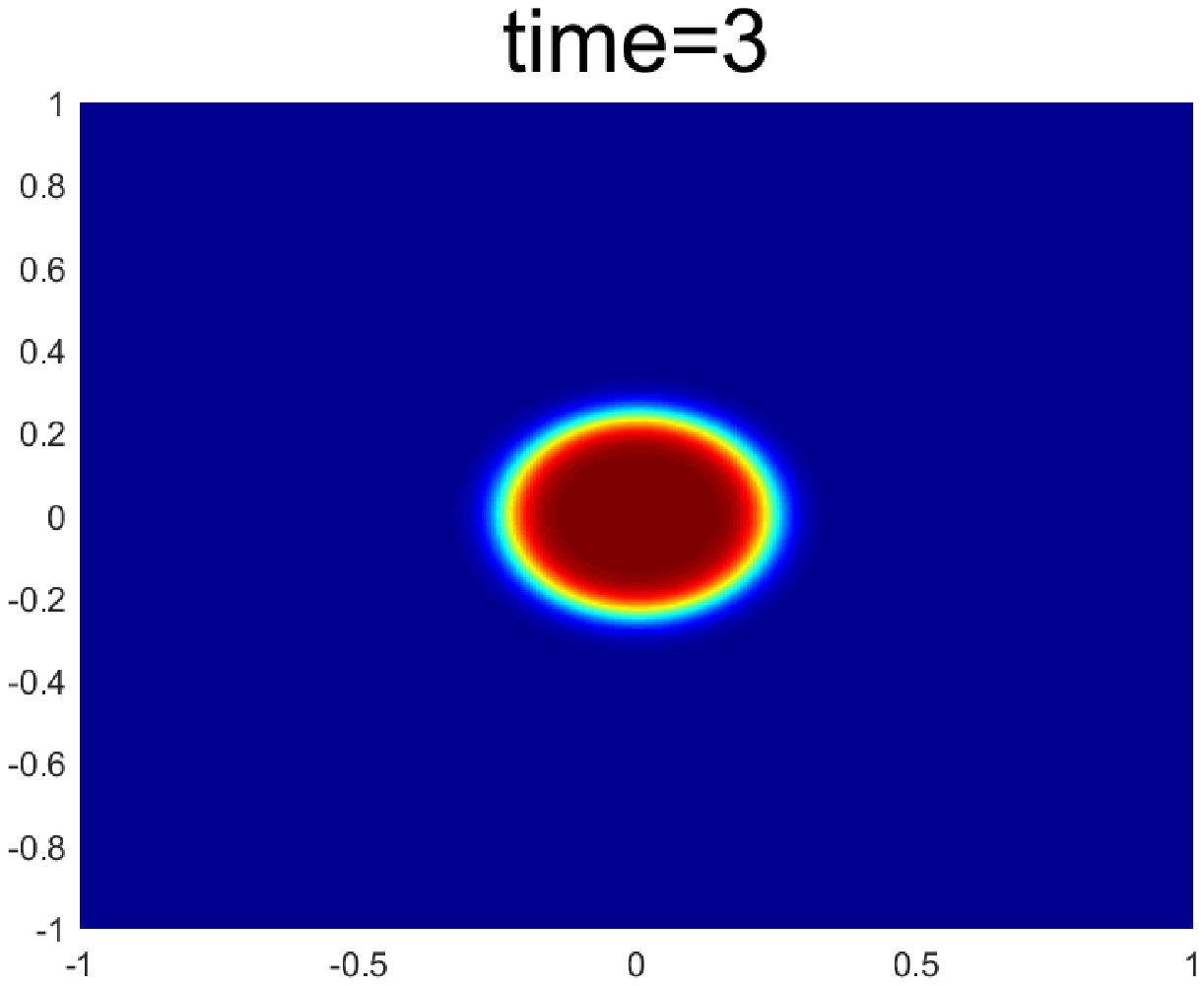}}
\centerline{}
\end{minipage}
\begin{minipage}[t]{0.19\linewidth}
\centerline{\includegraphics[width=3.5cm,height=3.5cm]{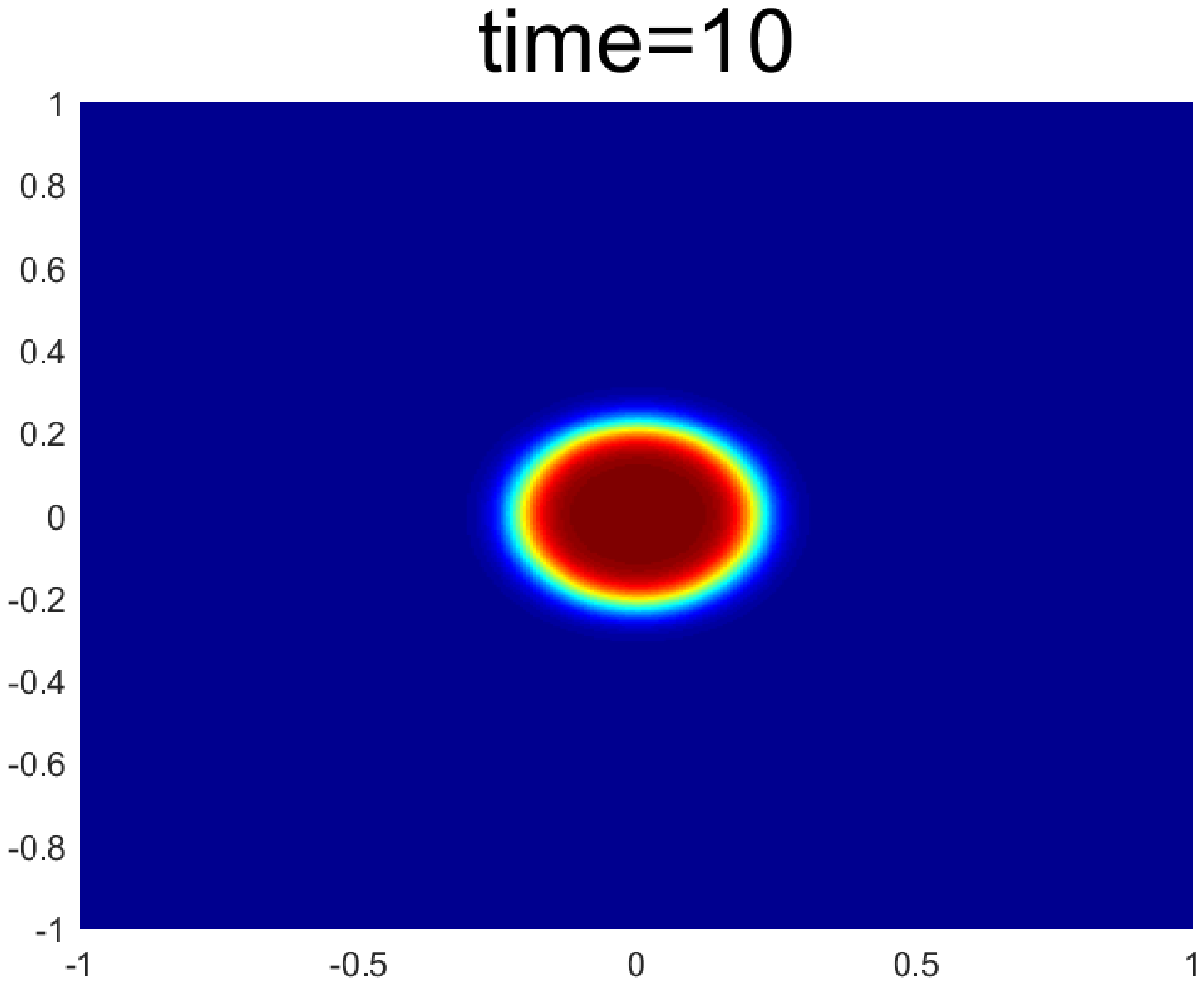}}
\centerline{ (b) $\alpha=0.9, M=100$, and ${\triangle t}=0.01$}
\end{minipage}
\begin{minipage}[t]{0.19\linewidth}
\centerline{\includegraphics[width=3.5cm,height=3.5cm]{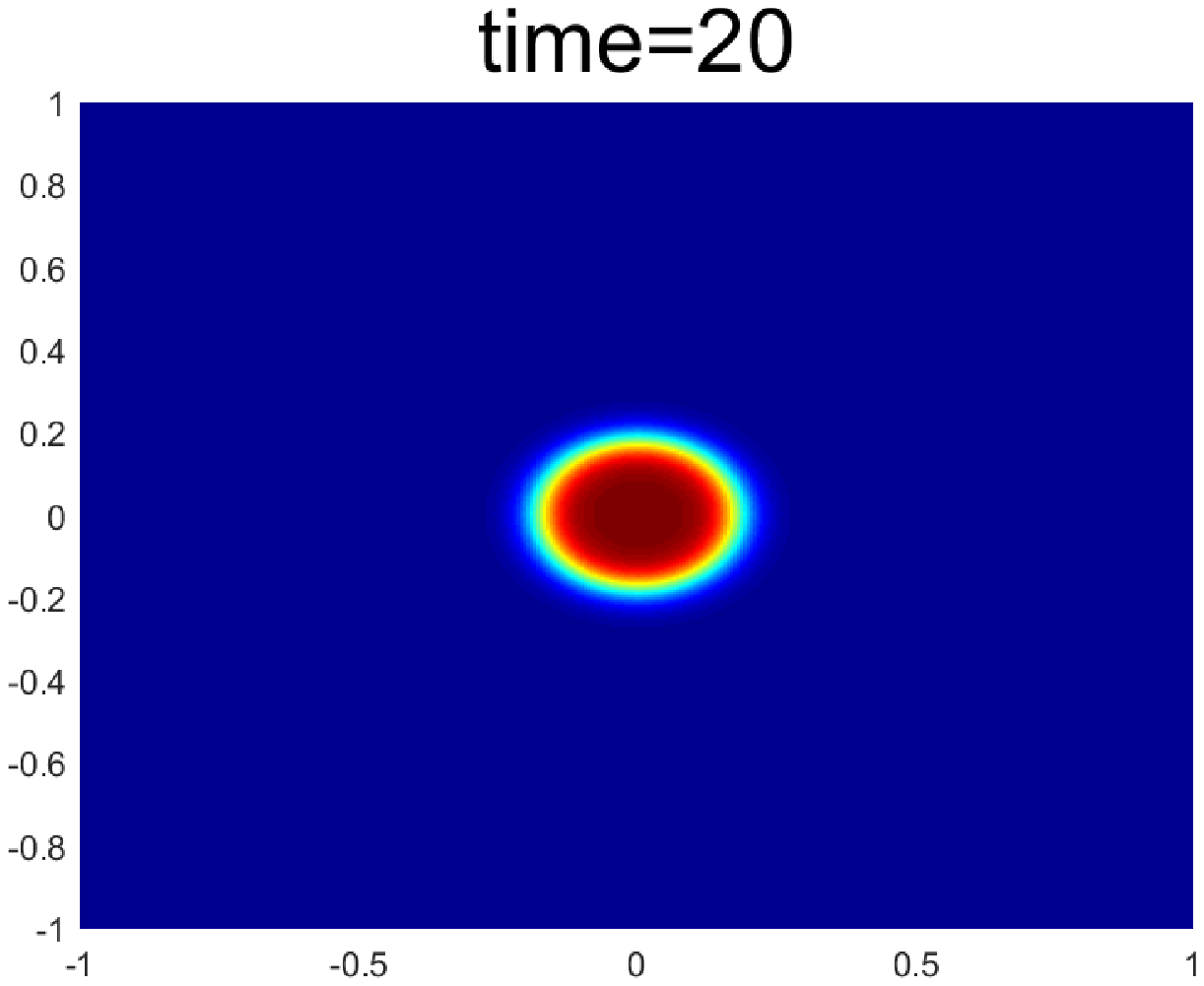}}
\centerline{}
\end{minipage}
\begin{minipage}[t]{0.19\linewidth}
\centerline{\includegraphics[width=3.5cm,height=3.5cm]{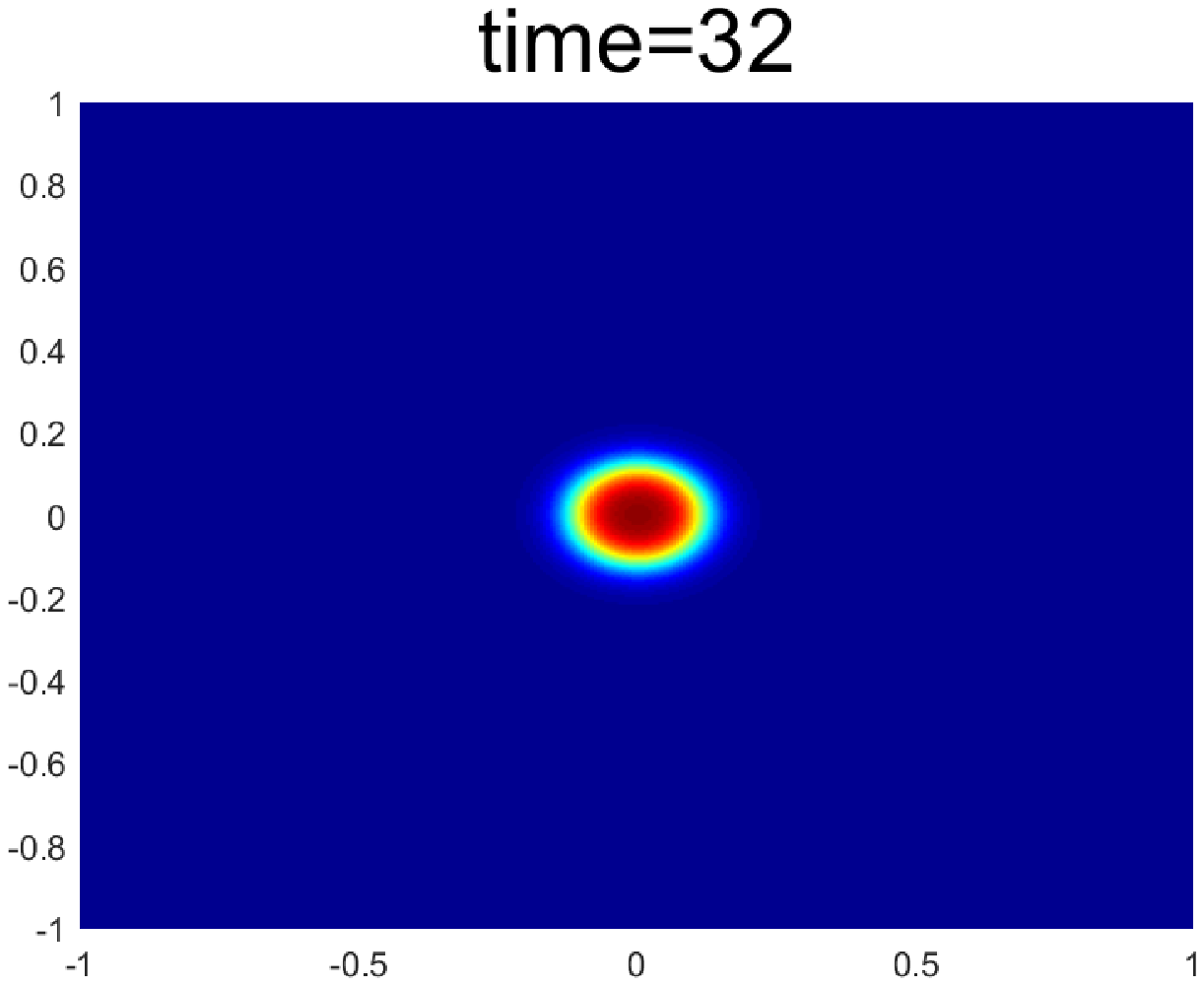}}
\centerline{}
\end{minipage}
\vskip 3mm
\begin{minipage}[t]{0.19\linewidth}
\centerline{\includegraphics[width=3.5cm,height=3.5cm]{FAC4_1.eps}}
\centerline{}
\end{minipage}
\begin{minipage}[t]{0.19\linewidth}
\centerline{\includegraphics[width=3.5cm,height=3.5cm]{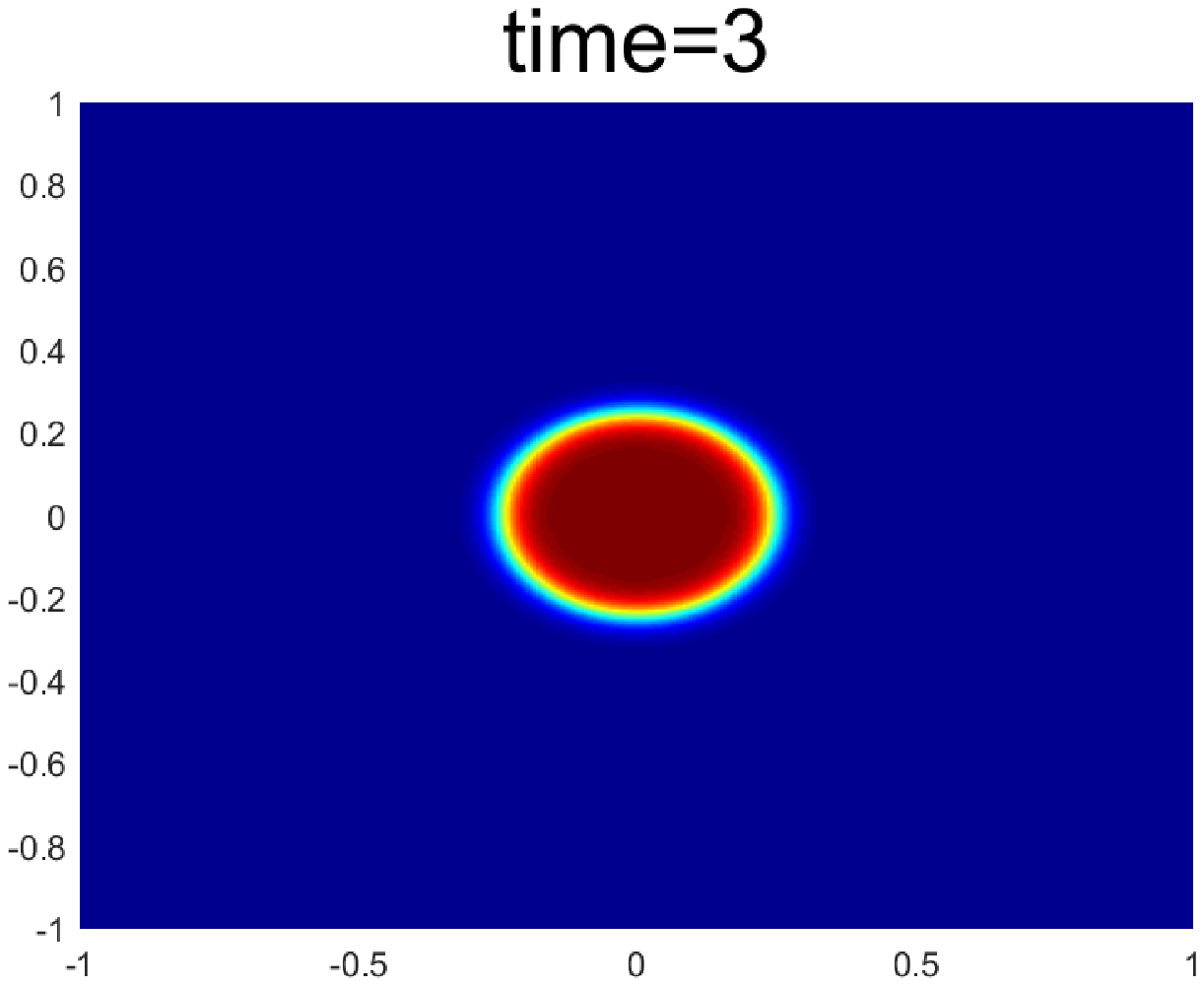}}
\centerline{}
\end{minipage}
\begin{minipage}[t]{0.19\linewidth}
\centerline{\includegraphics[width=3.5cm,height=3.5cm]{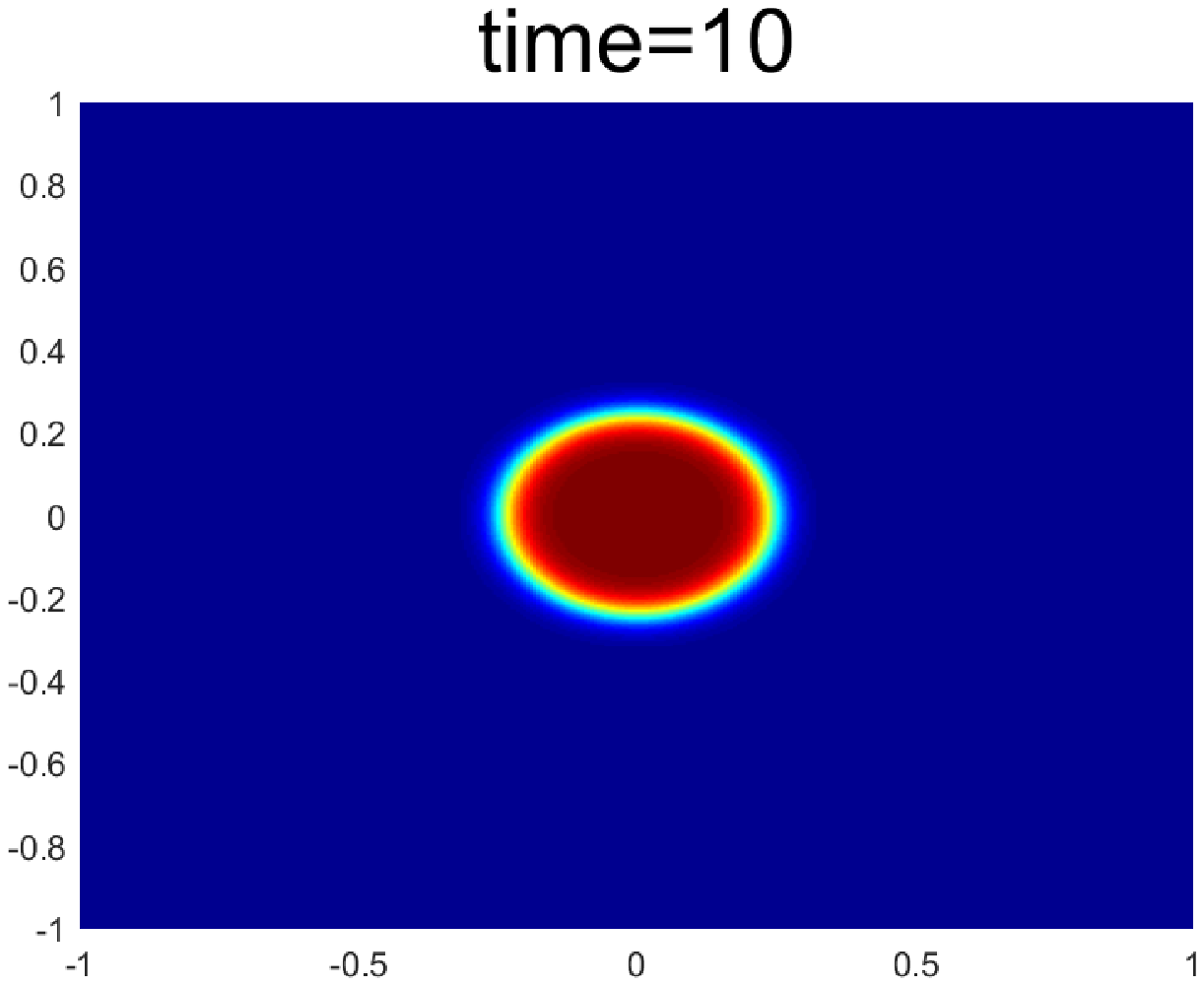}}
\centerline{ (c) $\alpha=0.4, M=1000$, and ${\triangle t}=0.01$}
\end{minipage}
\begin{minipage}[t]{0.19\linewidth}
\centerline{\includegraphics[width=3.5cm,height=3.5cm]{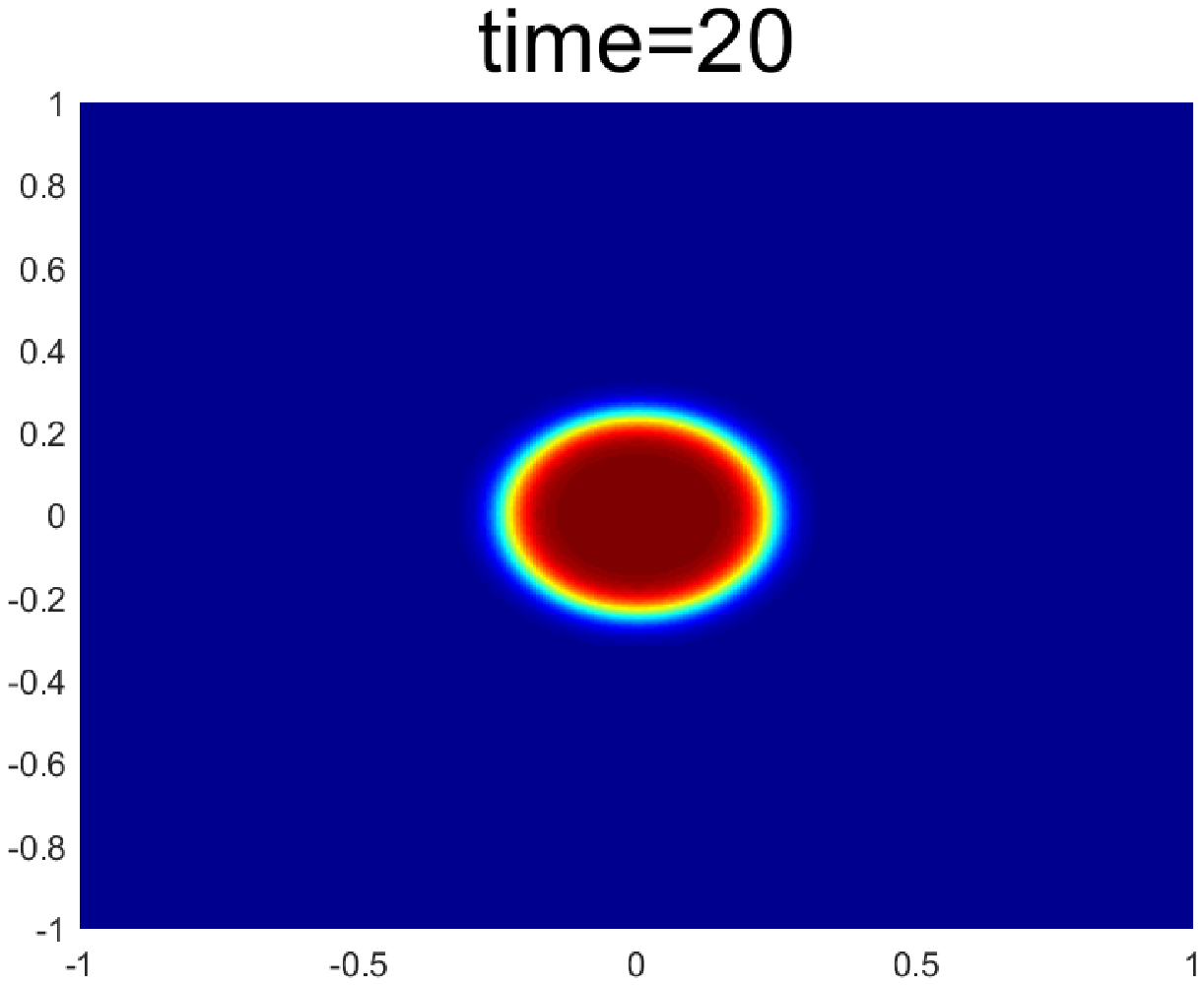}}
\centerline{}
\end{minipage}
\begin{minipage}[t]{0.19\linewidth}
\centerline{\includegraphics[width=3.5cm,height=3.5cm]{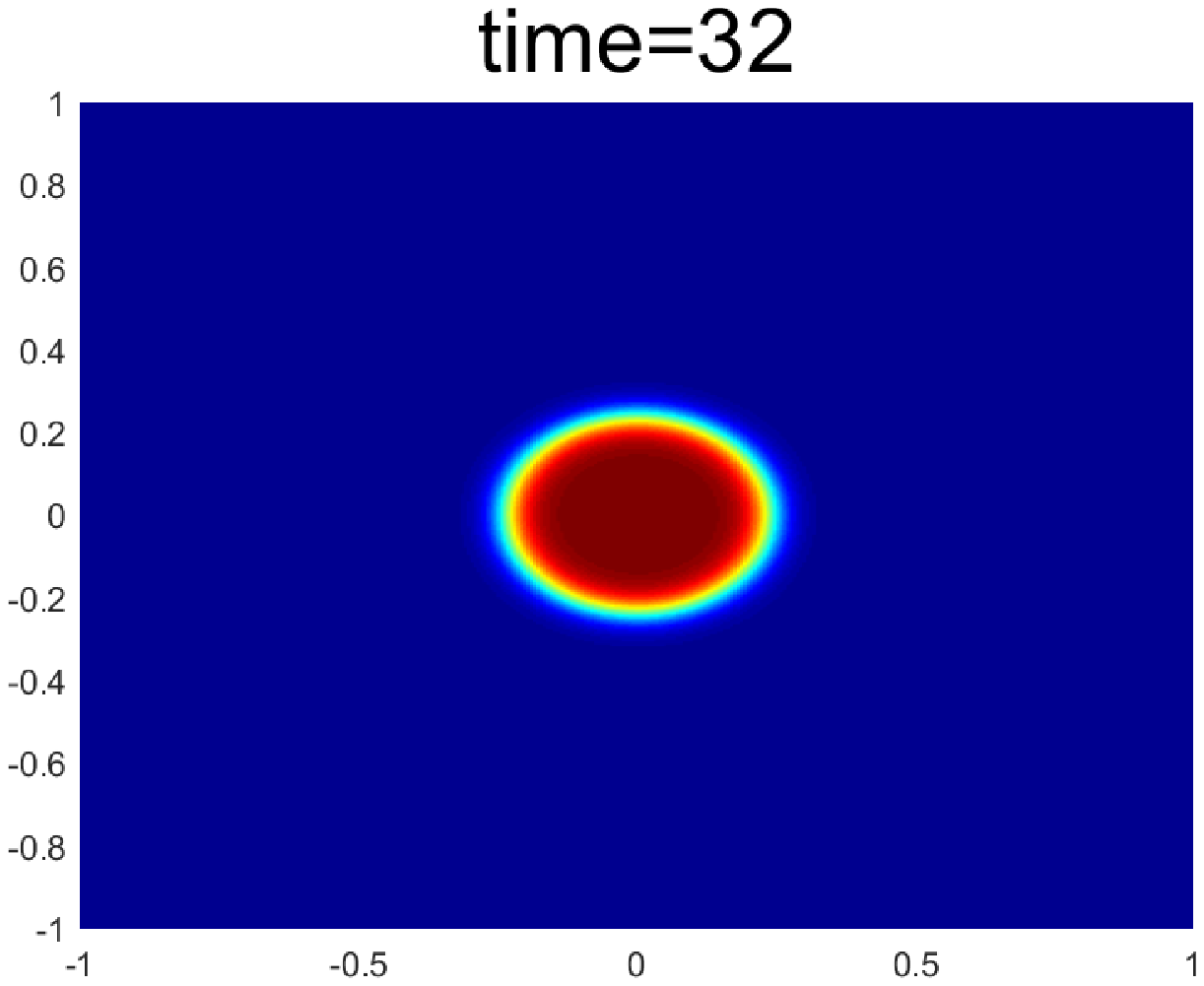}}
\centerline{}
\end{minipage}
\caption{Snapshots of the interface evolution simulated by using the L1-CN scheme with $C_0=1000$ for $\alpha=1, 0.9$, and $0.4$.
}\label{fig5}
\end{figure*}

\begin{figure*}[htbp]
\begin{minipage}[t]{0.49\linewidth}
\centerline{\includegraphics[scale=0.55]{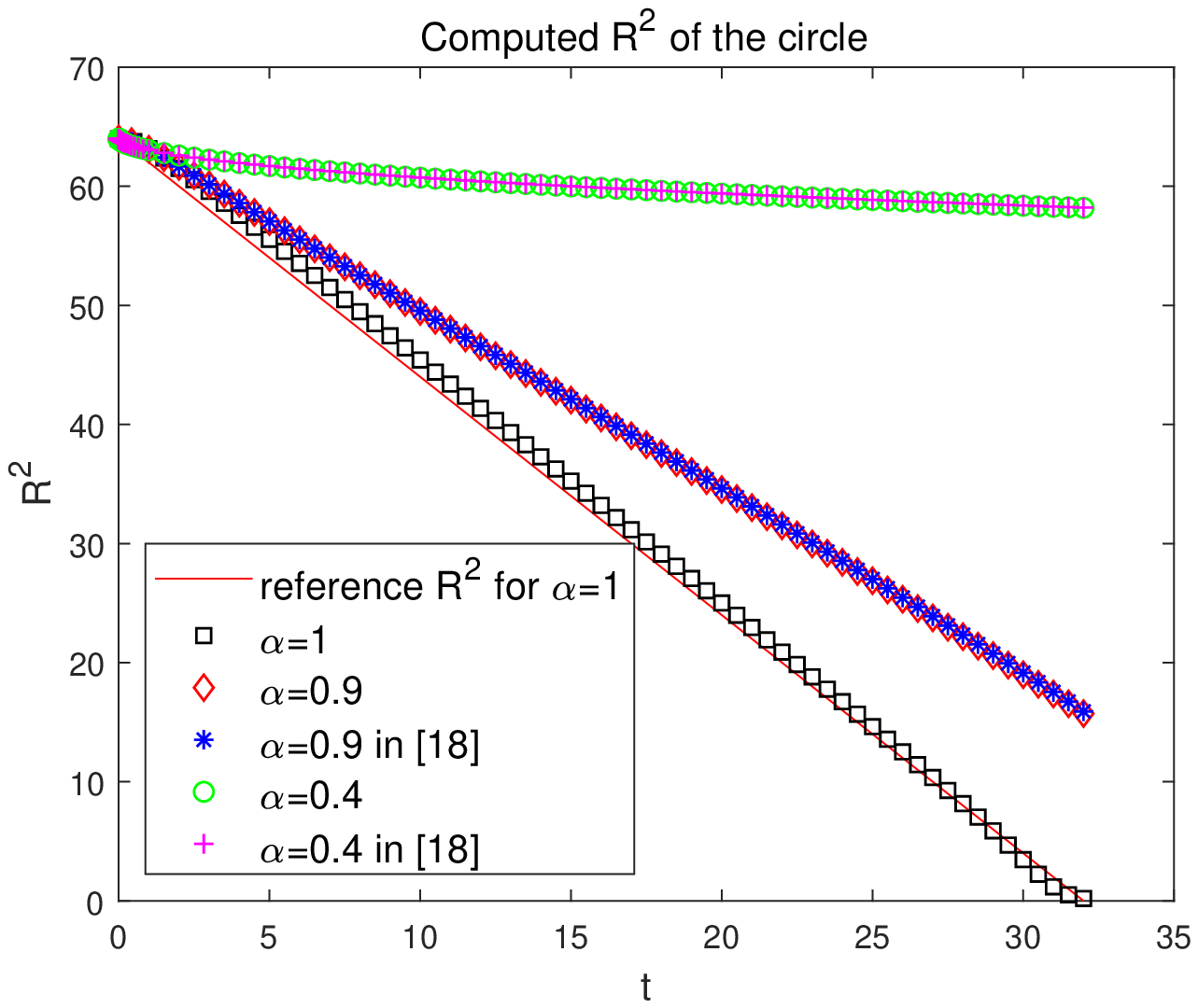}}
\centerline{(a) Interface radius as functions of the time.}
\end{minipage}
\begin{minipage}[t]{0.49\linewidth}
\centerline{\includegraphics[scale=0.55]{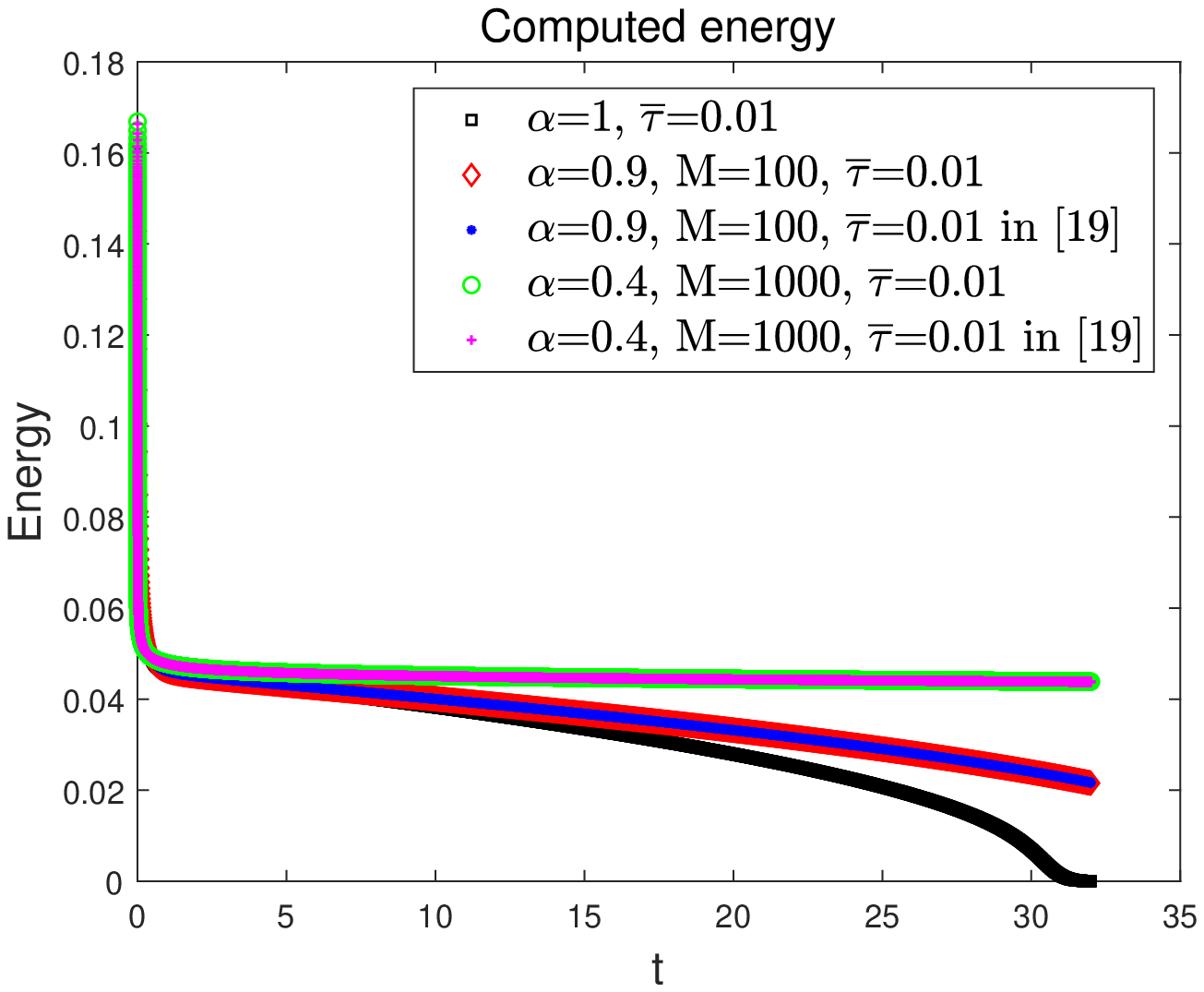}}
\centerline{(b) Free energy dissipation.}
\end{minipage}
\caption{Comparison of the computed $R^{2}$ and total free energy computed between the L1-CN scheme and the scheme in \cite{HZX20}.
}\label{fig6}
\end{figure*}

\subsection{Coarsening dynamics}

Finally we test the L$1^{+}$-CN scheme \eqref{L1p_CN1} to the two-phase coarsening problem,
by solving the fractional Allen-Cahn equation with $\varepsilon^{2}=0.001$ in $(-1,1)^{2}$. The
initial condition is a random data, same as used in \cite{HZX20}.
The simulation is performed in the graded mesh with $r=\frac{2}{\alpha}, M=100$ in
$[0,1]$ and the uniform mesh with {\color{black}time step size ${\triangle t}=0.01$} in $(1,T]$.
The spatial spectral approximation uses $128\times128$ basis functions.
Figure \ref{fig7} presents some snapshots of the simulated phase function and the computed free energy $E(\phi)$ versus time for a number of fractional orders $\alpha=1, 0.9, 0.7$, and $0.5$.
It is observed that the results are not sensitive to the fractional orders at the early stage as
there is no distinguishable difference on the isoline plots among different fractional orders before $t=5$.
After that, the solutions apparently start to deviate, and develop into long time phase coarsening.
Note that similar results and interpretation of these results have been presented
in \cite{HZX20}.
The computed free energy $E(\phi^{n})$ shown in the last row figures
is also in a good agreement with the result reported in \cite{HZX20}.

\begin{figure*}[htbp]
  \begin{center}
    \begin{tabular}{C{0.9cm}cccc}
    &$\alpha=1$&$\alpha=0.9$&$\alpha=0.7$&$\alpha=0.5$\\
\tabincell{c}{$t=0$\\ \\ \\ \\ \\ \\ \\}
&\begin{minipage}[t]{0.2\linewidth}
\includegraphics[scale=0.22]{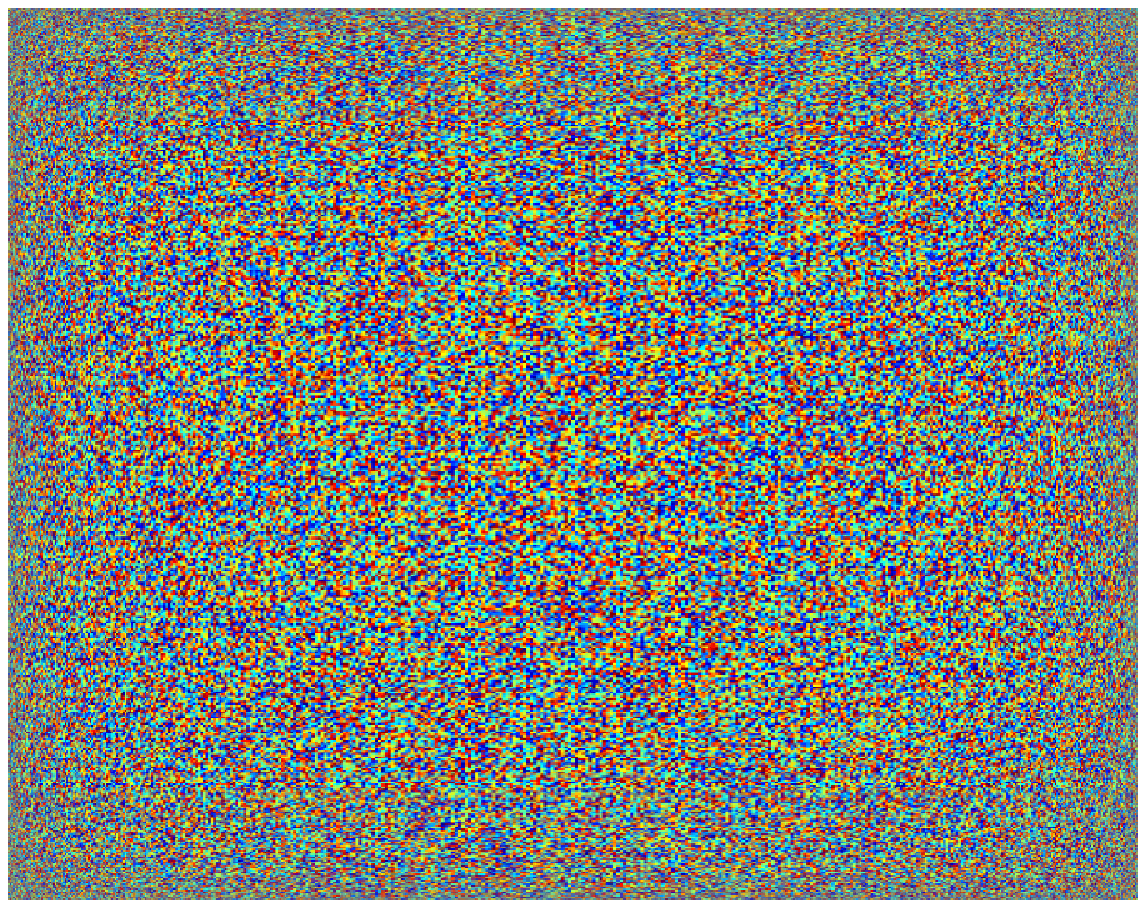}
\end{minipage}&
\begin{minipage}[t]{0.2\linewidth}
\includegraphics[scale=0.22]{FAC5_1.eps}
\end{minipage}&
\begin{minipage}[t]{0.2\linewidth}
\includegraphics[scale=0.22]{FAC5_1.eps}
\end{minipage}
&\begin{minipage}[t]{0.2\linewidth}
\includegraphics[scale=0.22]{FAC5_1.eps}
\end{minipage}\\
\specialrule{0em}{-13mm}{.001pt}
\tabincell{c}{$t=5$\\ \\ \\ \\ \\ \\ \\}
&\begin{minipage}[t]{0.2\linewidth}
\includegraphics[scale=0.22]{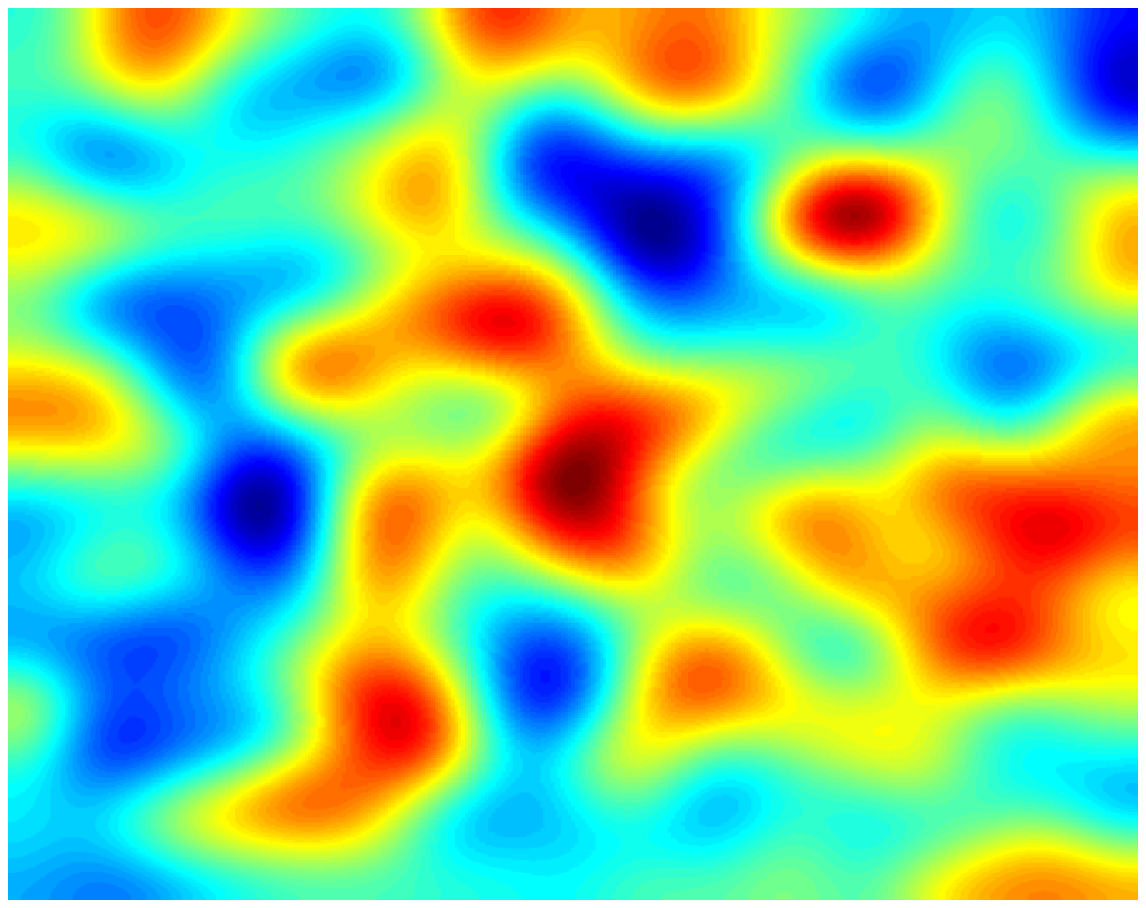}
\end{minipage}&
\begin{minipage}[t]{0.2\linewidth}
\includegraphics[scale=0.22]{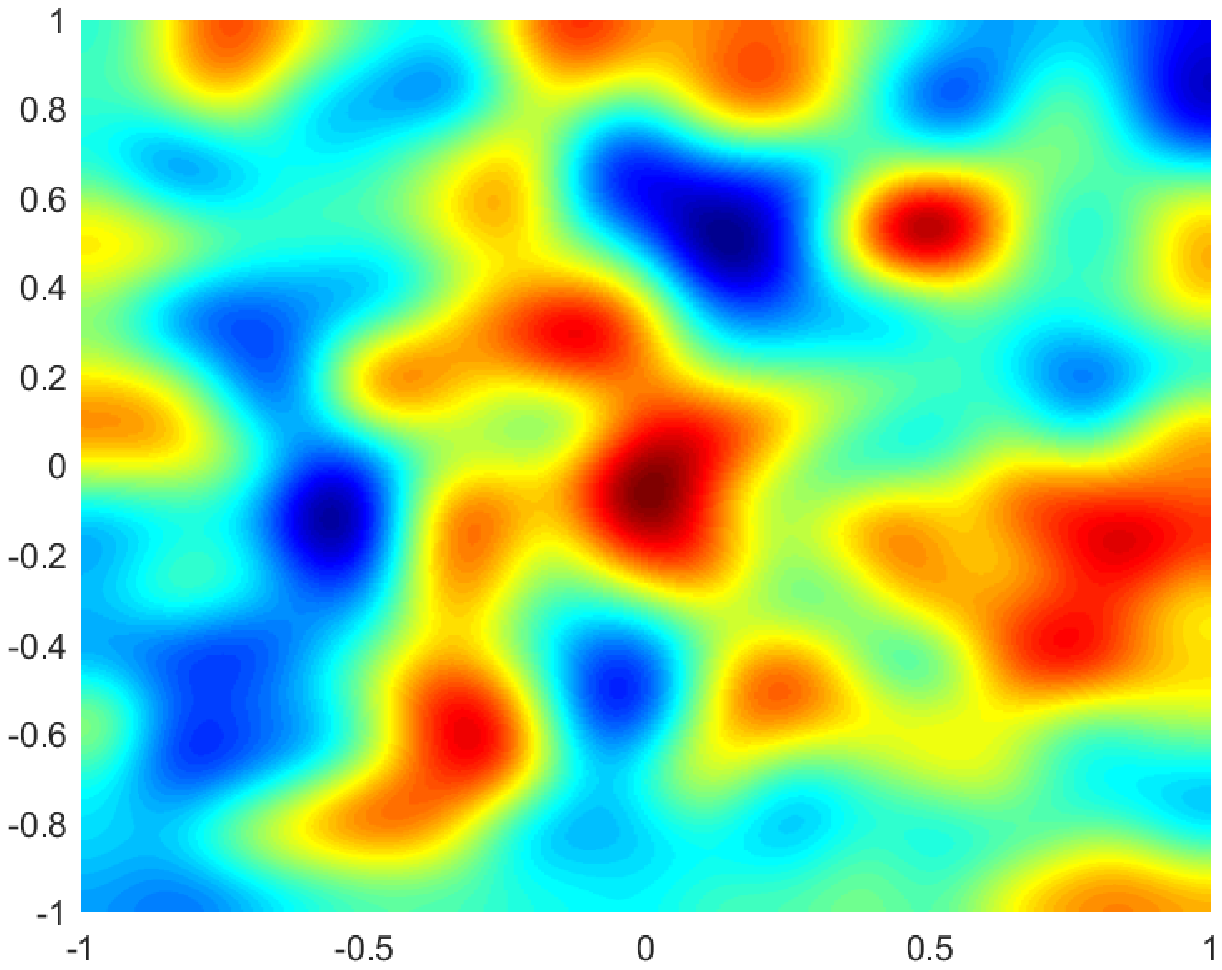}
\end{minipage}&
\begin{minipage}[t]{0.2\linewidth}
\includegraphics[scale=0.22]{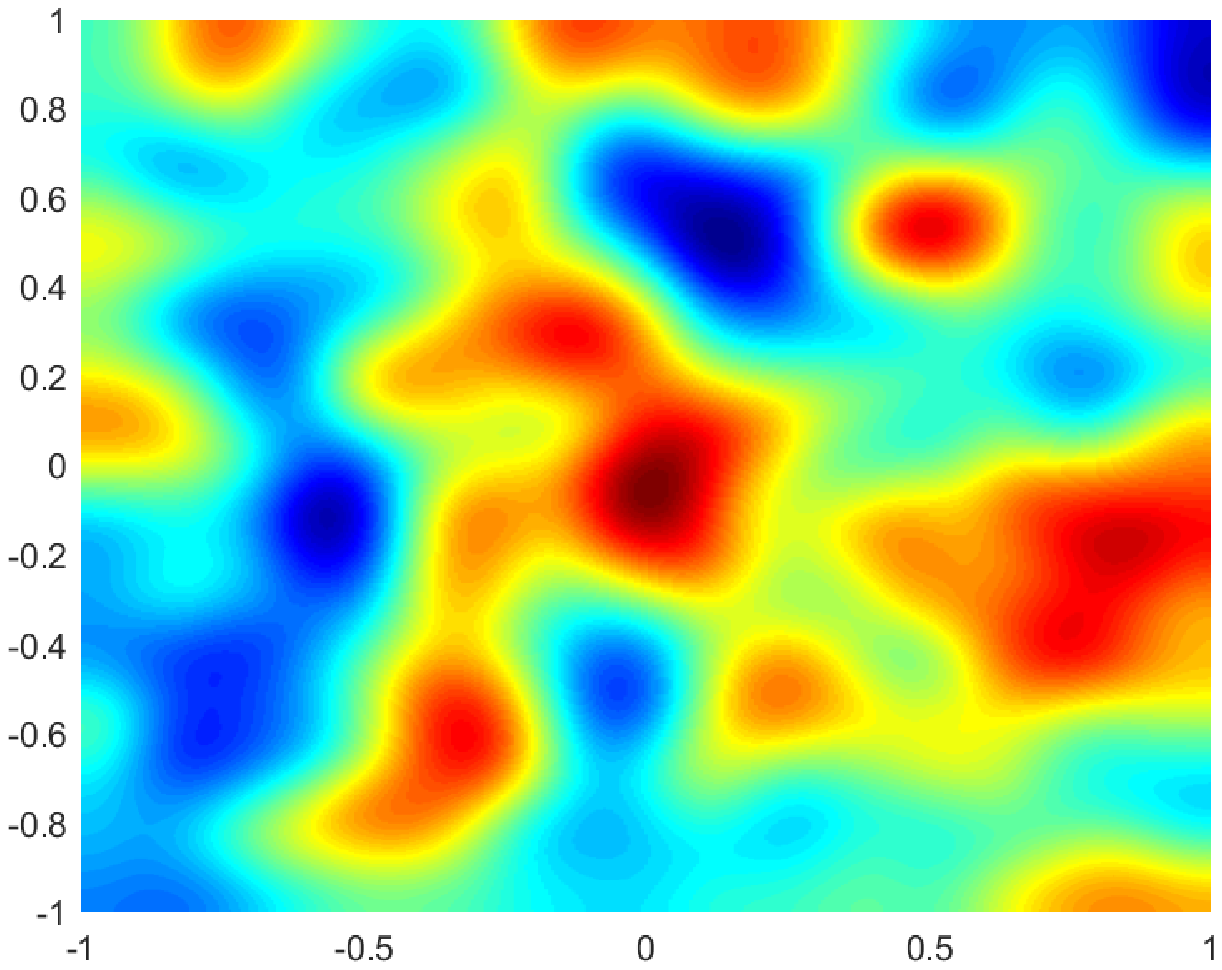}
\end{minipage}
&\begin{minipage}[t]{0.2\linewidth}
\includegraphics[scale=0.22]{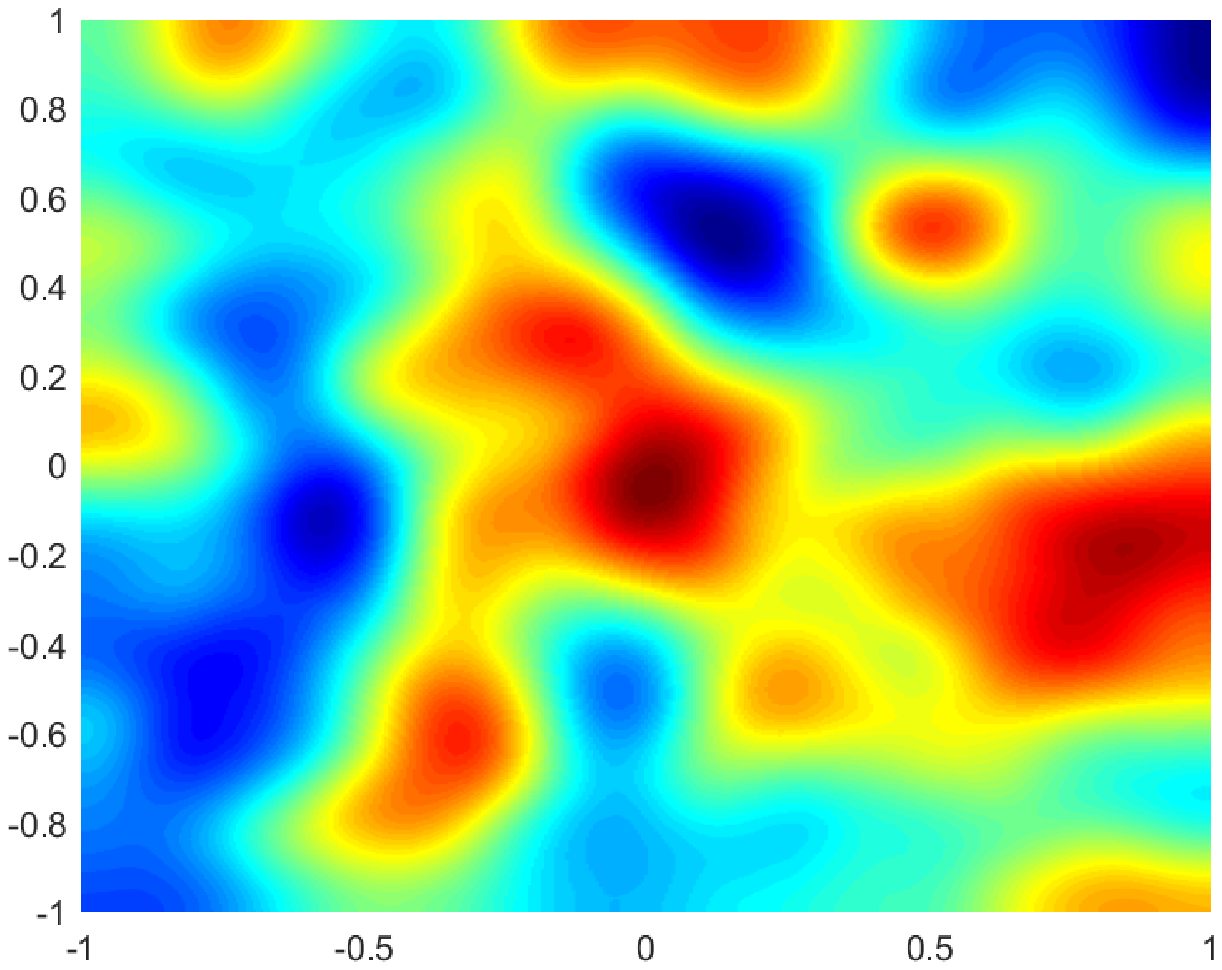}
\end{minipage}\\
\specialrule{0em}{-15mm}{.001pt}
\tabincell{c}{$t=20$\\ \\ \\ \\ \\ \\ \\}
&\begin{minipage}[t]{0.2\linewidth}
\includegraphics[scale=0.22]{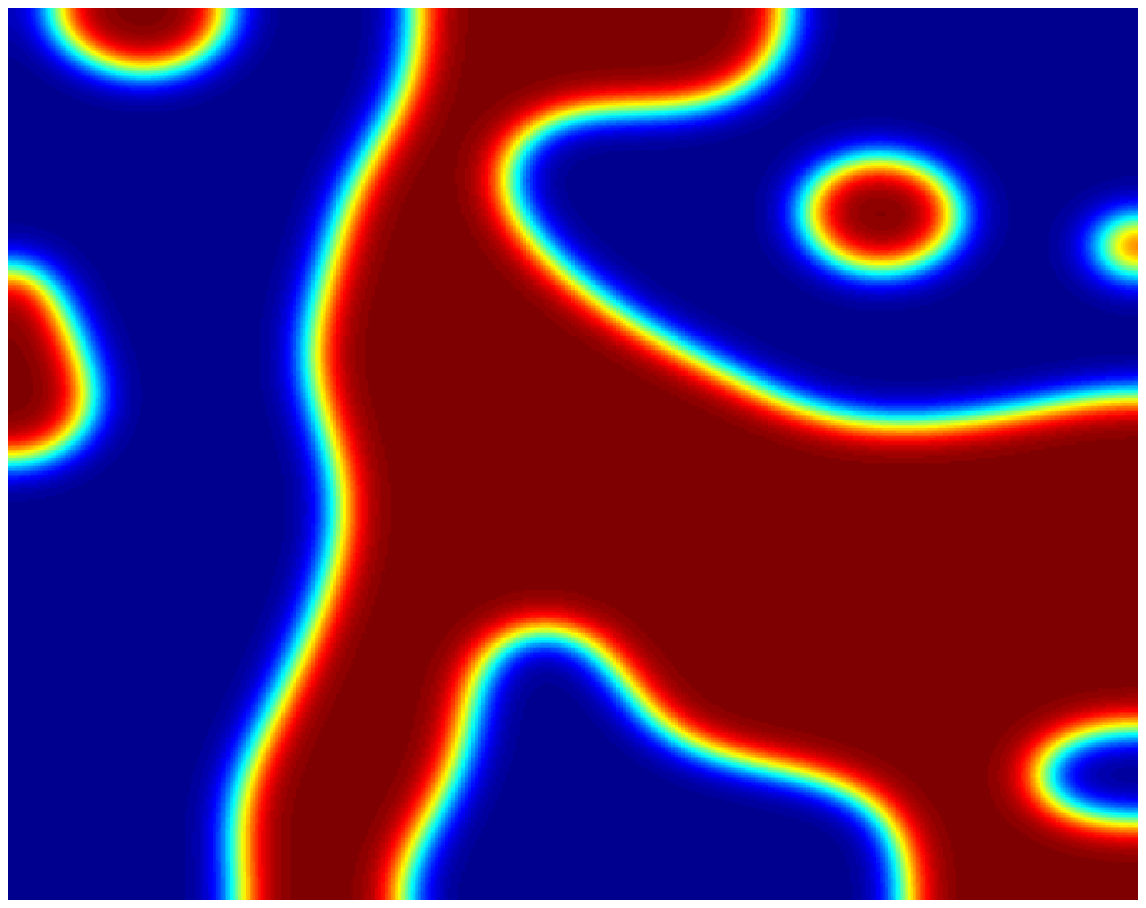}
\end{minipage}&
\begin{minipage}[t]{0.2\linewidth}
\includegraphics[scale=0.22]{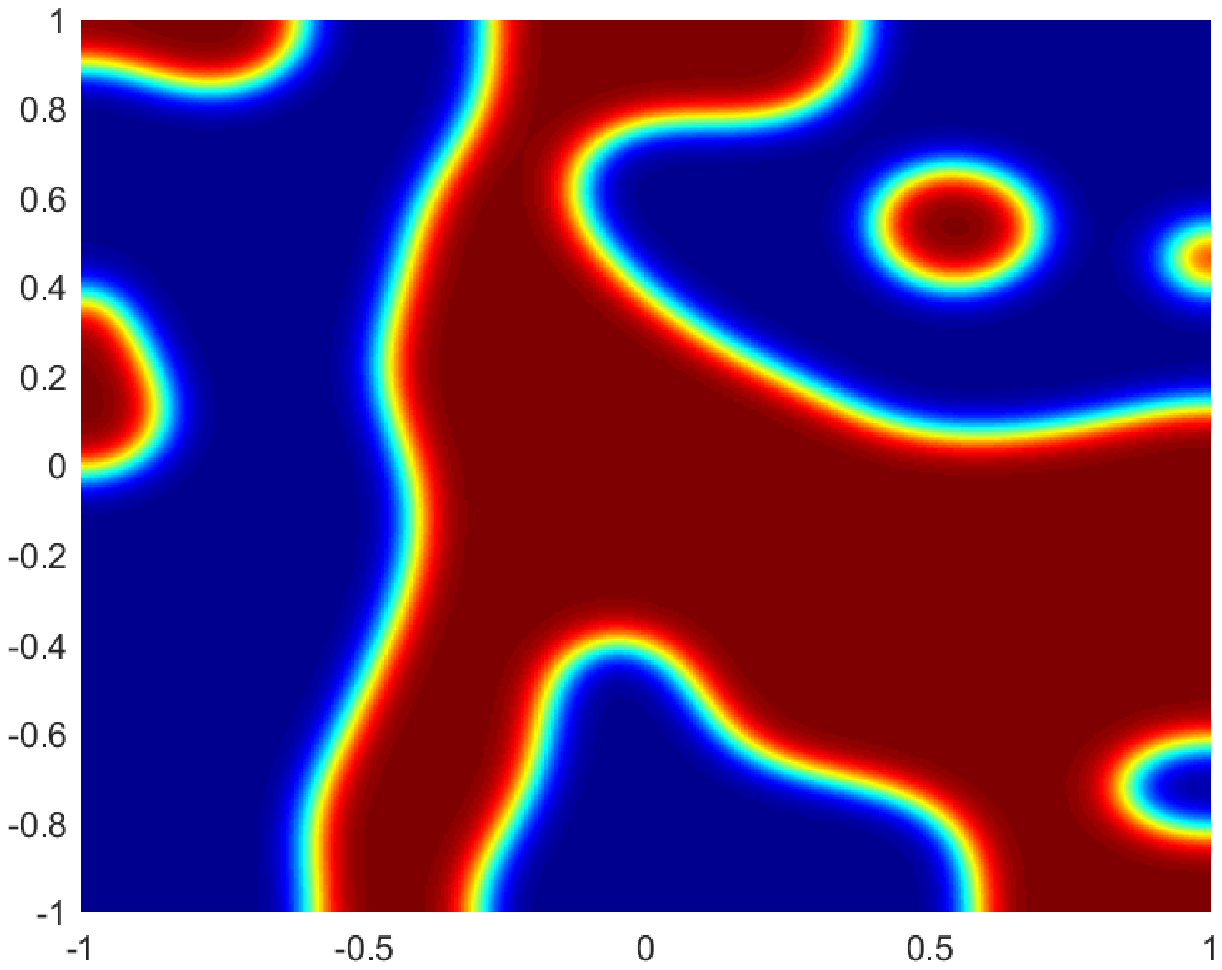}
\end{minipage}&
\begin{minipage}[t]{0.2\linewidth}
\includegraphics[scale=0.22]{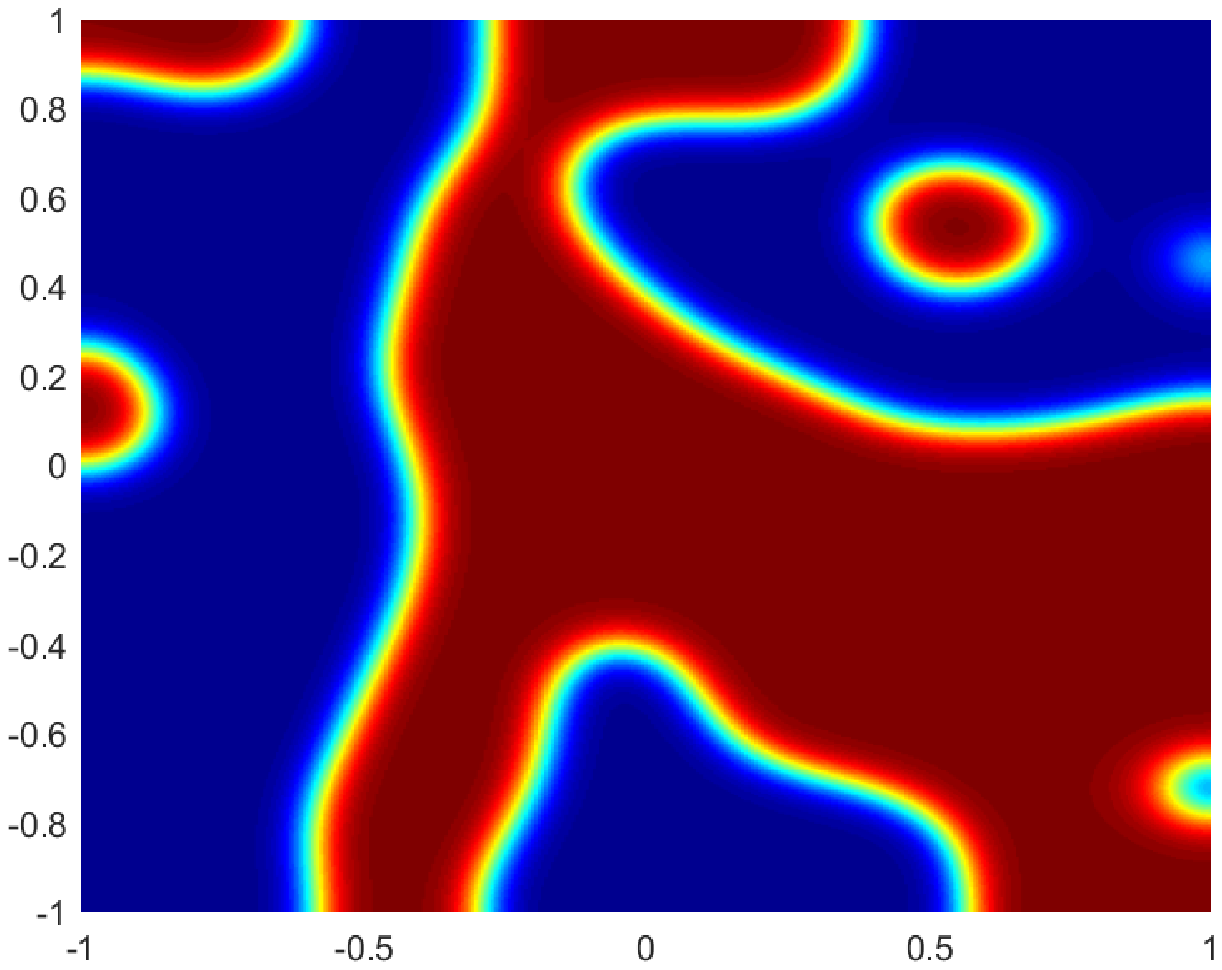}
\end{minipage}
&\begin{minipage}[t]{0.2\linewidth}
\includegraphics[scale=0.22]{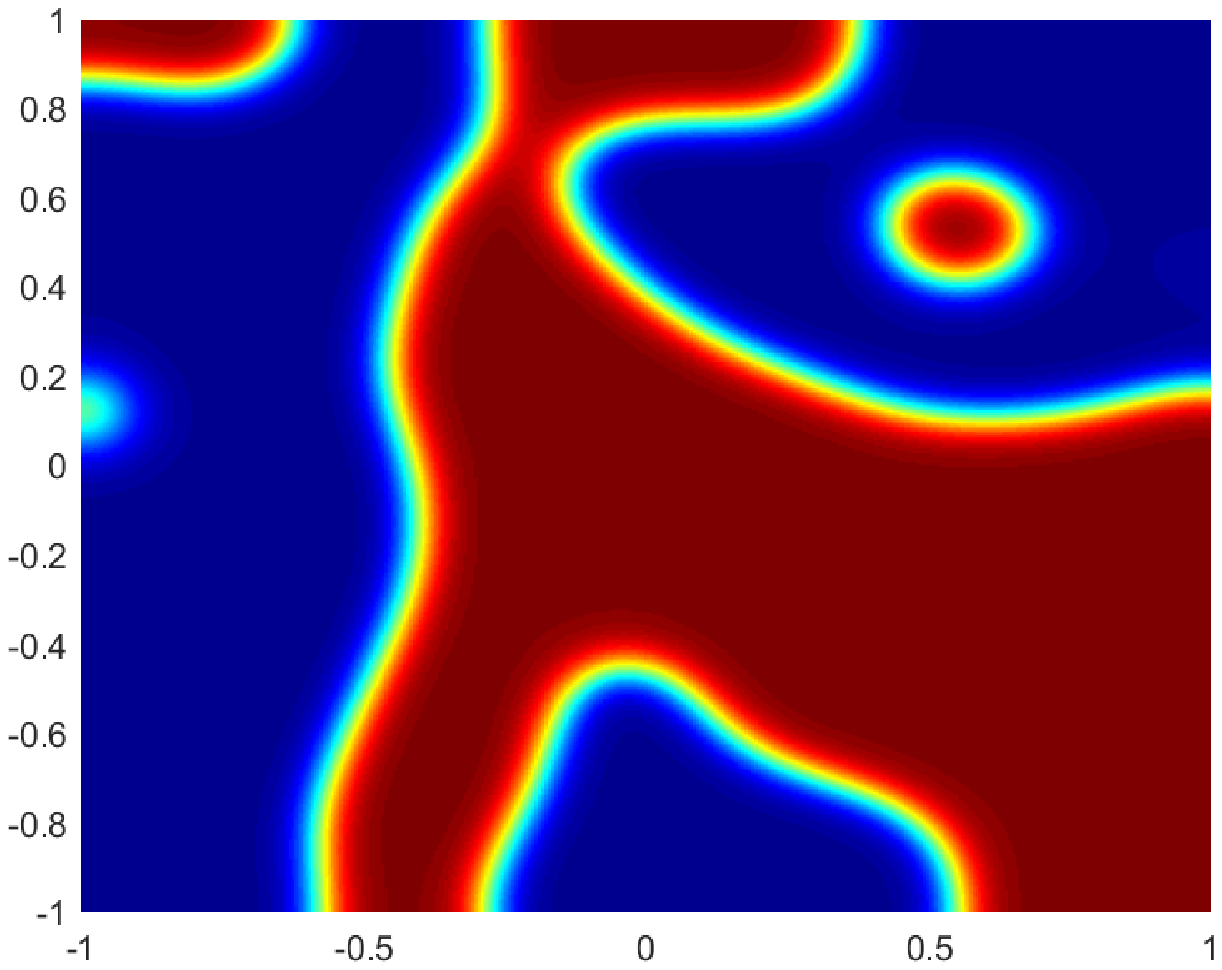}
\end{minipage}\\
\specialrule{0em}{-15mm}{.001pt}
\tabincell{c}{$t=50$\\ \\ \\ \\ \\ \\ \\}
&\begin{minipage}[t]{0.2\linewidth}
\includegraphics[scale=0.22]{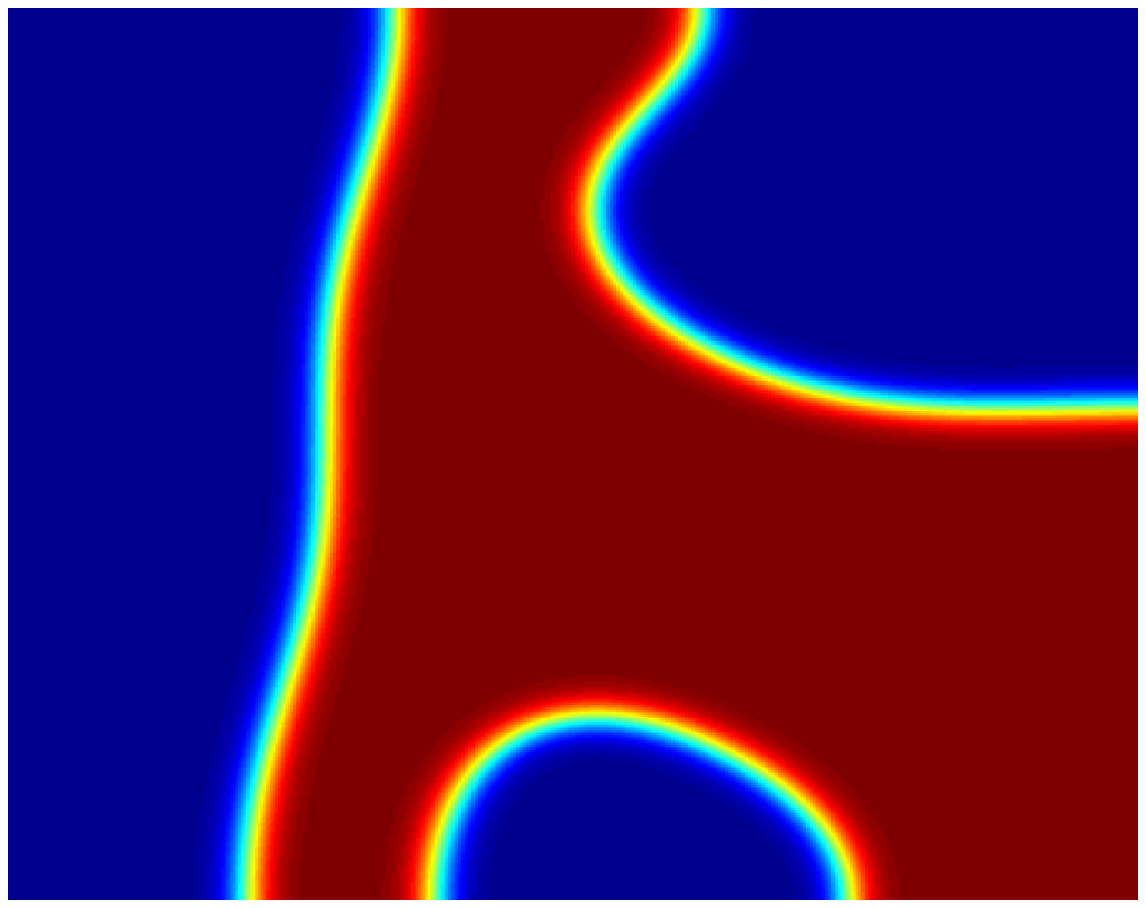}
\end{minipage}&
\begin{minipage}[t]{0.2\linewidth}
\includegraphics[scale=0.22]{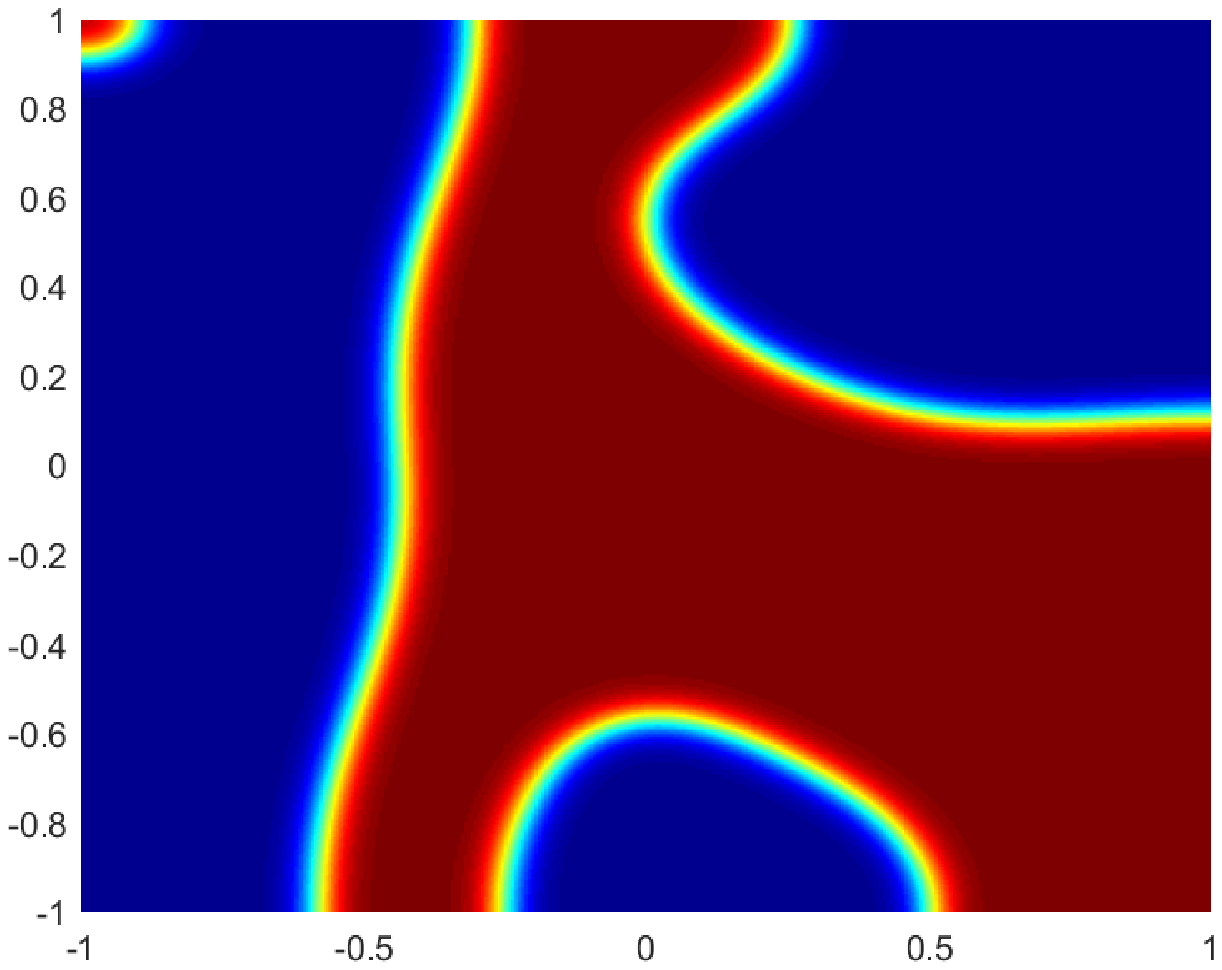}
\end{minipage}&
\begin{minipage}[t]{0.2\linewidth}
\includegraphics[scale=0.22]{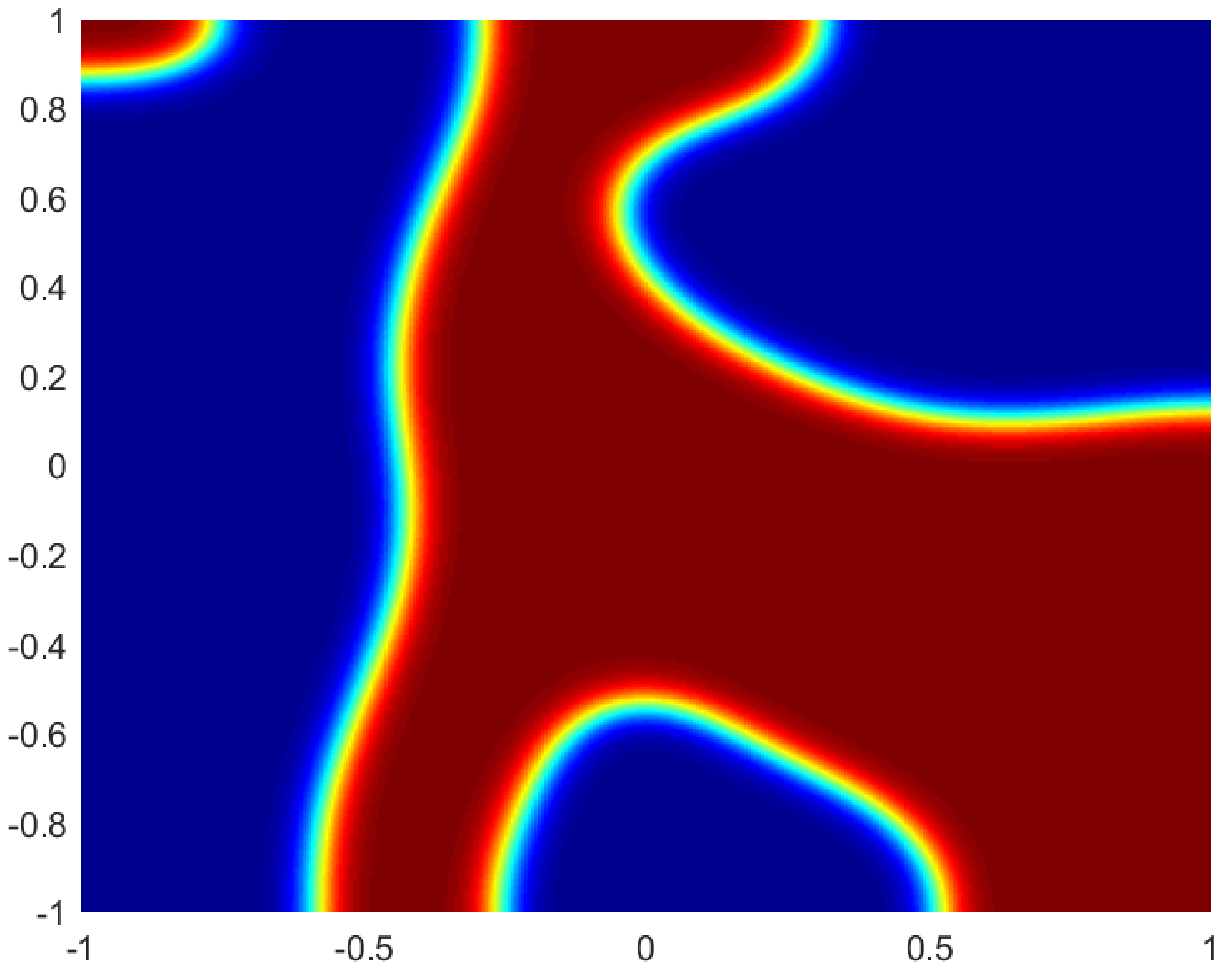}
\end{minipage}
&\begin{minipage}[t]{0.2\linewidth}
\includegraphics[scale=0.22]{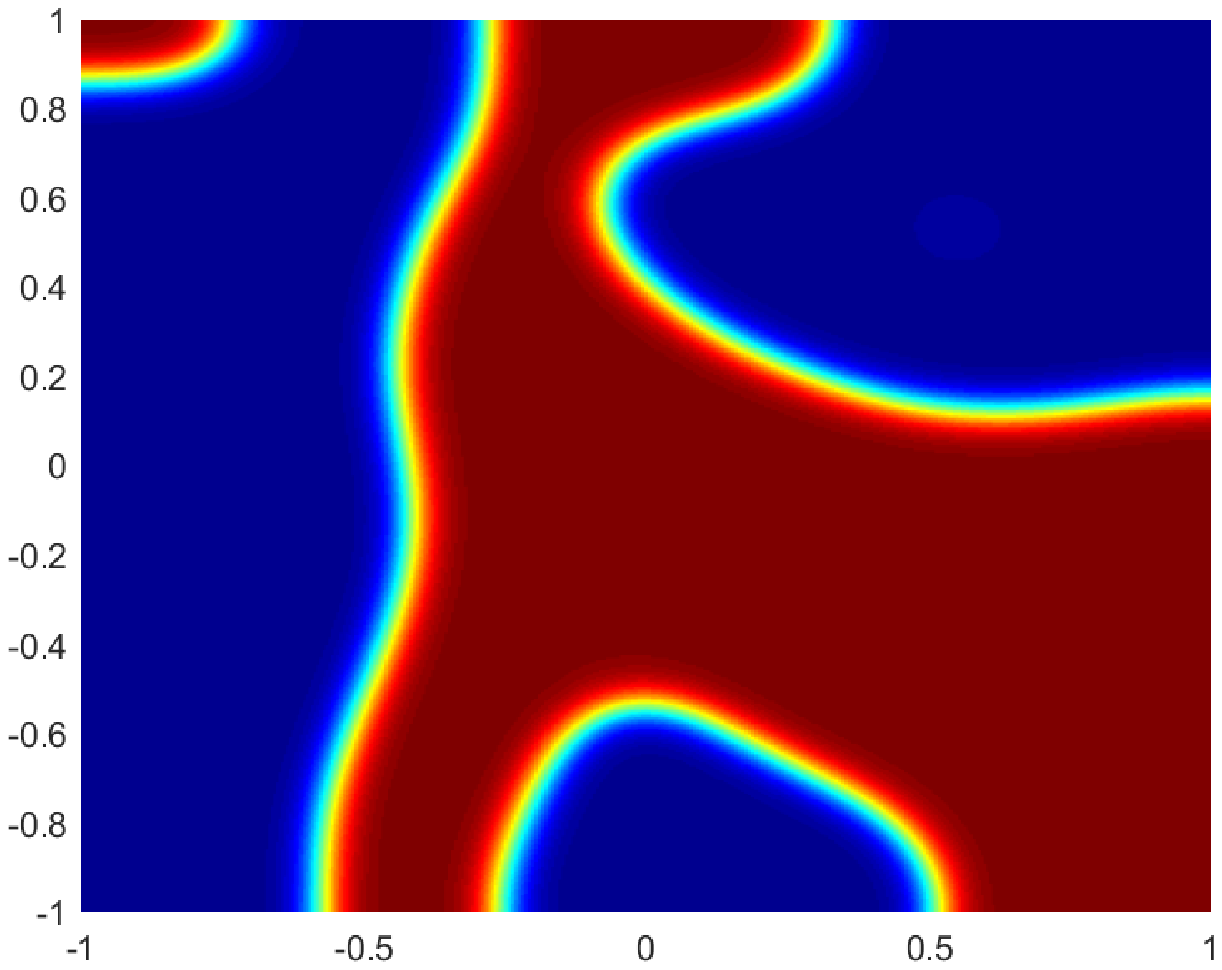}
\end{minipage}\\
\specialrule{0em}{-15mm}{.001pt}
\tabincell{c}{$t=100$\\ \\ \\ \\ \\ \\ \\}
&\begin{minipage}[t]{0.2\linewidth}
\includegraphics[scale=0.22]{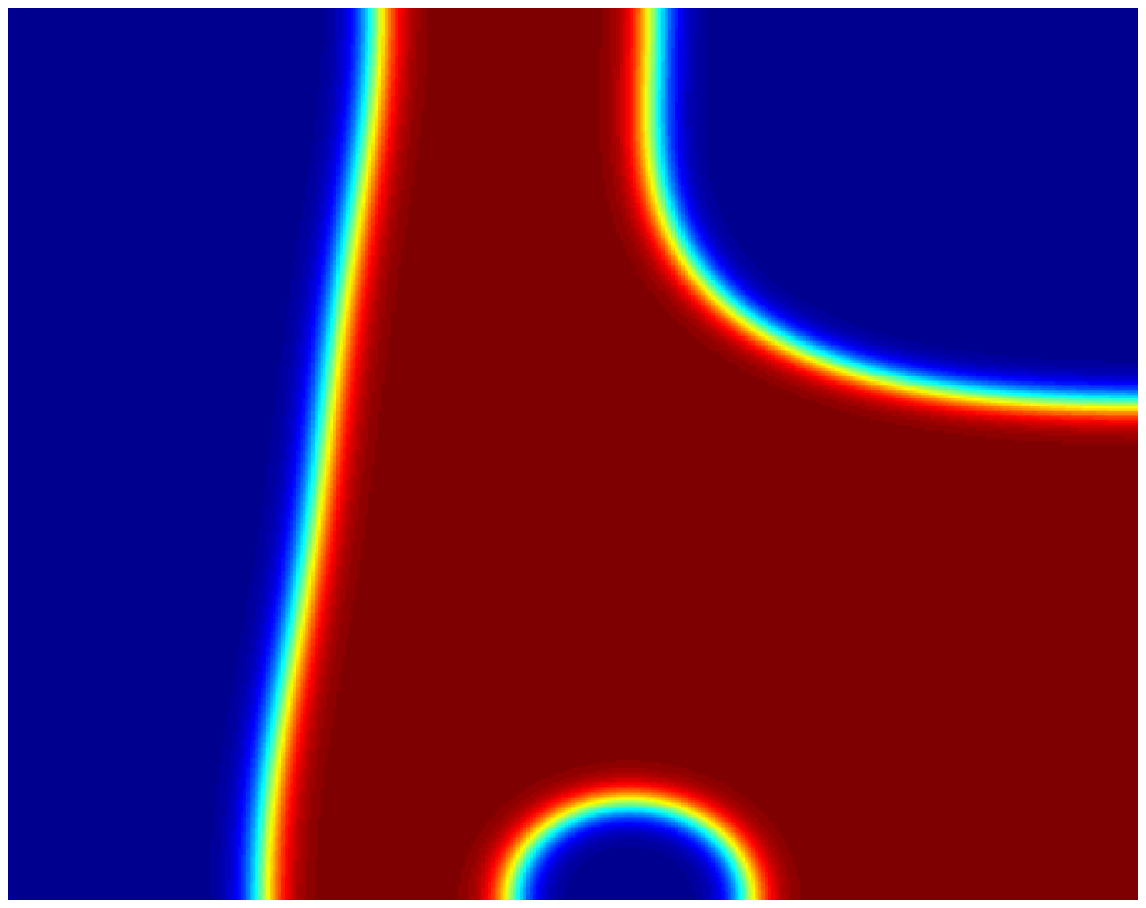}
\end{minipage}&
\begin{minipage}[t]{0.2\linewidth}
\includegraphics[scale=0.22]{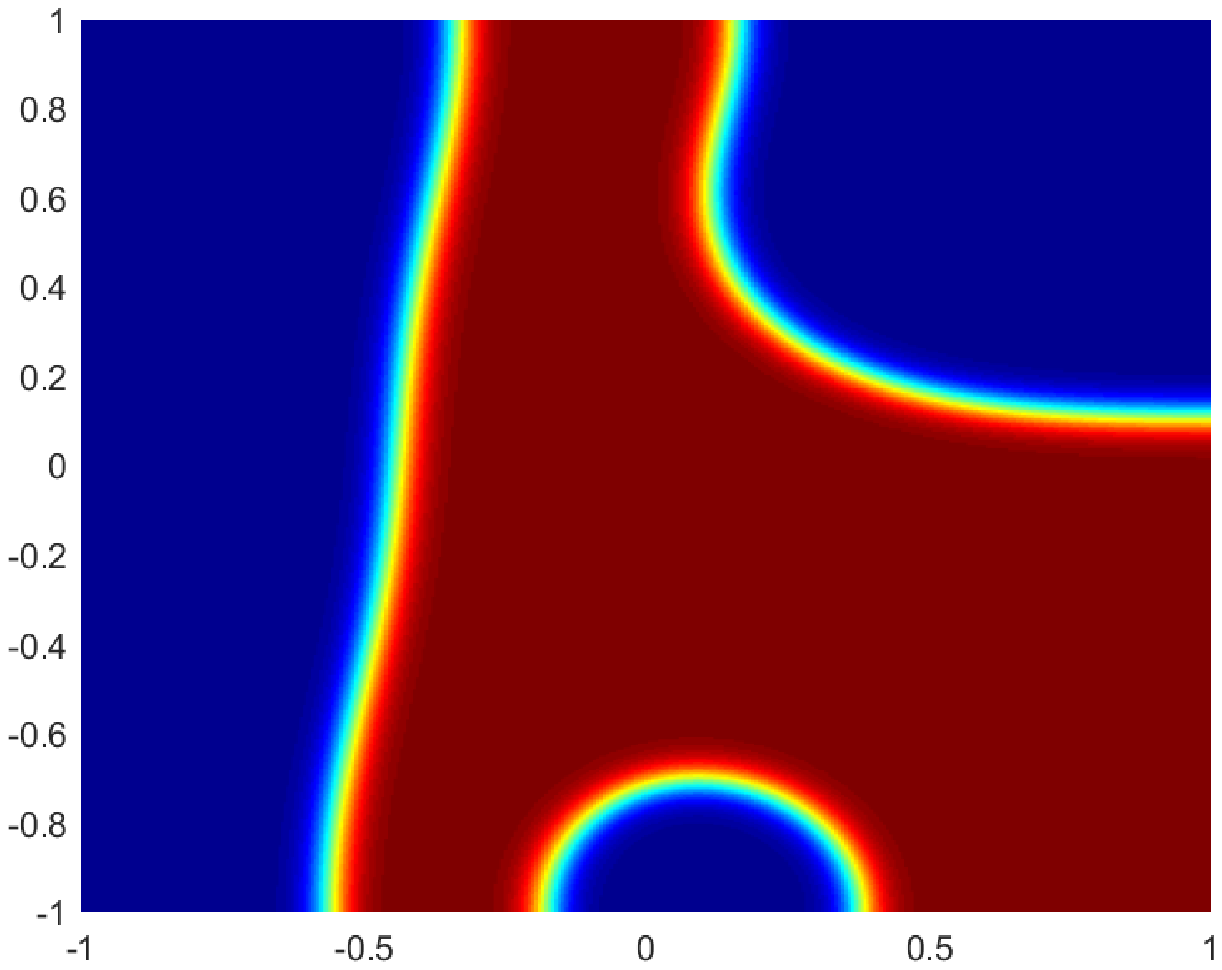}
\end{minipage}&
\begin{minipage}[t]{0.2\linewidth}
\includegraphics[scale=0.22]{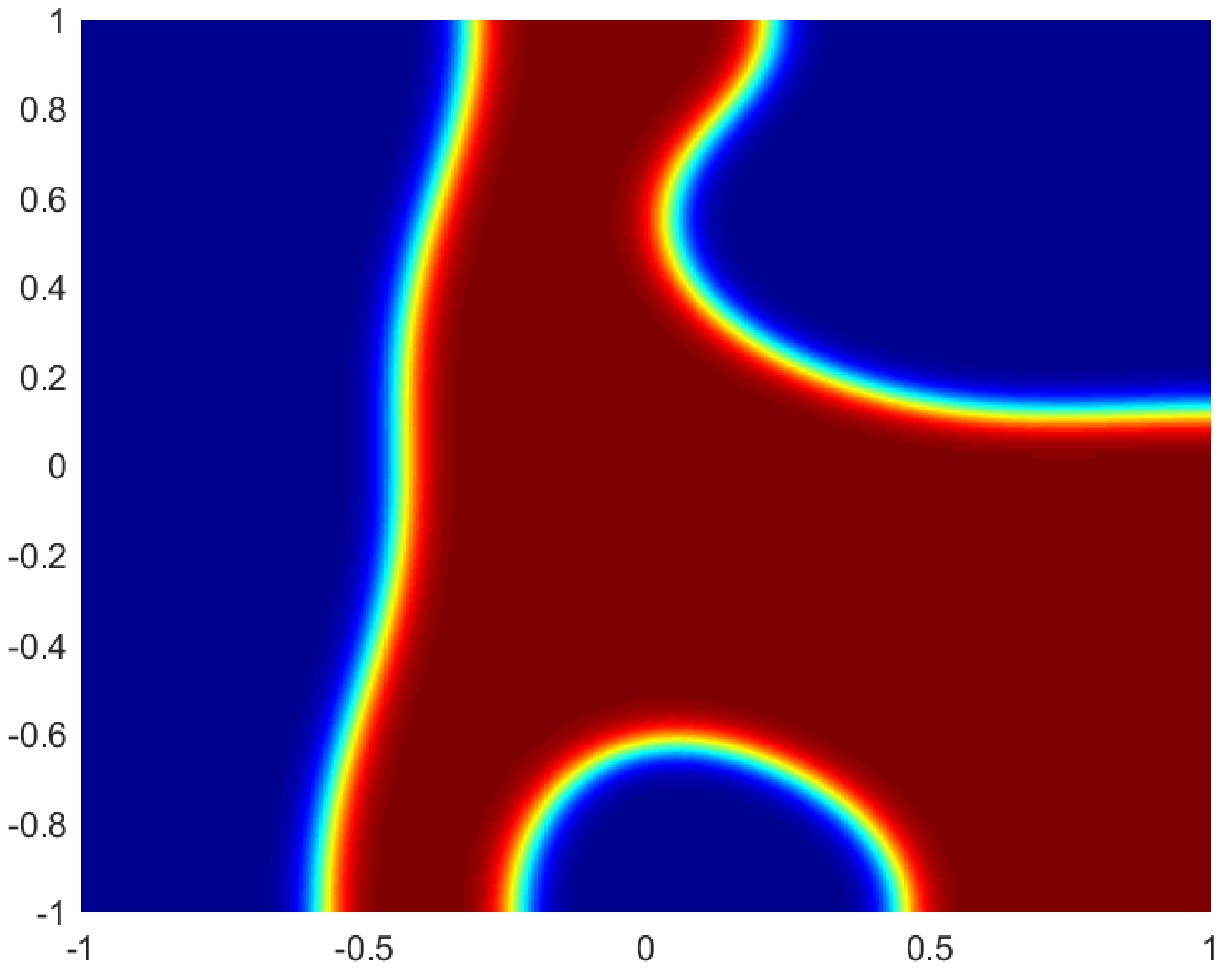}
\end{minipage}
&\begin{minipage}[t]{0.2\linewidth}
\includegraphics[scale=0.22]{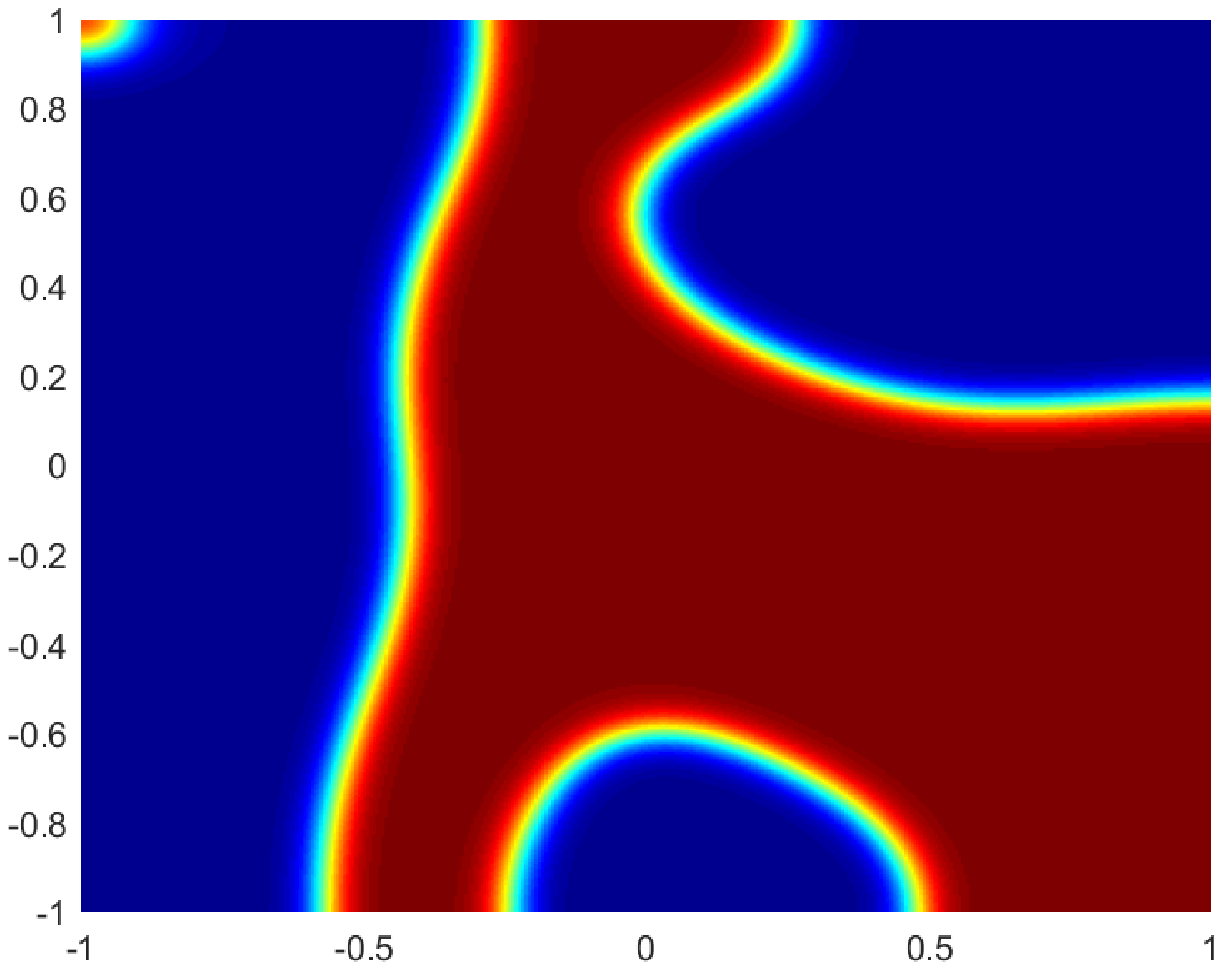}
\end{minipage}\\
\specialrule{0em}{-15mm}{.001pt}
\tabincell{c}{{\small Energy}\\ \\ \\ \\ \\ \\ \\}
&\begin{minipage}[t]{0.2\linewidth}
\includegraphics[scale=0.22]{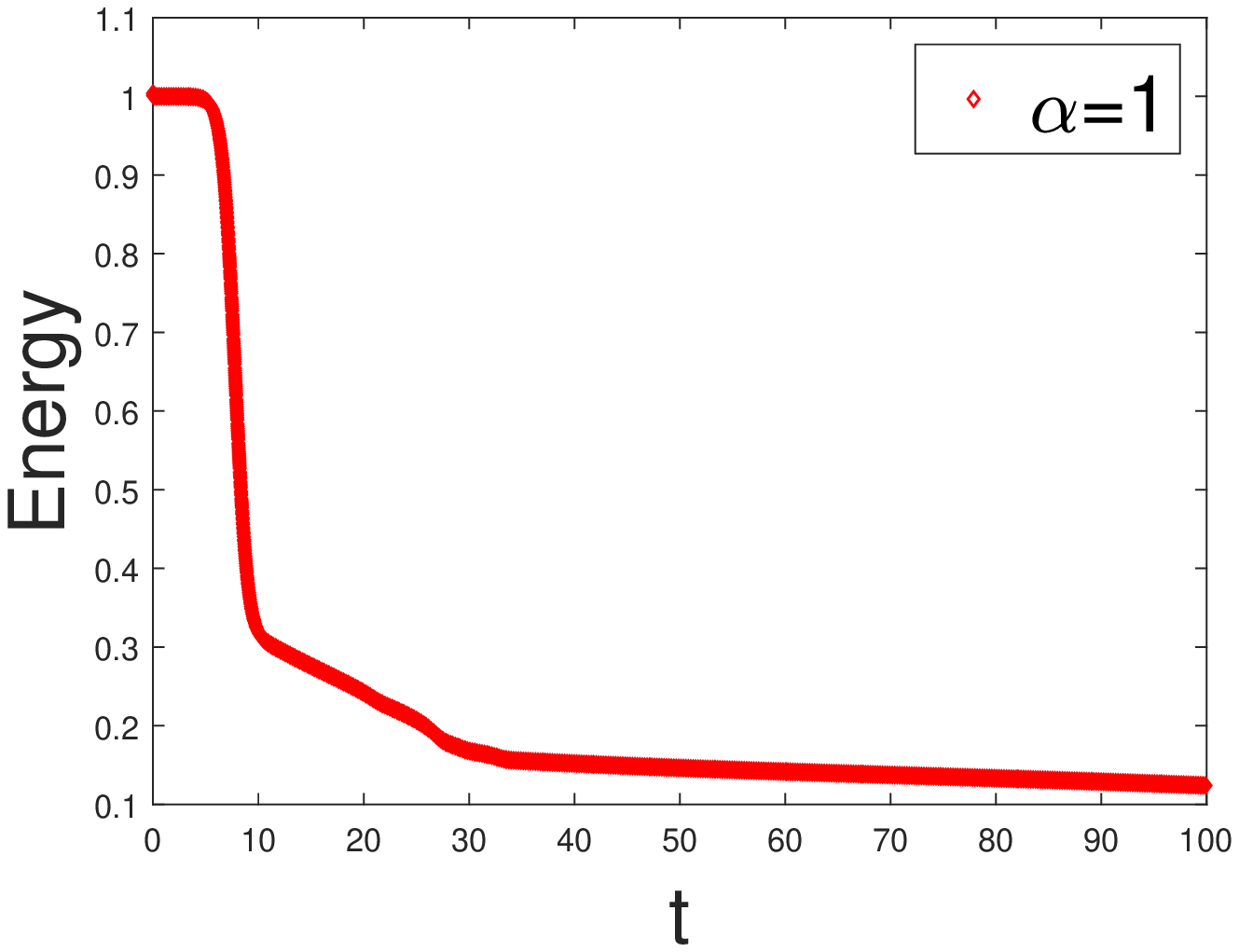}
\end{minipage}&
\begin{minipage}[t]{0.2\linewidth}
\includegraphics[scale=0.22]{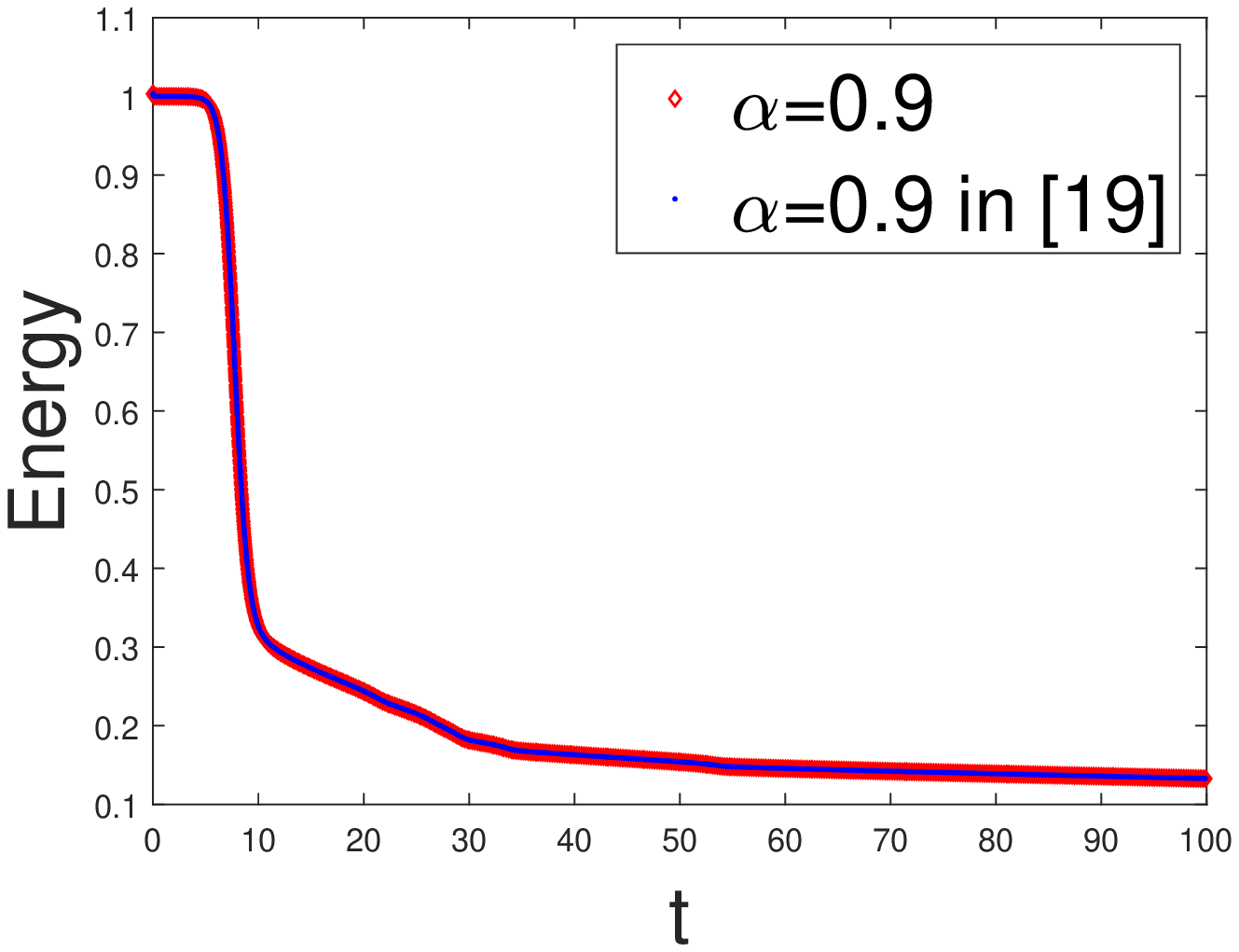}
\end{minipage}&
\begin{minipage}[t]{0.2\linewidth}
\includegraphics[scale=0.22]{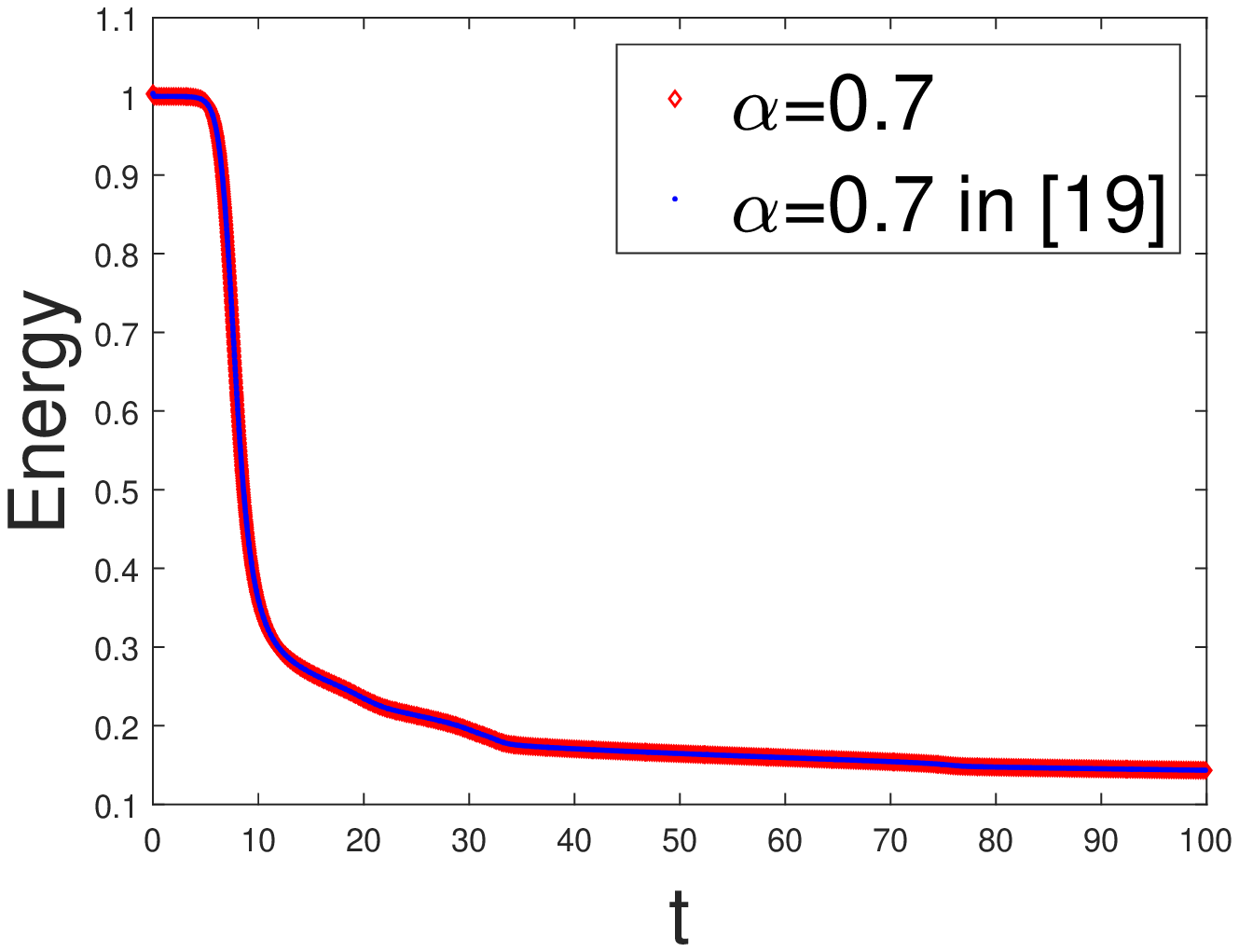}
\end{minipage}
&\begin{minipage}[t]{0.2\linewidth}
\includegraphics[scale=0.22]{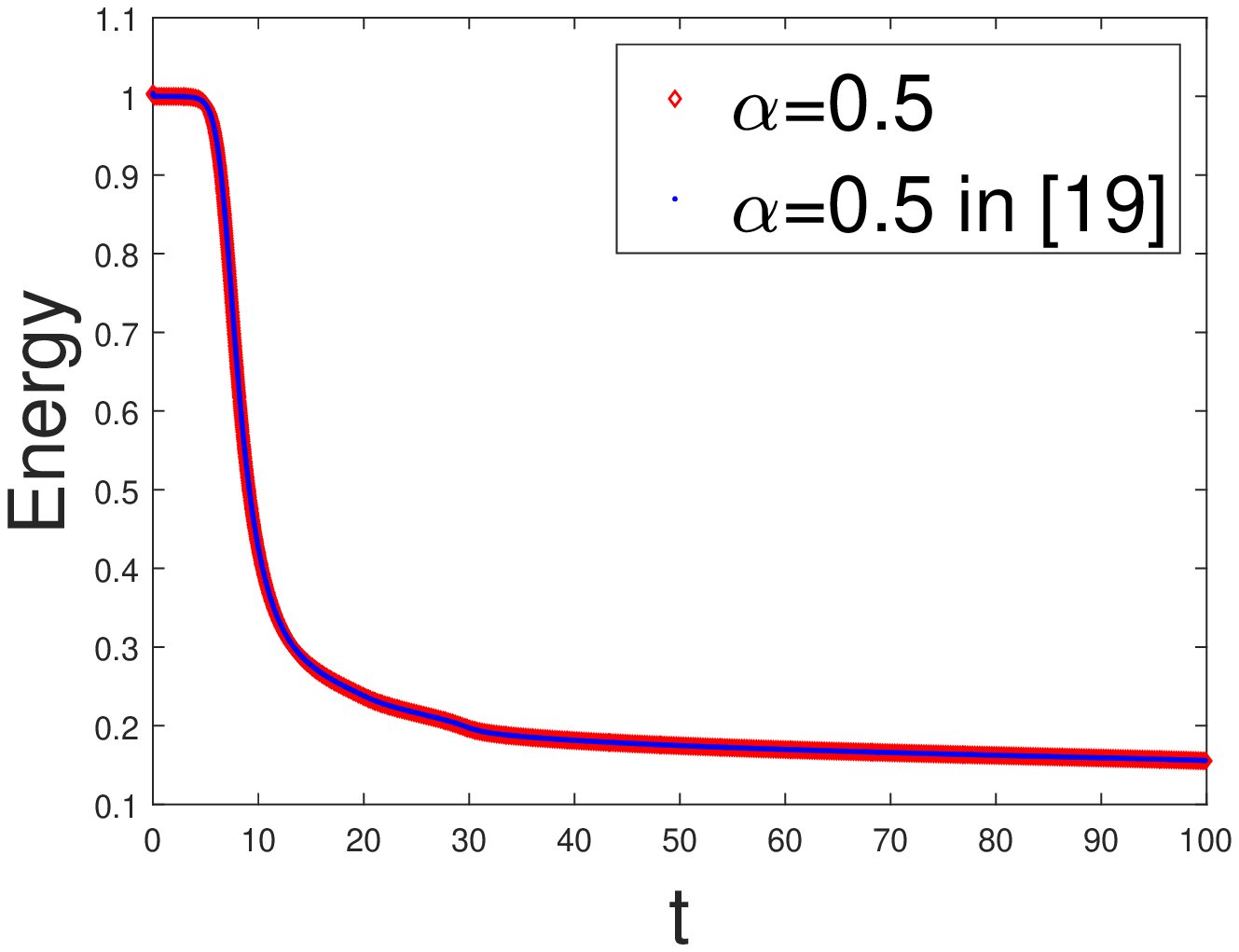}
\end{minipage}\\
\specialrule{0em}{-15mm}{.001pt}
\end{tabular}
  \end{center}
\caption{Snapshots of the simulated phase field evolution starting with a random initial data
for $\alpha=1, 0.9, 0.7, 0.5$ (the first five rows); Comparisons of the computed energy
between the L$1^{+}$-CN scheme and the scheme in \cite{HZX20}.
}\label{fig7}
\end{figure*}

\section {Concluding remarks}
We have proposed a class of energy dissipative schemes for the time fractional
Allen-Cahn equation. The construction of the schemes made use of a new idea to reformulate
the original equation. By splitting the time fractional derivative into
a local part and a history part, and adding the history part to the new defined
energy, a dissipation law for the new energy can be established in any given time grid.
Based on this splitting and an auxiliary variable approach, several schemes of different convergence
orders were constructed by combining the L1 and L$1^{+}$ discretizations to the time fractional derivative and Crank-Nicolson formula to other necessary terms.
The main property of the
proposed schemes is its unconditional stability for general meshes.
The proved stability of the schemes built on the graded mesh is of particular interesting
because this type of mesh has been found very useful to treat with the starting point singularity
of the time fractional differential equations. {\color{black}
Moreover the splitting-based approach allows use of the sum-of-exponentials techinique
to fast evaluate the history part of the fractional derivative
without affecting the stability property of the schemes.}
The efficiency of the proposed method was verified by a series of numerical experiments.
the authors
proved that the fractional derivative of the traditional free energy is always nonpositive.

{\color{black} It is notable that there exist some similar dissipation laws for memorized or mean energy, and
different energy laws lead to SAV-based schemes having quite different stability properties.
It seems to us that the energy dissipation law established in the current paper
facilitates construction of high order stable schemes.}
It is also worth to mention that
the idea of the present work is most likely extendable
to some other gradient flows,
such as the Cahn-Hilliard equation and molecular beam epitaxial growth models.

\bibliographystyle{plain}
\bibliography{ref}
\end{document}